\documentclass[11pt,reqno]{amsart} % AMS article style with right equation numbering
\pdfoutput=1

\usepackage[displaymath, mathlines]{lineno}

\usepackage{graphicx}

\usepackage{ifthen}

\usepackage{microtype}

\usepackage{xparse}

\usepackage{tocvsec2}

\usepackage[arrow,curve,matrix,tips,2cell]{xy}

\usepackage[top=1.25in,bottom=1.25in,left=1.25in,right=1.25in]{geometry} % EJP, ECP
\usepackage{xcolor}

\usepackage{enumitem}

\usepackage[labelfont=bf]{caption}

\usepackage[font={small},labelfont=md]{subfig}

\usepackage[noadjust,nosort]{cite}

\usepackage{amsmath, amsthm, amssymb, tikz, mathtools, array, mathrsfs, tensor}

\usepackage{polynom}

\usepackage{esint}

\usepackage{multicol}

\usepackage{dsfont}
	\newcommand{\one}{\mathds{1}}

\usepackage{framed}

\usepackage{stackengine}

\usepackage{upref}

\usepackage{hyperref}
	\hypersetup{colorlinks=true,linkcolor=blue,citecolor=blue,urlcolor=black} %hypertexnames=false

\numberwithin{equation}{section}


\newcommand{\red}{}

\newcommand{\eq}[1]{\begin{linenomath}\postdisplaypenalty=0\begin{align*} #1 \end{align*}\end{linenomath}}

\NewDocumentCommand{\eeq}{om}{\begin{linenomath}\postdisplaypenalty=0\begin{align} \IfNoValueTF{#1}{}{\tag{#1}} \begin{split} #2 \end{split} \end{align}\end{linenomath}}
\newcommand{\eeqc}[1]{\begin{linenomath}\postdisplaypenalty=0\begin{gather} #1 \end{gather}\end{linenomath}}

\newcommand{\eeqs}[1]{\begin{linenomath}\postdisplaypenalty=0\begin{align} #1 \end{align}\end{linenomath}}

\newcommand{\stackref}[2]{
\readlist*\mylist{#1}
\stackrel{\mbox{\footnotesize\foreachitem\x\in\mylist[]{\ifnum\xcnt=1\else,\fi\eqref{\x}}}}{#2}
}
\newcommand{\stackrefp}[2]{
\readlist*\mylist{#1}
\stackrel{\hphantom{\mbox{\footnotesize\foreachitem\x\in\mylist[]{\ifnum\xcnt=1\else,\fi\eqref{\x}}}}}{#2}
}
\newcommand{\stackrefpp}[3]{
\readlist*\mylist{#1}
\readlist*\mylistt{#2}
\stackrel{\parbox{\widthof{\footnotesize\foreachitem\x\in\mylistt[]{\ifnum\xcnt=1\else,\fi\eqref{\x}}}}{\centering\footnotesize\foreachitem\x\in\mylist[]{{\ifnum\xcnt=1\else,\fi\eqref{\x}}}}}{#3}
}

\def\eps{\varepsilon}
\def\vphi{\varphi}

\newcommand{\Ab}{\mathbf{A}}

\newcommand{\E}{\mathbb{E}}

\newcommand{\G}{\mathbb{G}}
\newcommand{\Hb}{\mathbb{H}}
\newcommand{\N}{\mathbb{N}}
\renewcommand{\P}{\mathbb{P}}

\newcommand{\R}{\mathbb{R}}

\newcommand{\T}{\mathbb{T}}

\newcommand{\CC}{\mathcal{C}}
\newcommand{\DD}{\mathcal{D}}

\newcommand{\FF}{\mathcal{F}}
\newcommand{\GG}{\mathcal{G}}

\newcommand{\II}{\mathcal{I}}
\newcommand{\JJ}{\mathcal{J}}

\newcommand{\LL}{\mathcal{L}}

\newcommand{\PP}{\mathcal{P}}

\newcommand{\RR}{\mathcal{R}}

\newcommand{\PPP}{\mathscr{P}}

\newcommand{\SSS}{\mathscr{S}}

\newcommand{\WWW}{\mathscr{W}}

\newcommand{\Asf}{\mathsf{A}}
\newcommand{\Bsf}{\mathsf{B}}

\newcommand{\vc}[1]{{\boldsymbol #1}}
\newcommand{\vct}[1]{{{#1}}}
\newcommand{\bq}{\boldsymbol{q}}

\newcommand{\wt}[1]{\widetilde{#1}}
\newcommand{\wh}[1]{\widehat{#1}}

\DeclareFontFamily{U}{mathx}{\hyphenchar\font45}
\DeclareFontShape{U}{mathx}{m}{n}{
      <5> <6> <7> <8> <9> <10>
      <10.95> <12> <14.4> <17.28> <20.74> <24.88>
      mathx10
      }{}
\DeclareSymbolFont{mathx}{U}{mathx}{m}{n}
\DeclareFontSubstitution{U}{mathx}{m}{n}
\DeclareMathAccent{\widecheck}{0}{mathx}{"71}
\DeclareMathAccent{\wideparen}{0}{mathx}{"75}

\DeclareMathOperator{\e}{e} % exponential
\newcommand{\cc}{\mathrm{c}} % for set complements and "critical" subscripts
\newcommand{\dd}{\mathrm{d}} % for differentials
            \makeatletter
            \DeclareFontFamily{OMX}{MnSymbolE}{}
            \DeclareSymbolFont{MnLargeSymbols}{OMX}{MnSymbolE}{m}{n}
            \SetSymbolFont{MnLargeSymbols}{bold}{OMX}{MnSymbolE}{b}{n}
            \DeclareFontShape{OMX}{MnSymbolE}{m}{n}{
                <-6>  MnSymbolE5
               <6-7>  MnSymbolE6
               <7-8>  MnSymbolE7
               <8-9>  MnSymbolE8
               <9-10> MnSymbolE9
              <10-12> MnSymbolE10
              <12->   MnSymbolE12
            }{}
            \DeclareFontShape{OMX}{MnSymbolE}{b}{n}{
                <-6>  MnSymbolE-Bold5
               <6-7>  MnSymbolE-Bold6
               <7-8>  MnSymbolE-Bold7
               <8-9>  MnSymbolE-Bold8
               <9-10> MnSymbolE-Bold9
              <10-12> MnSymbolE-Bold10
              <12->   MnSymbolE-Bold12
            }{}
            
            \let\llangle\@undefined
            \let\rrangle\@undefined
            \DeclareMathDelimiter{\llangle}{\mathopen}%
                                 {MnLargeSymbols}{'164}{MnLargeSymbols}{'164}
            \DeclareMathDelimiter{\rrangle}{\mathclose}%
                                 {MnLargeSymbols}{'171}{MnLargeSymbols}{'171}
            \makeatother
            
    \DeclareFontFamily{U}{matha}{\hyphenchar\font45}
    \DeclareFontShape{U}{matha}{m}{n}{ <-6> matha5 <6-7> matha6 <7-8>
    matha7 <8-9> matha8 <9-10> matha9 <10-12> matha10 <12-> matha12 }{}
    \DeclareSymbolFont{matha}{U}{matha}{m}{n}
    \DeclareFontFamily{U}{mathx}{\hyphenchar\font45}
    \DeclareFontShape{U}{mathx}{m}{n}{ <-6> mathx5 <6-7> mathx6 <7-8>
    mathx7 <8-9> mathx8 <9-10> mathx9 <10-12> mathx10 <12-> mathx12 }{}
    \DeclareSymbolFont{mathx}{U}{mathx}{m}{n}
    
    \DeclareMathDelimiter{\llbrack} {4}{matha}{"76}{mathx}{"30}
    \DeclareMathDelimiter{\rrbrack} {5}{matha}{"77}{mathx}{"38}
            
\newcommand{\pert}{\mathrm{pert}}
\newcommand{\cav}{\mathrm{cav}}
\newcommand{\alpham}{\alpha_\mathrm{min}}
\newcommand{\alphaM}{\alpha_\mathrm{max}}
\newcommand{\lambdam}{\lambda_\mathrm{min}}
\newcommand{\lambdaM}{\lambda_\mathrm{max}}
\DeclareMathOperator*{\Motimes}{\text{\raisebox{0.25ex}{\scalebox{0.8}{$\bigotimes$}}}}

\newenvironment{proofclaim}
{\begin{proof}}
{\renewcommand{\qedsymbol}{$\square$ (Claim)}
\end{proof}
\renewcommand{\qedsymbol}{$\square$}
}

\newtheorem{thm}{Theorem}[section]
\newtheorem{prop}[thm]{Proposition}
\newtheorem{cor}[thm]{Corollary}
\newtheorem{lemma}[thm]{Lemma}
\newtheorem{claim}[thm]{Claim}
\newtheorem{theirthm}{Theorem} % special environment for referencing theorems of others
\newtheorem{defn}[thm]{Definition}

\newtheorem{remark}[thm]{Remark}
\newtheorem{assumption}[thm]{Assumption}

\renewcommand{\thefootnote}{\fnsymbol{footnote}}

\title[Multi-species mixed $p$-spin spherical models]{Free energy in multi-species mixed $p$-spin spherical models}



\author{Erik Bates}
\thanks{E.B. was partially supported by NSF grant DMS-1902734} 
\address{\newline Department of Mathematics \newline University of Wisconsin--Madison \newline Van Vleck Hall \newline 480 Lincoln Drive \newline Madison, Wisconsin 53706-1324 %\newline United States 
\newline \textup{\tt ewbates@wisc.edu}}

\author{Youngtak Sohn}
\thanks{Y.S. was partially supported by NSF grant DMS-1954337}
\address{\newline Department of Mathematics \newline Massachusetts Institute of Technology \newline 77 Massachusetts Avenue \newline Cambridge, Massachusetts 02139-4307 %\newline United States
\newline \textup{\tt youngtak@mit.edu}}

\begin{document}
\bibliographystyle{acm}

% changes footnote labeling back to numbers
\renewcommand{\thefootnote}{\arabic{footnote}} \setcounter{footnote}{0}

%: ABSTRACT
\begin{abstract}
We prove a Parisi formula for the limiting free energy of multi-species spherical spin glasses with mixed $p$-spin interactions.
The upper bound involves a Guerra-style interpolation and requires a convexity assumption on the model's covariance function.
Meanwhile, the lower bound adapts the cavity method of Chen so that it can be combined with the synchronization technique of Panchenko; this part requires no convexity assumption.
In order to guarantee that the resulting Parisi formula has a minimizer, we formalize the pairing of synchronization maps with overlap measures so that the constraint set is a compact metric space.
This space is not related to the model's spherical structure and can be carried over to other multi-species settings. 
\end{abstract}

\maketitle
\vspace{-\baselineskip}
\tableofcontents

% uncomment to display line numbers
%\linenumbers

\section{Introduction}
Spin glasses are models of disordered magnetism, in which interacting magnetic spins have irregular alignments.
Mean-field spin glasses, most famously the \mbox{Sherrington}--\mbox{Kirkpatrick} (SK) model \cite{sherrington-kirkpatrick75,kosterlitz-thouless-jones76}, have served as rich prototypes for more physical models such as that of \mbox{Edwards} and \mbox{Anderson} \cite{edwards-anderson75}.
A centerpiece of the mean-field paradigm is the ability to express the limiting free energy with variational formulas.
Following the inspiration of \mbox{Parisi} \cite{parisi79,parisi80a,parisi80b,parisi83}, mathematicians have managed to make these formulas rigorous and subsequently reveal remarkable structure arising in the associated Gibbs measures.
The landmark work of \mbox{Talagrand} \cite{talagrand06a} in the case of the SK model was followed by similar results for general mixed $p$-spins \cite{panchenko14,auffinger-chen17} and spherical models \cite{talagrand06b,chen13,chen-sen17,jagannath-tobasco17b}.

In order to relax the mean-field assumptions of classical models, certain asymmetric models have been promoted and studied recently.
These include so-called ``multi-species'' models in which the spin coordinates are partitioned into several groups, between which various strengths of interactions are allowed, e.g.~\cite{gallo-contucci08,talagrand09,barra-genovese-guerra11,fedele-contucci11,fedele-unguendoli12,barra-galluzzi-guerra-pizzoferrato-tantari14,barra-contucci-mingione-tantari15}.
By raising new challenges, this direction has repeatedly inspired upgrades to the theoretical toolbox used to prove, among other things, variational expressions for free energy.
This paper furthers this effort by addressing a multi-species version of classical mixed $p$-spin spherical models.

Our main result is a Parisi-type variational formula for the limiting free energy of these models (Theorem \ref{main_thm}).
Along the way, we formally define a metric space of ``synchronized'' overlap measures (Definition \ref{lambda_av_def}), objects which were used by Panchenko \cite{panchenko15} in proving the analogous formula for the multi-species SK model on the hypercube. 
With this formalization we are able to establish Lipschitz continuity for the Parisi functional (Theorem \ref{lipschitz_continuity}) and the existence of minimizers (Corollary \ref{main_cor}).
Furthermore, the framework we develop here enables a companion work \cite{bates-sohn22b} to elucidate the effect of interspecies interactions on the structure of minimizers.

\subsection{Definitions}
Fix a finite set $\SSS$, to index the various species.
Suppose that for each positive integer $N$, we have a partition
$\{1,2,\dots,N\}=\uplus_{s\in\SSS}\II^s$.
Denote the cardinality of $\II^s$ by $\Lambda^s(N)$, so that $N = \sum_{s\in\SSS}\Lambda^s(N)$.
When the value of $N$ is clear from context, we will usually write $N^s = \Lambda^s(N)$.

We consider spin configurations $\sigma = (\sigma_1,\dots,\sigma_N)\in\R^N$ such that
\eq{
\sum_{i\in\II^s}\sigma_i^2 = N^s \quad \text{for each $s\in\SSS$}.
}
%
%This naturally defines a partition $\{1,2,\dots,N\}=\uplus_{s\in\SSS}\II^s$, where $\II^s$ is an integer interval of cardinality $|\II^s|=N^s$. 
%We then define the following product of spheres:
In other words, $\sigma$ belongs to the following product of spheres:
\eeq{ \label{original_TN_def}
\T_N \coloneqq %\sqrt{\Lambda^11(N)}\S^{\Lambda^11(N)-1}\times\cdots\times\sqrt{\Lambda^1K(N)}\S^{\Lambda^1K(N)-1},
\Motimes_{s\in\SSS} S_{N^s}, \quad \text{where} \quad
 S_n \coloneqq \{\sigma\in\R^n:\, \|\sigma\|_2^2 = n\}.
}
We say that coordinate $i$ belongs to species $s$ whenever $i\in\II^s$.
Conversely, we will write $s(i) = s(i,N)$ to express whichever species a given coordinate $i$ belongs to.
We assume that the fraction of coordinates allocated to each species, which we denote by $\lambda^s(N)\coloneqq N^s/N$, converges as $N\to\infty$: 
\eeq[H1]{\label{lambda_assumption}
\lim_{N\to\infty}\lambda^s(N)=\lambda^s\in(0,1] \quad \text{for each $s\in\SSS$}.
}
%where necessarily $\sum_{s\in\SSS}\lambda^s = 1$.

For each integer $p\geq1$, let $\vc\Delta^2_{p} = (\Delta^2_{s_1,\dots,s_p})_{ s_1,\dots,s_p\in\SSS}$ be a symmetric $p$-dimensional tensor of size $|\SSS|^p$, which will govern the $p$-spin interaction strengths between species.
The \textit{$p$-spin Hamiltonian} on $\T_N$ is defined as
\eeq{ \label{p_Hamiltonian}
H_{N}^{(p)}(\sigma) \coloneqq \frac{1}{N^{(p-1)/2}}\sum_{i_1,\dots,i_p=1}^N\sqrt{\Delta^2_{s(i_1),\dots,s(i_p)}}g_{i_1,\dots,i_p}\sigma_{i_1}\cdots\sigma_{i_p},
}
where each $g_{i_1,\dots,i_p}$ is an independent standard Gaussian random variable.
To simplify notation, we will use the following shorthands:
\begin{itemize}
\item The set of integers $\{1,2,\dots,N\}$ will be denoted by $[N]$.
\item For a $p$-tuple of coordinates $\vct i = (i_1,\dots,i_p)\in[N]^p$, we have the corresponding $p$-tuple of species:
\eq{
s(\vct i) = s(\vct i,N) \coloneqq (s(i_1),\dots,s(i_p))\in\SSS^p.
}
In addition, if $\sigma\in\R^N$, then we have the $p$-spin product 
\eq{
\sigma_{\vct i} \coloneqq \sigma_{i_1}\cdots\sigma_{i_p}\in\R.
}
\item For a $p$-tuple of species $\vct s = (s_1,\dots,s_p) \in \SSS^p$ and $\vc q = (q^s)_{s\in\SSS}\in\R^\SSS$, we will write 
\eq{
q^{\vct s} \coloneqq q^{s_1}\cdots q^{s_{p}}\in\R.
}
For instance, given the parameters $\vc\lambda = (\lambda^s)_{s\in\SSS}$ from \eqref{lambda_assumption} governing the proportion of coordinates belonging to each species, we can write $\lambda^{\vct s} = \lambda^{s_1}\cdots\lambda^{s_p}$.
\end{itemize}

\begin{remark}
We have elected to not burden the reader with symbolic cues such as $\vec i$ or $\vc i$ to distinguish vector quantities and scalar quantities, since the nature of such objects should always be clear from context.
The single exception is a vector indexed by $\SSS$, such as $\vc q = (q^s)_{s\in\SSS}\in\R^\SSS$.
For these quantities, the boldface indicates that the analogous object in the classical single-species model would be a scalar.
This distinction will be especially important when we discuss replica overlaps.
Also note that the species identifier usually appears as a superscript and should not be mistaken for an exponent.
\end{remark}

With these notational conventions, we can rewrite \eqref{p_Hamiltonian} as
\eeq{ \label{HNp_def}
H_{N}^{(p)}(\sigma) = \frac{1}{N^{(p-1)/2}}\sum_{\vct i\in[N]^p} \sqrt{\Delta^2_{s(\vct i)}}g_{\vct i}\sigma_{\vct i}.
}
The \textit{mixed Hamiltonian} is then given by
\eeq{ \label{mixed_H_def}
H_N(\sigma) \coloneqq \sum_{p\geq1}\beta_p H_{N}^{(p)}(\sigma),
}
where $\beta = (\beta_p)_{p\geq1}$ satisfies a decay condition of the form
\eeq[H2]{\label{decay_condition}
%\textcolor{gray}{\displaystyle \sum_{p\geq1} \beta_p^2\|\vc\Delta^2_{p}\|_\infty(1+\eps)^p < \infty \quad \text{for some $\eps>0$}.} \\
\sum_{p\ge1}\beta_p^2(1+\eps)^p\sum_{s\in\SSS^p}\Delta_s^2\lambda^s \quad \text{for some $\eps>0$}.
}
If $\beta_p = 0$ for all $p \neq 2$, then \eqref{mixed_H_def} would be called an SK model.

With $\mu_n$ denoting normalized surface measure on the sphere $S_n$, we equip the configuration space $\T_N$ from \eqref{original_TN_def} with the product measure
\eq{
\tau_N \coloneqq \Motimes_{s\in\SSS}\mu_{N^s}.
}
With $\tau_N$ serving as a reference measure, the Hamiltonian \eqref{mixed_H_def} naturally produces a Gibbs probability measure $G_N$ on $\T_N$, defined by
\eeq{ \label{gibbs_measure}
G_N(\dd\sigma) \coloneqq \frac{1}{Z_N}\exp(H_N(\sigma))\ \tau_N(\dd\sigma).
}
The random normalizing constant $Z_N$ is called the \textit{partition function},
\eq{
Z_N \coloneqq \int_{\T_N}\exp(H_N(\sigma))\ \tau_N(\dd\sigma),
}
and we are interested in the limiting value of its exponential growth rate, or \textit{free energy}:
\eq{
F_N \coloneqq \frac{1}{N} \log Z_N.
}

\subsection{Main results: the Parisi formula}
We will show that $\lim_{N\to\infty} F_N$ exists, is non-random, and is given by a variational formula called the Parisi formula.
%Here we state the formula, and in the next section we will offer some interpretation.
In order to define the objective function, called the Parisi functional, we first need to introduce some other relevant functions and also define the constraint set over which the optimization will take place.

\subsubsection{Relevant functions}
As a centered Gaussian process, $(H_N(\sigma))_{\sigma\in\T_N}$ is characterized by its covariance function.
If we define, for any $\sigma^1,\sigma^2\in\T_N$, the \textit{overlap vector} $\vc R(\sigma^1,\sigma^2) = (R^s(\sigma^1,\sigma^2))_{s\in\SSS}$ with coordinates
\eeq{ \label{overlap_def}
R^s(\sigma,\sigma') \coloneqq \frac{1}{{N^s}}\sum_{i \in \II^s} \sigma_i\sigma_i',
}
then we have the following covariance relation:
\eeqs{
\label{covariance_at_N}
\E[H_N(\sigma) H_N(\sigma')] &= N\xi_N(\vc R(\sigma,\sigma')), \quad \text{where} \\
\label{xi_N_def}
\xi_N(\vc q) &\coloneqq \sum_{p\geq1} \beta_p^2\sum_{\vct s\in\SSS^p} \Delta^2_{\vct s}\lambda^{\vct s}(N)q^{\vct s} \quad \text{for $\vc q\in[-1,1]^\SSS$}.
}
Since we assume $\lambda^s(N)\to\lambda^s$ as $N\to\infty$, the function $\xi_N$ converges to
\eq{ %\label{xi_def}
\xi(\vc q) \coloneqq \sum_{p\geq1} \beta_p^2\sum_{\vct s\in\SSS^p} \Delta^2_{\vct s}\lambda^{\vct s}q^{\vct s}, \quad \vc q\in[-1,1]^\SSS.
}
We assume $\xi$ is convex on $[0,1]^\SSS$. 
That is, its Hessian is nonnegative definite on this domain:
\eeq[H3]{\label{xi_convex}
\nabla^2 \xi(\vc q)\geq 0 \quad \text{for $\vc q\in [0,1]^\SSS$}.
}
Next define, for each $s\in\SSS$, the function
\eeq{ \label{gamma_def}
\xi^s(\vc q)
\coloneqq \frac{1}{\lambda^s}\frac{\partial \xi}{\partial q^s}(\vc q)
&= \sum_{p\geq1}p\beta_p^2\sum_{\vct t\in\SSS^{p-1}} \Delta^2_{(\vct t,s)}\lambda^{\vct t}q^{\vct t},
}
as well as
\eeq{ \label{theta_def}
\theta(\vc q) 
\coloneqq \vc q\cdot\nabla\xi(\vc q) - \xi(\vc q)
= \sum_{p\geq1} (p-1)\beta_p^2\sum_{\vct s\in\SSS^p} \Delta^2_{\vct s}\lambda^{\vct s}q^{\vct s}.
}
Note that on $[0,1]^\SSS$, both $\xi^s$ and $\theta$ are non-decreasing in every coordinate. 
%We are now ready to define the objective function for our variational formula. Its argument will be a pair $(\zeta,\Phi)$, where $\zeta$ is a probability measure on $[0,1]$ (here and always, a Borel measure), and $\Phi$ belongs to the following space of functions.

\subsubsection{The constraint set}
The argument to the Parisi functional will be a pair $(\zeta,\Phi)$, where $\zeta$ is a probability measure on $[0,1]$ (always a Borel measure), and $\Phi$ belongs to the following space of functions.

\begin{defn} \label{lambda_av_def}
Given $\vc\lambda=(\lambda^s)_{s\in\SSS}$, let us say that a map $\Phi = (\Phi^s)_{s\in\SSS} \colon [0,1]\to[0,1]^\SSS$ is $\vc\lambda$-\textit{admissible} if each coordinate $\Phi^s$ is non-decreasing and continuous, and jointly they satisfy
\eq{ %\label{admissible_def}
\sum_{s\in\SSS}\lambda^s\Phi^s(q) = q \quad \text{for all $q\in[0,1]$}.
}
When $\zeta$ is a Borel probability measure on $[0,1]$, we will call $(\zeta,\Phi)$ a \textit{$\vc\lambda$-admissible pair}.
\end{defn}

Notice that if $\Phi$ is $\vc\lambda$-admissible, then $\Phi^s$ is $(1/\lambda^s)$-Lipschitz continuous because 
\eq{ %\label{single_species_Lipschitz}
\lambda^s|\Phi^s(q)-\Phi^s(u)| \leq \sum_{t\in\SSS}\lambda^{t}|\Phi^{t}(q)-\Phi^{t}(u)| = |q-u|.
}
This in turn implies
\eeq{ \label{lambda_av_consequence}
\|\Phi(q)-\Phi(u)\|_1\leq|q-u|\sum_{s\in\SSS}\frac{1}{\lambda^s} \quad \text{for any $q,u\in[0,1]$}.
}
In particular, for any Lipschitz continuous function $f:[0,1]^\SSS\to\R$, the composition $f\circ \Phi$ is also Lipschitz and thus differentiable almost everywhere by Rademacher's theorem.
Therefore, given a $\vc\lambda$-admissible pair $(\zeta,\Phi)$, we can define for each $s\in\SSS$ the following function:
\eeq{ \label{ds_def}
d^s(q) \coloneqq \int_q^1 \zeta\big([0,u]\big)(\xi^s\circ \Phi)'(u)\ \dd u, \quad q\in[0,1].
}
For any vector $\vc b = (b^s)_{s\in\SSS}$ satisfying the constraint
\eeq{ \label{parisi_constraint}
b^s > d^s(0) \quad \text{for each $s\in\SSS$},
}
we define the quantity
\eeq{ \label{A_def}
A(\zeta,\Phi,\vc b) \coloneqq
&\sum_{s\in\SSS} \frac{\lambda^s}{2}\Big[b^s -1-\log b^s + \frac{\xi^s(0)}{b^s-d^s(0)} + \int_0^1\frac{(\xi^s\circ \Phi)'(q)}{b^s-d^s(q)}\ \dd q\Big] \\
&- \frac{1}{2}\int_0^1 \zeta\big([0,q]\big)(\theta\circ\Phi)'(q)\ \dd q.
}
The \textit{Parisi functional} is given by
\eeq{ \label{parisi_functional_def}
\PPP(\zeta,\Phi) \coloneqq \inf_{\vc b} A(\zeta,\Phi,\vc b),
}
where the infimum is over $\vc b\in(0,\infty)^\SSS$ satisfying \eqref{parisi_constraint}.
We then have the following expression for the limiting free energy.

\begin{thm}[Parisi formula] \label{main_thm}
Assuming \eqref{lambda_assumption}, \eqref{decay_condition}, and \eqref{xi_convex}, we have
\eeq{ \label{parisi_formula}
\lim_{N\to\infty} F_N = \inf_{\zeta,\Phi} \PPP(\zeta,\Phi) \quad \mathrm{a.s.},
}
where the infimum is over $\vc\lambda$-admissible pairs.
Without the convexity assumption \eqref{xi_convex}, it is still true that
\eeq{ \label{parisi_lower}
\liminf_{N\to\infty} F_N \geq \inf_{\zeta,\Phi} \PPP(\zeta,\Phi).
}
\end{thm}

It may seem strange in \eqref{parisi_functional_def} to define the objective function itself using a variational expression.
We do this because the parameter $\vc b$ should really be thought of as a consequence of calculus rather than spin glass theory; it appears because of a large deviations calculation originally carried out by Talagrand \cite{talagrand06b} (translating here to Proposition \ref{explain_appearance}).
An optimality condition for $\vc b$ is given in \cite[Thm.~2.12]{bates-sohn22b}.
The objects $\zeta$ and $\Phi$, on the other hand, are physically meaningful.
Very briefly, if $\sigma^1$ and $\sigma^2$ are independent samples from the Gibbs measure $G_N$ of \eqref{gibbs_measure}, then $\zeta$ represents the limiting law (as $N\to\infty$) of the overlap averaged across all species,
\eq{
R(\sigma^1,\sigma^2) \coloneqq \frac{1}{N}\sum_{i=1}^N\sigma_i^1\sigma_i^2=\sum_{s\in\SSS}\lambda^s(N)R^s(\sigma^1,\sigma^2).
}
Meanwhile, $\Phi$ specifies the relationship between average overlap and overlap within each species: $\Phi(R(\sigma^1,\sigma^2))=\vc R(\sigma^1,\sigma^2)$.
More context will be provided in Section \ref{sketch_section}, where we elaborate on the origins of these two order parameters.

\begin{remark}
One can also add an external magnetic field to each species, in which case one replaces $H_N(\sigma)$ with
\eq{
H_N(\sigma) + \sum_{s\in\SSS}h_s\sum_{i\in\II^s}\sigma_i,
}
where $h_s\in\R$ is a fixed number.
In that case, we would add to \eqref{A_def} the following quantity:
\eeq{ \label{add_because_of_field}
\sum_{s\in\SSS}\frac{\lambda^s}{2}\cdot\frac{h_s^2}{b^s-d^s(0)}.
}
The proofs in this case would simply require that we carry the external field through every step.
The appearance of \eqref{add_because_of_field} would come in \eqref{after_ldp_2}, when we quote a calculation from \cite{talagrand06b};
see Remark \ref{even_spin_remark}.
\end{remark}

Following Theorem \ref{main_thm}, it becomes desirable to understand the regularity of the Parisi functional $\PPP$.
%Leaving questions of convexity for future work, 
Here we address its continuity.
%Here we make a continuity statement, but 
First we need a notion of distance on $\vc\lambda$-admissible pairs.
Given a probability measure $\zeta$ on $[0,\infty)$, let $Q_\zeta$ denote its quantile function:
\eq{ %\label{quantile_def}
Q_\zeta(z) \coloneqq \inf\{q\geq0:\zeta\big([0,q]\big)\geq z\}, \quad z\in[0,1].
}
We then have the following pseudometric:
\eeq{ \label{pseudometric_def}
\DD\big((\zeta_1,\Phi_1),(\zeta_2,{\Phi}_2)\big)
\coloneqq \int_0^1 \|\Phi_1(Q_{\zeta_1}(z))-{\Phi}_2(Q_{\zeta_2}(z))\|_1\ \dd z.
}
Note that this is simply the Wasserstein-1 distance between two pushforward measures $\zeta_1\circ \Phi_1^{-1}$ on $\zeta_2\circ{\Phi}_2^{-1}$ on $[0,1]^\SSS$.
In particular, convergence with respect to $\DD$ is equivalent to weak convergence.
Let us emphasize that if we replaced $\DD$ with the seemingly natural option of adding a metric on measures and a norm on functions, then only the forward direction of the previous sentence would be true.
Indeed, it is essential that the converse also be true.
With $\vc 1\in\R^\SSS$ denoting the vector of all ones, our continuity result is the following.

\begin{thm} \label{lipschitz_continuity}
Assume \eqref{decay_condition}.
For any $\vc\lambda$-admissible pairs $(\zeta_1,\Phi_1)$ and $(\zeta_2,{\Phi}_2)$, we have
\eeq{ \label{lipschitz_continuity_eq}
|\PPP(\zeta_1,\Phi_1)-\PPP(\zeta_2,{\Phi}_2)|
\leq \frac{C_*}{2}\DD\big((\zeta_1,\Phi_1),(\zeta_2,{\Phi}_2)\big), \quad \text{where} \quad
C_* \coloneqq \sup_{s,s'\in\SSS}\frac{\partial\xi^{s}}{\partial q^{s'}}(\vc 1).
}
\end{thm}

Note that the quotient topology generated by $\DD$ makes the space of $\vc\lambda$-admissible pairs compact. 
This is because the space of probability measures on $[0,1]$ is compact in the weak topology (see \cite[Rmk.~6.19]{villani09}), as is the space of $\vc\lambda$-admissible maps under the uniform $\ell^1$ norm.
Indeed, thanks to \eqref{lambda_av_consequence}, one can apply the Arzel\`a–Ascoli theorem (see \cite[Thm.~47.1]{munkres00} for a general version) to conclude the latter fact.
In light of this compactness, the continuity in Theorem \ref{lipschitz_continuity} implies the existence of a minimizer to the Parisi formula \eqref{parisi_formula}.

\begin{cor} \label{main_cor}
Assume \eqref{decay_condition}.
Then there exists a $\vc\lambda$-admissible pair $(\wt\zeta,\wt\Phi)$ such that
\eeq{ \label{minimizer}
\PPP(\wt\zeta,\wt\Phi) = \inf_{\zeta,\Phi}\PPP(\zeta,\Phi).
}
\end{cor}

There is great interest in understanding properties of minimizers.
In the spin glass parlance, if $(\wt\zeta,\wt\Phi)$ satisfies \eqref{minimizer}, then $\wt\zeta\circ\wt\Phi^{-1}$ is said to be a \textit{Parisi measure}.
In the single-species case (where the only admissible map is the identity function), the Parisi functional is known to have a unique minimizer.
This is because \eqref{parisi_formula} admits an alternative formulation known as the Crisanti--Sommers formula \cite{crisanti-sommers92,talagrand06b}, whose objective function is strictly convex.
The analogous result for Ising spin glasses (where the spins $\sigma_i$ only take values $\pm1$) is much less clear and was established in \cite{auffinger-chen15b} (see also \cite{jagannath-tobasco16}).
In a companion paper \cite{bates-sohn22b}, we provide the multi-species version of the Crisanti--Sommers formula, and while convexity in $\zeta$ still holds, the same may not be true for $\Phi$.
Therefore, addressing the uniqueness of solutions to \eqref{minimizer} is left for future work.
%is strictly convex in $\zeta$ and thus has a unique minimizer \cite{auffinger-chen15b}.

Whether or not an optimizer in \eqref{main_cor} is supported on a single point classifies the model as either \textit{replica symmetric} (RS) or \textit{replica symmetry breaking} (RSB).
The exact nature of symmetry breaking remains deeply mysterious in many ways, especially for Ising spin glasses.
For various results on this front, see \cite{panchenko-talagrand07,auffinger-chen15a,chen-sen17,jagannath-tobasco18,auffinger-zeng19,auffinger-chen-zeng20}, all dealing with single-species models.
In the multi-species setting, questions of symmetry breaking are even more delicate because of the possibility that symmetry breaking occurs in one species but not another.
However, a key contribution of \cite{bates-sohn22b} is to rule out this possibility under mild and natural assumptions, leading us to say there is \textit{simultaneous} symmetry breaking.
See \cite[Sec.~2.2]{bates-sohn22b}.

Finally, it is worth pointing out that we have made a stylistic choice in expressing the Parisi formula \eqref{parisi_formula} using a continuous functional order parameter.
That is, we allow $\zeta$ to be any Borel probability measure on $[0,1]$.
However, for simplicity, Parisi formulas are often expressed using just $\zeta$ with finite support, and then \eqref{A_def} takes the form \eqref{A_def_0}.
One nice outcome of extending the Parisi functional to all measures is Corollary \ref{main_cor}, although this result is not at all surprising.
%In fact, a similar observation was made in \cite[Prop.~1]{ko?}.
A more consequential outcome takes place in \cite{bates-sohn22b}, where the use of a continuous order parameter is essential to obtaining simultaneous symmetry breaking in the greatest possible generality. 

%the size of the support $\zeta\circ(\Phi^s)^{-1}$ depends on $s$ (see also \cite[Sec.~6]{panchenko15}).

\subsection{Proof sketch for derivation of the Parisi formula}
\label{sketch_section}
This paper synthesizes several themes and tools from the mathematical theory of spin glasses, suitably adapted to the multi-species spherical setting.
Owing to the many technical ingredients, it may be hard to  
identify a cohesive story within a linear reading of the manuscript.
Therefore, in this section we offer a generous overview of the arguments leading to Theorem \ref{main_thm}.  
In broad strokes, the upper bound for \eqref{parisi_formula} is proved in Section \ref{upper_section}, and the lower bound \eqref{parisi_lower} in Sections \ref{redefine_model}, \ref{ass_section}, and \ref{final_lower_section}, while Section \ref{properties_section} contains technical preliminaries needed throughout.
Finally, Appendix \ref{appendix} provides some well-known facts about Gibbs measures that nevertheless cannot be read directly from the literature.
Therefore, we state and prove these facts for a very general setting.

Suppose $\sigma^1,\sigma^2,\dots$ are independent samples from the Gibbs measure $G_N$ of \eqref{gibbs_measure}.
For each pair of indices $\ell,\ell'$, we have a vector of overlaps $\vc \RR_{\ell,\ell'} = \vc R(\sigma^\ell,\sigma^{\ell'})$ as defined in \eqref{overlap_def}.
Since \eqref{covariance_at_N} tells us that the Gaussian field $H_N$ is governed by these overlaps, it can be intuited that the free energy $F_N$ is related to the law of the array $\vc\RR = (\vc \RR_{\ell,\ell'})_{\ell,\ell'\geq1}$, which we denote by $\mathsf{Law}(\vc\RR;G_N)$.\footnote{This is a slight abuse of notation because the Gibbs measure is random.
We mean for $\mathsf{Law}(\vc\RR;G)$ to be a deterministic object depending only on the law of the random Gibbs measure $G$.
More precisely, if we use the shorthand $\vc\LL = \mathsf{Law}(\vc\RR;G)$, then
\eq{
\int f(\vc\RR)\ \vc\LL(\dd\vc\RR) = \E\langle f(\vc\RR)\rangle,
}
where $\langle\cdot\rangle$ averages over the replicas $(\sigma^\ell)_{\ell\geq1}$ according to $G$, and $\E(\cdot)$ denotes expectation over realizations of $G$.
A similar comment will apply to notation introduced in Theorem \ref{ds_rep}.\label{abuse_footnote}}

The Parisi formula \eqref{parisi_formula} makes the relationship between this law and $\lim_{N\to\infty}\E F_N$ precise, and this will be enough since it is a standard fact that $F_N$ concentrates around its mean (see Lemma \ref{general_concentration_lemma}). 
But understanding this relationship---and indeed proving it---requires that we develop two fundamental concepts, namely (i) how the overlap distribution $\mathsf{Law}(\vc\RR;G_N)$ is identified with some pair $(\zeta,\Phi)$; and (ii) how the Parisi functional $\PPP$ emerges as the correct objective function.
The rest of this section is to explain (i) and (ii).

%Let us work at the level of \textit{expected} free energy (concentration is a separate issue addressed in Section REF by standard methods).
For any real-valued sequence $(a_N)_{N\geq1}$, it is an elementary fact that for any $M\geq1$,
\eeq{ \label{elementary_observation}
\liminf_{N\to\infty} \frac{a_N}{N} \geq \frac{1}{M}\liminf_{N\to\infty} (a_{N+M}-a_N).
}
Applying this observation to $a_N = \E\log Z_N$, we have
\eeq{ \label{cavity_begin}
\liminf_{N\to\infty} \E F_N \geq \frac{1}{M}\liminf_{N\to\infty}\E\log \frac{ Z_{N+M}}{Z_N}.
}
This inequality is the basis of the so-called cavity method for proving \eqref{parisi_lower}.
That is, we study how the free energy changes when a fixed number $M$ of ``cavity coordinates'' are added to the configuration space, turning $\sigma\in\T_N$ into $(\sigma,\kappa)\in\T_{N+M}$.
This is done by rewriting the Hamiltonian $H_{N+M}$ in three parts: 
\eq{
H_{N+M}(\sigma,\kappa) = H_{M,N}(\sigma) + \sum_{j=1}^M\kappa_jX_j(\sigma) + D(\sigma,\kappa).
}
More precisely, the first part $H_{M,N}$ consists of all the terms in $H_{N+M}$ that involve no cavity coordinates, the second part isolates those terms with just one cavity coordinate, while the third part contains all other terms and has negligible contribution.
This type of analysis is commonly called the Aizenman--Sims--Starr (A.S.S.) scheme after the influential works \cite{aizenman-sims-starr03,aizenman-sims-starr07}.
In applying this scheme to the present setting, we take as inspiration the work of Chen \cite{chen13} for single-species spherical models.

%The challenge is to relate the right-hand side of \eqref{cavity_begin} to the overlap distribution $\vc\LL$.
%The resulting expression \eqref{ass_thm_eq} may be called an Aizenman--Sims--Starr (A.S.S.) scheme after the influential works \cite{aizenman-sims-starr03,aizenman-sims-starr07} in which this type of analysis was proposed as a way of understanding the validity of Parisi's conjecture \cite{parisi79,parisi80} for the SK model.

The difference between $\E\log Z_{N}$ and $\E\log Z_{N+M}$ is captured by two effects.
First, there is the direct contribution from the terms of the form $\kappa_j X_j(\sigma)$; these collectively increase the free energy by an amount we call $\Pi_{M,1}$.
Second, the only difference between $H_{M,N}$ and $H_N$ is scaling (compare \eqref{mixed_H_def} and \eqref{H_NM_def}), which decreases the free energy by an amount we call $\Pi_{M,2}$.
The beauty of the A.S.S. scheme is that upon replacing $G_N$ by a Gibbs measure $G_{M,N}$ corresponding to the modified Hamiltonian $H_{M,N}$, we can express the quantities $\Pi_{M,1}$ and $\Pi_{M,2}$ as functions of $\vc\LL_{M,N}\coloneqq \mathsf{Law}(\vc\RR;G_{M,N})$.
Indeed, up to negligible terms, Theorem \ref{ass_thm} gives
\eeq{ \label{ass_thm_summary}
\liminf_{N\to\infty} \E F_N \geq \frac{1}{M}\liminf_{N\to\infty} (\Pi_{M,1}(\vc\LL_{M,N})-\Pi_{M,2}(\vc\LL_{M,N})).
}
%Insomuch as the A.S.S. scheme is natural, the exact definition of $\Pi_M$ arises naturally and does not require too much anticipation.
For brevity, we will write $\Pi_M = \Pi_{M,1}-\Pi_{M,2}$.
See Section \ref{prelimit_section} for a precise definition; it is too lengthy to be reproduced here.

In view of \eqref{ass_thm_summary}, one is naturally motivated to pass to a subsequence $(N_k)_{k\geq1}$ along which $\vc\LL_{M,N_k}$ converges weakly to some abstract law $\vc\LL_M$.
Indeed, since $\Pi_M$ is uniformly continuous---a fact we check in Proposition \ref{uniform_continuity}---it can be continuously extended to a domain including $\vc\LL_M$. 
The A.S.S. scheme \eqref{ass_thm_summary} then leads to
\eeq{ \label{ass_thm_summary_rewrite}
\liminf_{N\to\infty} \E F_N \geq \frac{\Pi_M(\vc\LL_M)}{M}.
}
This statement in itself, however, is not so useful, for two reasons:
\begin{itemize}
    \item[(a)]First, there is the technical fact that $\Pi_M(\vc\LL_M)$ is defined only by abstractly extending $\Pi_M$ to a completed domain.
That is, $\Pi_M$ as an explicit functional is conceived as a function of a certain type of object---namely overlap distributions produced from Gibbs measures---and it is not clear that $\vc\LL_M$ can be realized in this way.
Therefore, we do not immediately have an actual formula for $\Pi_M(\vc\LL_M)$.
\item[(b)] Second, there is the more central obstacle that even if $\Pi_M$ were extended via an explicit formula, its definition is too complicated for meaningful analysis (let alone to be compatible with a matching upper bound).
After all, $\vc\LL_M$ is a measure on an infinite-dimensional space, and so we should hope to simplify the dependence of $\Pi_M$ on $\vc\LL_M$ to some finite-dimensional statistic.
\end{itemize}

Let us first recall how issue (a) is resolved in the classical single-species case.
In that setting, $\vc\RR$ would instead be an array of scalars rather than vectors,
namely the replica overlaps averaged across all coordinates (not separately within each species).
Let us denote these averaged overlaps by
\eeq{ \label{overlap_across_all}
\RR_{\ell,\ell'} \coloneqq \frac{1}{N}\sum_{i=1}^N\sigma^\ell_i\sigma^{\ell'}_i = \sum_{s\in\SSS}\lambda^s(N)\RR_{\ell,\ell'}^s.
}
The scalar array $\RR = (\RR_{\ell,\ell'})_{\ell,\ell'\geq1}$ is easily seen to be a Gram--de Finetti array: symmetric, nonnegative definite, and having entries that are exchangeable under finite permutations.
Moreover, as $N\to\infty$, any subsequential weak limit of this array will inherit these properties (see Lemma \ref{check_gram}).
A Gibbs representation is then found by appealing to the Dovbysh--Sudakov theorem \cite{dovbysh-sudakov82,panchenko10II}.

\begin{theirthm} {\textup{\cite[Thm.~1.7]{panchenko13a}}} \label{ds_rep}
Let $\RR = (\RR_{\ell,\ell'})_{\ell,\ell'\geq1}$ be a Gram--de Finetti array such that $\RR_{\ell,\ell}=1$ with probability one for every $\ell\geq1$. 
Then $\RR$ can be coupled with i.i.d.~samples $(\sigma^\ell)_{\ell\geq1}$ from a random measure $\GG$ on the unit ball of a separable Hilbert space, such that with probability one
\eq{ %\label{ds_rep_coupling}
\RR_{\ell,\ell'} = \sigma^\ell\cdot\sigma^{\ell'} + \one_{\{\ell=\ell'\}}(1-\sigma^\ell\cdot\sigma^\ell) \quad \text{for all $\ell,\ell'\geq1$}.
}
In this case, we write $\mathsf{Law}(\RR;\GG)$ to denote the law of $\RR$.
\end{theirthm}

As for issue (b), we need a second fundamental result, which requires that we introduce the Ghirlanda--Guerra (G.G.) identities.
Still in setting of Theorem \ref{ds_rep}, let $\langle\cdot\rangle$ denote the Gibbs average over the independent samples $(\sigma^\ell)_{\ell\geq1}$, while $\E(\cdot)$ will denote expectation over realizations of the Gibbs measure $\GG$.
We say that the array $\RR$ from Theorem \ref{ds_rep} satisfies the G.G. identities if for any bounded measurable function $f$ of the finite sub-array $\RR^n = (\RR_{\ell,\ell'})_{\ell,\ell'\in[n]}$, and any bounded measurable $\psi\colon[-1,1]\to\R$, we have
\eeq{ \label{GG_for_RR_with_Gibbs}
\E[\langle f(\RR^n)\psi(\RR_{1,n+1})\rangle]
= \frac{1}{n}\E\langle f(\RR^n)\rangle\cdot\E\langle \psi(\RR_{1,2})\rangle + \frac{1}{n}\sum_{\ell=2}^n \E\langle f(\RR^n)\psi(\RR_{1,\ell})\rangle.
}

\begin{theirthm} \label{representation_thm}
\textup{\cite[Thm.~2.13, 2.16, and 2.17]{panchenko13a}}
Let $\RR$ and $\GG$ be as in Theorem \ref{ds_rep}.
If $\RR$ satisfies the G.G. identities \eqref{GG_for_RR_with_Gibbs}, then 
\begin{enumerate}[label=\textup{(\alph*)}]
    \item \label{representation_thm_a}
    $\mathsf{Law}(\RR;\GG)$ depends only on
the probability measure $\zeta$ on $[-1,1]$ defined by
\eq{ %\label{zeta_from_R12_intro}
\zeta(\cdot) = \E\langle \one_{\{\RR_{1,2}\in \cdot\}}\rangle.
}
\item \label{representation_thm_b} (Talagrand's positivity principle) In fact, $\zeta\big([0,1]\big) = 1$.
\item \label{representation_thm_c} The map $\zeta\mapsto\mathsf{Law}(\RR;\GG)$ is continuous with respect to weak convergence.
\end{enumerate}
\end{theirthm}

In summary, we have considered some distributional limit of the infinite scalar array from \eqref{overlap_across_all}.
First Theorem \ref{ds_rep} allows us to couple this limit to an abstract Gibbs measure.
Then Theorem \ref{representation_thm} gives conditions under which this limit can be completely identified by just a single marginal, which is some probability measure $\zeta$ on $[0,1]$.
The extreme reduction brought by this second result should underscore just how strong the G.G. identities are.
Because these identities have played such a critical role in modern spin glass theory,
there is fortunately a standard perturbation technique to ensure they are satisfied by some overlap distribution realized in the large-$N$ limit; we carry this out in Appendix \ref{appendix} for a very general setting.

%
%Now let us return to the multi-species setting. 
%Regarding issue (a), it is still the case that for each $s\in\SSS$, the overlap array $(\RR_{\ell,\ell'}^s)_{\ell,\ell'\geq1}$ is a Gram--de Finetti array.
%However, these arrays may have a complicated correlation structure across species, which would limit the benefit of giving a Dovbysh--Sudakov representation separately to each one.
%This is why we apply \eqref{ds_rep} only once, specifically to

To connect these results back to the multi-species setting, recall the limit $\vc\LL_M$ from before; this is some law on infinite vector arrays.
Suppose $\vc\RR$ is distributed according to $\vc\LL_M$.
The breakthrough of Panchenko \cite{panchenko15} was to identify a ``synchronization'' theory by which the vector array $\vc\RR$ is proved to be a deterministic function of the scalar array $\RR$, provided that a \textit{multi-species} version of the G.G. identities is satisfied.
Namely, given any bounded measurable function $\vphi\colon[-1,1]^\SSS\to\R$, define $Q_{\ell,\ell'} = \vphi(\vc \RR_{\ell,\ell'})$.
We say that $\vc\RR$ satisfies the multi-species G.G. identities if for any bounded measurable function $f$ of the finite sub-array $\vc\RR^n = (\vc\RR_{\ell,\ell'})_{\ell,\ell'\in[n]}$, we have
\eeq{ \label{GG_projection_intro}
\E[f(\vc\RR^n)Q_{1,n+1}]
= \frac{1}{n}\E[f(\vc\RR^n)]\cdot\E[Q_{1,2}]
+ \frac{1}{n}\sum_{\ell=2}^n \E[f(\vc\RR^n)Q_{1,\ell}].
}
Then Panchenko's result is the following.

\begin{theirthm}\textup{\cite[Thm.~4]{panchenko15}}
\label{sync_thm}
%Let $\vc\RR = (\vc \RR_{\ell,\ell'})_{\ell,\ell'\geq1}$ be any array of $\SSS$-tuples such that for each $s\in\SSS$, the array of scalars $(\RR_{\ell,\ell'}^s)_{\ell,\ell'\geq1}$ is a Gram--de Finetti array.
%Let $\RR = (\RR_{\ell,\ell'})_{\ell,\ell'\geq1}$ be the convex combination of these scalar arrays as in \eqref{average_overlap_def}.
If $\vc\RR$ satisfies the multi-species G.G. identities \eqref{GG_projection_intro}, then there exist non-decreasing $(1/\lambda^s)$-Lipschitz functions $\Phi^s\colon[0,1]\to[0,1]$ such that almost surely,
\eeq{ \label{sync_eq_1}
\RR_{\ell,\ell'}^s = \Phi^s(\RR_{\ell,\ell'}) \quad \text{for all $\ell,\ell'\geq1$, $s\in\SSS.$}
}
\end{theirthm}

Definition \ref{lambda_av_def} can now be understood as a characterization of the fact that $\RR_{\ell,\ell'}$ is recoverable from $\Phi(\RR_{\ell,\ell'})$ by way of \eqref{overlap_across_all}.
Regarding the hypotheses of Theorem \ref{sync_thm}, the following comment is essential and thus set aside to be referenced later on.

\begin{remark} \label{gg_remark}
%As usual, for any positive integer $n$, we will write $\RR_M^n = (\R\RR_{\ell,\ell'})_{\ell,\ell'\in[n]}$.
If $\vc\RR$ satisfies the multi-species G.G. identities \eqref{GG_projection_intro}, then $\RR$ satisfies the classical G.G. identities.
%That is, for any bounded measurable function $f$ of the finite sub-array $\RR^n = (\RR_{\ell,\ell'})_{\ell,\ell'\in[n]}$ and any bounded measurable $\psi\colon[-1,1]\to\R$, we have
%\eeq{ \label{GG_for_RR}
%\E[ f \psi(\RR_{1,n+1})]
%= \frac{1}{n}\E[f]\cdot\E[ \psi(\RR_{1,2})] + \frac{1}{n}\sum_{\ell=2}^n \E[ f\cdot\psi(\RR_{1,\ell})],
%}
%where $\langle\cdot\rangle$ denotes expectation with respect to the measure $\GG$ from Theorem \ref{ds_rep}.
Indeed, to verify \eqref{GG_for_RR_with_Gibbs}, simply set $\phi(\vc x) = \sum_{s\in\SSS}\lambda^s x^s$, and take $\vphi = \psi\circ\phi$ in \eqref{GG_projection_intro}.
Once the G.G. identities are known to hold for $\RR$, Theorem \hyperref[representation_thm_b]{\ref*{representation_thm}\ref*{representation_thm_b}} guarantees that $\RR_{\ell,\ell'}\geq0$ with probability one.
Therefore, the domain of $\Phi^s$ makes sense.
\end{remark}

%More specifically, there is a map $\Phi:\R\mapsto\R^\SSS$ such that 
%\eeq{ \label{synchronized}
%\Phi(\RR_{\ell,\ell'})= \vc \RR_{\ell,\ell'} \quad \text{for all $\ell,\ell'\geq1$}.
%}
%This is the content of Theorem \ref{sync_thm}.
%Definition \ref{lambda_av_def} can then be understood as a characterization of the fact that $\RR_{\ell,\ell'}$ is recoverable from $\Phi(\RR_{\ell,\ell'})$ by way of \eqref{overlap_across_all}.
%
%The existence of the synchronization map $\Phi$ is contingent on $\vc\RR$ satisfying a multi-species version of the Ghirlanda--Guerra (G.G.) identities \cite{ghirlanda-guerra98}, which are stated in Lemma \ref{GG_projection}.
%Because these identities have played such a critical role in modern spin glass theory,
%there is fortunately a standard perturbation technique to ensure they are satisfied by some overlap distribution realized in the large-$N$ limit; we carry this out in Appendix \ref{appendix} for a very general setting.
As we mentioned before, it is possible via perturbation to guarantee that the G.G. identities hold, so that Theorems \ref{representation_thm} and \ref{sync_thm} can be applied.
Correspondingly, the A.S.S. scheme discussed previously actually needs to be performed for a perturbed Hamiltonian which is defined in Section \ref{perturb_section}.
But once this is done, we may assume that the law $\vc\LL_M$ appearing in \eqref{ass_thm_summary_rewrite} satisfies the G.G. identities and is thus a candidate for Panchenko's synchronization theory.
More precisely, $\vc\LL_M$ has the following representation.
For a random vector array $\vc\RR$ whose law is $\vc\LL_M$, let $\LL_M$ be the law of the scalar array $\RR$ realized by the map $\vc\RR\mapsto\RR$ from \eqref{overlap_across_all}.
Then there is some synchronization map 
$\RR\mapsto \Phi_M(\RR)=\vc\RR$ under which $\vc\LL_M$ has the pushforward representation\footnote{We again ask the reader to tolerate a slight abuse of notation, since the argument of a synchronization map such as $\Phi_M$ is not an entire array but rather a single real number.
But when it is convenient do so, we think of $\Phi_M$ as acting on the full array $\RR$ by acting separately on every entry.\label{pushforward_foonote}} $\vc\LL_M = \LL_M \circ \Phi_M^{-1}$.
Furthermore, the scalar array $\RR$ satisfies the hypotheses of Theorem \ref{ds_rep}, and so there is a random Gibbs measure $\GG_M$ such that $\LL_M = \mathsf{Law}(\RR;\GG_M)$.
Putting these two facts together, we have
\eeq{ \label{gibbs_representation}
\vc\LL_M = \mathsf{Law}(\RR;\GG_M) \circ \Phi_M^{-1}.
}

\begin{remark} \label{avoid_difficulties_remark}
At this point, the Gibbs representation \eqref{gibbs_representation} does make an explicit definition of $\Pi_M(\vc\LL_M)$ possible.
However, the fact that the Gibbs measure $\GG_M$ is on an infinite-dimensional space poses certain technical difficulties we would rather avoid. 
Therefore, we will content ourselves with simply knowing that $\vc\LL_M$ has a Gibbs representation rather than trying to use that representation to write down an explicit formula for $\Pi_M(\vc\LL_M)$. 
Indeed, the former is essential for overcoming issue (a) declared before, while the latter is not.
%So we continue our discussion with $\Pi_M(\vc\LL_M)$ defined only implicitly.
\end{remark}

%The G.G. identities are also known to produce another dramatic simplification which addresses issue (b).
%Namely, if a (scalar) overlap array satisfies the G.G. identities, then its law is completely determined by the marginal of $\RR_{1,2}$ (see Theorem \ref{representation_thm}).
In light of Remark \ref{gg_remark}, we are further able to apply Theorem \hyperref[representation_thm_a]{\ref*{representation_thm}\ref*{representation_thm_a}} to the array $\RR$.
This means that in the representation \eqref{gibbs_representation},
the quantity $\mathsf{Law}(\RR;\GG_M)$ is completely determined
by the law of $\sigma^1\cdot\sigma^2$ under $\E(\GG_M^{\otimes2})$, which is just some measure $\zeta_M$ on $\R$.
Since $\vc\LL_M$ is now seen to depend only on the $\vc\lambda$-admissible pair $(\zeta_M,\Phi_M)$, we can rewrite \eqref{ass_thm_summary_rewrite} as
\eeq{ \label{ass_thm_summary_rerewrite}
\liminf_{N\to\infty} \E F_N \geq \frac{\PPP_M(\zeta_M,\Phi_M)}{M},
}
where now $\PPP_M$ is a simpler function realized when $\Pi_M$ is restricted to overlap distributions satisfying the G.G. identities.
This function is defined more precisely in Section \ref{restrict_prelimit}, and \eqref{ass_thm_summary_rerewrite} later appears as Proposition \ref{find_pair}.

The last step to prove the lower bound \eqref{parisi_lower} is understanding the dependence of \eqref{ass_thm_summary_rerewrite} on $M$.
%Moreover, \eqref{pseudometric_def} is simply the Wasserstein-1 distance between two laws $\zeta\circ \Phi^{-1}$ on $\wt\zeta\circ\wt{\Phi}^{-1}$ on $[0,1]^\SSS$.
To obtain a Parisi formula for Ising spin glasses, it suffices to consider just a single value of $M$; see \cite[Sec.~3.5]{panchenko13a}. 
This remains true even in the multi-species setting \cite{panchenko15}.
For spherical models, however, the functional $\Pi_M$ is too complicated to yield a useful objective function.
The strategy thus pivots to finding a limit as $M\to\infty$.

In the single-species case, a large deviations calculation of Talagrand \cite{talagrand06b} (used here in \eqref{after_ldp_2}) would establish that $\PPP_M/M$ converges to a limiting functional similar to $\PPP$ from \eqref{parisi_functional_def}.
The difficulty here, however, is that the preceding steps have already required we send $N\to\infty$, and the number of cavity coordinates assigned to each species does not necessarily converge as $N\to\infty$.
An obvious workaround is to pass to a subsequence along which these limits \textit{do} exist, but even then it is not necessarily true that as $M$ tends to infinity, the fraction of cavity coordinates allocated to species $s$ converges (let alone to $\lambda^s$).
Therefore, a critical step---carried out in Section \ref{redefine_model} before the cavity method and synchronization---is to actually redefine the model \eqref{original_TN_def} in a strategic way, in order to ensure that these species proportions behave properly even once $M$ is brought to infinity.
For this redefined model, we can use Talagrand's calculation to identify \eqref{parisi_functional_def} as the limiting functional; see Proposition \ref{explain_appearance}.
By further passing to a subsequence along which $(\zeta_M,\Phi_M)$ converges to some $(\zeta,\Phi)$, we obtain
\eq{
\lim_{M\to\infty}\frac{\PPP_M(\zeta_M,\Phi_M)}{M} = \PPP(\zeta,\Phi).
}
In view of \eqref{ass_thm_summary_rerewrite}, this immediately implies the lower bound \eqref{parisi_lower}.

The task of establishing the matching upper bound is less involved.
 In Proposition \ref{prop:upper:bound}, we use the standard approach of Guerra's RSB interpolation to verify that
\eeq{ \label{desired_upper}
\lim_{N\to\infty} \E F_N \leq \PPP(\zeta,\Phi) \quad \text{for any $\vc\lambda$-admissible pair $(\zeta,\Phi)$}.
}
The interpolation is reminiscent of \cite[Sec.~3]{ko20} in that the interpolating Hamiltonian $\Hb_{N,t}(\sigma,\alpha)$ has two arguments: $\sigma\in\T_N$ and $\alpha\in\N^{k-1}$, where the reference measure on $\N^{k-1}$ is a Poisson--Dirichlet cascade (see Section \ref{cascade_review} for a review).
When $t=0$, the resulting Gibbs measure is a product measure, allowing the original free energy $F_N$ to be easily recovered.
When $t=1$, the configurations $\sigma$ and $\alpha$ are coupled in such a way that the functional $\PPP_M$ from \eqref{ass_thm_summary_rerewrite} appears.
%\textcolor{orange}{May be we should mention multi-species analogue of Talagrand's positivity principle here}
The convexity assumption \eqref{xi_convex} ensures the desired inequality \eqref{desired_upper}; see Claim \ref{lem:guerra:interpolation}.
In fact, this is the only place convexity is required.
It is worth noting that \eqref{xi_convex} is needed only on the nonnegative orthant, even though overlaps can be negative.
This narrowing of the domain is enabled by Talagrand's positivity principle (Theorem \hyperref[representation_thm_b]{\ref*{representation_thm}\ref*{representation_thm_b}}), a multi-species version of which is proved in Lemma \ref{lem:positivity:principle}.

\subsection{Related works} The Parisi formula for the classical SK model with Ising spins was first proved by Talagrand \cite{talagrand06a}, building on the seminal work of Guerra \cite{guerra03} which introduced the technique of RSB interpolation.
Later, Panchenko proved the Parisi formula for general mixed $p$-spin models \cite{panchenko14} by showing that the Ghirlanda--Guerra identities imply ultrametricity for replica overlaps \cite{panchenko13b}. Recently Mourrat \cite{mourrat22} has reinterpreted these Parisi formulas as the solution to a Hamilton--Jacobi equation in the Wasserstein
space of probability measures on the positive half-line; see \cite{mourrat21a,mourrat20,chen22,chen_xia22,chen_mourrat_xia22} for finite-dimensional analogues, and \cite{mourrat-panchenko20} for a generalized result.

In the context of spherical spin glasses, the Parisi formula for mixed $p$-spin models with even $p$ was proved by Talagrand \cite{talagrand06b} and extended by Chen \cite{chen13} to include odd $p$-spin interactions. Later, Subag \cite{subag17} computed the logarithmic second-order term for the free energy of pure $p$-spin models with $p\geq 3$, by developing a geometric description of the Gibbs measure at low enough temperature. Further analysis was carried out for mixed $p$-spin spherical models close to pure by Ben Arous, Subag, and Zeitouni \cite{benarous-subag-zeitouni20}.

The general multi-species SK model (Ising case) was introduced in \cite{barra-contucci-mingione-tantari15}, where Barra \textit{et.~al.}~gave an upper bound for the free energy using a variant of Guerra's RSB bound \cite{guerra03}, under a condition equivalent to \eqref{xi_convex}.
Panchenko produced the matching lower bound in \cite{panchenko15} by using the synchronization mechanism discussed above.
By generalizing this mechanism, Panchenko obtained variational formulas for the free energy of Potts spin glass models \cite{panchenko18a} and mixed $p$-spin models with vector spins \cite{panchenko18b}. 
The synchronization technique has since been pivotal in a variety of related models \cite{jagannath-ko-sen18,contucci-mingione19,chen19,ko20,mourrat-panchenko20,mourrat23}.
%Using a similar technique, Ko obtained the analogue for multiple spherical models subject to constrained overlaps \cite{ko20}.
Using the formula produced by Panchenko in \cite{panchenko15}, the authors together with Sloman \cite{bates-sloman-sohn19} studied symmetry breaking for multi-species SK models (see also \cite{hartnett-parker-geist18} from the physics literature).
This work has since been improved by Dey and Wu \cite{dey-wu21}, who also considered non-convex models and properties of the replica symmetric phase.
The RS condition identified in \cite{bates-sloman-sohn19,dey-wu21} also leads to fluctuation results \cite{liu21}.

A natural and interesting special case is a bipartite model, in which two species interact with each other but not among themselves.
In the Ising case, there are conjectured formulas for the limiting free energy \cite{barra-genovese-guerra11,barra-galluzzi-guerra-pizzoferrato-tantari14,mourrat21b} of the bipartite SK model, although not much is known rigorously.
See \cite{alberici-barra-contucci-mingione20,alberici-contucci-mingione21,genovese23,agliari_albanese_alemanno_fachechi21} for results on a generalization of the bipartite SK model, and \cite{alberici-camilli-contucci-mingione21I,alberici-camilli-contucci-mingione21II} for its restriction to a special subset of phase space.

More progress has been made for spherical bipartite models.
Auffinger and Chen \cite{auffinger-chen14} proved a variational formula for the free energy at high temperature (i.e.~$\xi(\vc 1)$ is sufficiently small); see also the recent min-max formulation \cite{genovese22}.
Focusing on the SK version, Baik and Lee \cite{baik-lee20} were able to obtain a formula at all temperatures and also determine limiting fluctuations by drawing on connections with random matrix theory.
In all of these works, the fundamental difficulty is that bipartite models do not satisfy \eqref{xi_convex}.
This causes Guerra's interpolation method---among other things---to break down, although certain methods can bypass this issue, for instance complexity-based approaches \cite{mckenna24,kivimae23} and the TAP representation (pioneered by Thouless, Anderson, and Palmer \cite{thouless-anderson-palmer77}).

Regarding the latter, a trio of works by Subag \cite{subag25,subag23b,subag23c} appeared shortly after this paper was first released, containing respectively (i) a TAP representation for the free energy of general multi-species spherical models; (ii) an analysis of the critical inverse temperature in such models; and (iii) a formula for the limiting free energy \eqref{parisi_formula}  in \textit{pure} models (i.e.~$\xi(\vc q) = \beta^2 q^s$ for some $s\in\SSS^p$, $p\geq2$), which do not satisfy \eqref{xi_convex}.
The TAP approach executed in \cite{subag25,subag23c} is analogous to \cite{subag24,subag23a} in the single-species case (with \cite{subag23a} going beyond the aforementioned \cite{subag17} to cover all temperatures); that methodology bypasses the Parisi framework of the present paper and works on the assumption that $\E(F_N)$ converges as $N\to\infty$.
At present, this assumption is not known rigorously beyond the cases considered here and in \cite{auffinger-chen14,baik-lee20}.
%\subsection{Outline of paper}
%In broad strokes, the upper bound for Theorem \ref{main_thm} is contained in Section \ref{upper_section}, the lower bound in Sections \ref{redefine_model}, \ref{ass_section}, and \ref{final_lower_section}, while Section \ref{properties_section} covers technical ingredients needed throughout.

\section{Properties of the Parisi functional} \label{properties_section}
This section develops some preliminary facts about the Parisi functional \eqref{A_def}, including Theorem \ref{lipschitz_continuity}.
Establishing these facts requires that we return to the analytic origins of this functional, which are motivated by the A.S.S.~scheme of Theorem \ref{ass_thm}.
Consequently, the motivation for some of the coming definitions may currently seem absent, although our work here will ultimately streamline the arguments in later sections.
Since the current section is quite long, we provide the reader a road map of its contents:

\begin{itemize}
    \item In Section \ref{prelimit_section} we will define a sequence of functionals $(\Pi_M)_{M\geq1}$ such that, in a suitable sense, $\Pi_M/M$ converges as $M\to\infty$ to the Parisi functional $\PPP$ from \eqref{parisi_functional_def}.
The functional $\Pi_M$ is the central player that emerges from the cavity method, which will be developed in Section \ref{ass_section}.
The key fact we prove here is a uniform continuity statement (Proposition \ref{uniform_continuity}).

\item Finer analysis of $\Pi_M$ is only possible once we restrict its domain
to certain ``nice'' overlap distributions which are synchronized (in the sense of \eqref{sync_eq_1}), satisfy the Ghirlanda--Guerra identities, and are such that individual overlaps $\RR_{\ell,\ell'}$ can only take finitely many values.
The overlap distributions satisfying the last two conditions are precisely those generated by the Ruelle probability cascades.
Section \ref{cascade_review} gives a self-contained review of the relevant facts about these fundamental objects.

\item We perform the restriction of $\Pi_M$ to these nice distributions in Section \ref{restrict_prelimit}.
For clarity and so that we can transition to the language of $\vc\lambda$-admissible pairs, we give this restriction its own notation: $\PPP_M$. 
We then prove Lipschitz continuity for $\PPP_M$ (Proposition \ref{lipschitz_continuity_restricted}) and convergence to the Parisi functional $\PPP$ (Proposition \ref{explain_appearance}).

\item Throughout Section \ref{restrict_prelimit} the functional $\PPP_M$ is defined only on $\vc\lambda$-admissible pairs $(\zeta,\Phi)$ in which $\zeta$ has finite support.
With Lipschitz continuity established on this dense subset, we start Section \ref{extension_section} by continuously extending $\PPP_M$ to all $\vc\lambda$-admissible pairs.
The limiting functional $\PPP$ could also be implicitly extended, but we would like to know that this extension coincides with the definition \eqref{parisi_formula}.
Therefore, we prove directly that $\PPP$ is continuous (Proposition \ref{density_lemma}).
A short proof of Theorem \ref{lipschitz_continuity} then follows.
\end{itemize}

\subsection{Prelimit of the Parisi functional} \label{prelimit_section}
A key difficulty is that the domain of $\PPP$ is, in loose terms, restricted to ``synchronized'' overlap distributions.
This synchronization is only realized in the large-$N$ limit, and so the functional $\Pi_M$ must be defined more broadly in order to include the overlap distributions realized from finite-volume Gibbs measures.
We will soon make this definition, but first we require the following setup.

\subsubsection{The cavity space}
Suppose we have fixed a partition of the integer interval $[M]$ into the various species, say $[M] = \biguplus_{s\in\SSS}(\JJ^s)_{s\in\SSS}$, where $|\JJ^s| = M^s$.
Analogously to \eqref{original_TN_def}, we consider the following product of spheres:
\begin{subequations}
\label{cavity_space_def}
\eeq{
\mathbf{T}_M \coloneqq \Motimes_{s\in\SSS} S_{M^s},
}
which is equipped with the corresponding product measure,
\eeq{
\vc \tau_M \coloneqq \Motimes_{s\in\SSS}\mu_{M^s}.
}
\end{subequations}

\subsubsection{Allowable overlap maps}
In Section \ref{sketch_section} we introduced the notation $\mathsf{Law}(\vc\RR;G)$ to denote the law of the replica overlap array $\vc\RR$ when the i.i.d.~replicas are drawn from the random Gibbs measure $G$.
In that case $\vc\RR$ was defined via the map $(\sigma,\sigma')\mapsto\vc R(\sigma,\sigma')$ from \eqref{overlap_def}, but now we allow any map fitting the following description.
Let $\Sigma$ be a metric space, and take any continuous symmetric function $\vc R\colon\Sigma\times\Sigma\to[-1,1]^\SSS$ satisfying the following condition.
%\begin{itemize}
    %\item[(A1)] There is a random Borel probability measure $\GG$ on some metric space $\Sigma$, and some deterministic, measurable function $\vc R\colon\Sigma\times\Sigma\to[-1,1]^\SSS$ such that if $(\sigma^\ell)_{\ell\geq1}$ are i.i.d.~samples from $\Sigma$, then with probability one we have
    %\eq{
    %\vc \RR_{\ell,\ell'} = \vc R(\sigma^\ell,\sigma^{\ell'})
    %+\one_{\{\ell=\ell'\}}(1-\vc R(\sigma^\ell,\sigma^\ell))\quad \text{for all $\ell,\ell'$}.
    %}
    %\item[(A1)] There is some $\vc q\in[0,1]^\SSS$ such that 
    %\eq{
    %\vc R(\sigma,\sigma) = \vc q \quad \text{for all $\sigma\in\Sigma$}.
    %}
    %\item [(A)] 
    %\end{itemize}

\begin{assumption} \label{map_assumption}
    There exist centered Gaussian processes $(X_j)_{j\in[M]}$ and $Y$ on $\Sigma$ whose covariance structures are given by
\eeq{ \label{A3_cov}
\E[X_{j}({\sigma})X_{j'}({\sigma}')] &= 
\one_{\{j=j'\}}\xi^s(\vc R(\sigma,\sigma'))
\quad \text{for $j\in\JJ^s$}, \\
\E[Y(\sigma)Y(\sigma')] &= 
\theta(\vc R(\sigma,\sigma')).
}
Furthermore, these processes are almost surely measurable functions on $\Sigma$.
\end{assumption}

\subsubsection{The overlap distribution} \label{prelimit_section_3}

Given a random (Borel) probability measure $G$ on $\Sigma$ which is independent of the processes from \eqref{A3_cov}, let
$(\sigma^\ell)_{\ell\geq1}$ be i.i.d.~samples from $G$.
Apply the overlap map $\vc R$ to each pair of samples, and set
\eeq{ \label{generic_array}
\vc \RR_{\ell,\ell'} = \vc R(\sigma^\ell,\sigma^{\ell'})
    +\one_{\{\ell=\ell'\}}(\vc 1-\vc R(\sigma^\ell,\sigma^{\ell'})),
}
where $\vc 1\in\R^\SSS$ is the constant vector with $1$ in every coordinate.
This defines a random array $\vc\RR = (\vc\RR_{\ell,\ell'})_{\ell,\ell'\geq1}$.
Denote the law of $\vc\RR$ by $\mathsf{Law}(\vc\RR;G)$, where the dependence on $\vc R$ is implicit (also recall Footnote \ref{abuse_footnote}).

\begin{remark}
When we are not using a generic $\Sigma$ and $\vc R$, their identities should always be clear from context.
Outside of this Section \ref{prelimit_section}, there are really only two cases we need to consider.
The first is when $\Sigma=\T_N$ and $\vc R$ is equal to the map from \eqref{overlap_def}, in which case Assumption \ref{map_assumption} is verified in Remarks \ref{X_existence} and \ref{Y_existence}.
Moreover, the Gibbs measure $G$ will usually be $G_{M,N}$, meaning the distribution of \eqref{generic_array}, namely $\mathsf{Law}(\vc\RR;G_{M,N})$, is the same one discussed in Section \ref{sketch_section}.

The second case is when $\Sigma$ is some abstract Hilbert space and $\vc R$ is the composition of its inner product with some $\vc\lambda$-admissible map $\Phi$.
That is, $\vc R(\sigma,\sigma') = \Phi(\sigma\cdot\sigma')$.
Here the Gibbs measure $G$ will be some $\GG$ as in Theorem \ref{ds_rep}.
Using the notation of \eqref{gibbs_representation}, we then have $\mathsf{Law}(\vc\RR;\GG) = \mathsf{Law}(\RR;\GG)\circ\Phi^{-1}$. %(recall Footnote \ref{pushforward_foonote}).
\end{remark}

\subsubsection{The functional}
%Finally, let $\eta,\eta_1,\dots,\eta_M$ be i.i.d.~standard normal random variables independent of everything else.
%With $\E_\eta$ denoting expectation over only these variables, and 
We are finally ready to define the functional $\vc\LL\mapsto\Pi_M(\vc\LL)$.
It accepts as input any law $\vc\LL = \mathsf{Law}(\vc\RR;G)$ realized as above.

Take $(\eta_j)_{j\in[M]}$ and $\eta'$ to be standard normal random variables that are independent of each other and everything else.
Let $\E_\eta$ denote expectation over just these variables, and set
\eq{
X_j^\eta(\sigma) &\coloneqq X_j(\sigma)+\eta_j\sqrt{\xi^s(\vc 1)-\xi^s(\vc R(\sigma,\sigma))} \quad \text{for $j\in\JJ^s$}, \\
Y^\eta(\sigma) &\coloneqq Y(\sigma)+\eta'\sqrt{\theta(\vc 1)-\theta(\vc R(\sigma,\sigma))}.
}
Let $\langle\cdot\rangle$ denote expectation with respect to $G$. Finally, let $\E(\cdot)$ denote expectation over both realizations of $G$ and the Gaussian processes from Assumption \ref{map_assumption}.
Now define the following quantities:
\begin{subequations}
\label{defining_Psi_parts}
\eeqs{
\label{psi_1_def}
\Pi_{M,1}(\vc\LL) &\coloneqq \E\log\int_{\mathbf{T}_M}
\E_\eta\Big\langle\exp\Big(\sum_{j=1}^M\kappa_jX_j^\eta(\sigma)\Big)\Big\rangle\ \vc\tau_M(\dd\kappa), \\
%\Big\langle\exp\Big(\sum_{j=1}^M\kappa_jX_j(\sigma)\Big)\Big\rangle\ \vc\tau_{M}(\dd\kappa)
%+\sum_{s\in\SSS}\frac{M^s}{2}(\xi^s(\vc 1)-\xi^s(\vc q)),  \\
\label{psi_2_def}
\Pi_{M,2}(\vc\LL) &\coloneqq \E\log\E_\eta\big\langle \exp\big(\sqrt{M}Y^\eta(\sigma)\big)\big\rangle.
%+\frac{M}{2}(\theta(\vc 1)-\theta(\vc q)). 
}
\end{subequations}
The functional of interest is then given by
\eeq{ \label{Psi_def}
\Pi_M(\vc\LL) \coloneqq \Pi_{M,1}(\vc\LL)-\Pi_{M,2}(\vc\LL).
}

\begin{remark} \label{make_sense_remark}
In order for \eqref{defining_Psi_parts} to make sense, we need to know that $\exp(\sum_{j=1}^M\kappa_jX_j^\eta(\cdot))$ and $\exp(\sqrt{M} Y^\eta(\cdot))$ are almost surely integrable with respect to the Gibbs measure $G$.
This is actually automatic from the boundedness of overlaps.
Indeed, since the Gaussianity is assumed to be independent of $G$, we can average over the former before the latter.
That is,
\eeq{ \label{EX_computation}
\E\Big\langle\exp\Big(\sum_{j=1}^M\kappa_jX_j^\eta(\sigma)\Big)\Big\rangle
&\stackrefp{A3_cov}{=} \E_G\Big\langle \E_{X}\E_\eta\exp\Big(\sum_{j=1}^M\kappa_jX_j^\eta(\sigma)\Big)\Big\rangle \\
%&=\E_G\Big\langle\exp\Big(\sum_{s\in\SSS}\sum_{j\in\JJ^s}\frac{\kappa_j^2\xi^s(\vc R(\sigma,\sigma))}{2}\Big)\Big\rangle \\
&\stackref{A3_cov}{=} \E_G\Big\langle\exp\Big(\sum_{s\in\SSS}\sum_{j\in\JJ^s}\frac{\kappa_j^2\xi^s(\vc1)}{2}\Big)\Big\rangle
=\exp\Big(\sum_{s\in\SSS}\frac{M^s\xi^s(\vc1)}{2}\Big),
}
and by similar reasoning
\eeq{ \label{EY_computation}
\E\big\langle \exp(\sqrt{M}Y^\eta(\sigma))\big\rangle = \exp\Big(\frac{M\theta(\vc 1)}{2}\Big).
}
In particular, the processes $\exp(\sum_{j=1}^M\kappa_jX_j^\eta(\cdot))$ and $\exp(\sqrt{M}Y^\eta(\cdot))$ are integrable with probability one.
\end{remark}

\begin{remark}
In future sections, it will always be the case that $\vc R(\sigma,\sigma)$ is constant under the Gibbs measure $G$.
That is, there is some $\vc q_*\in[-1,1]^\SSS$ such that
\eeq{ \label{constant_overlap_for_G}
\langle\one_{\{\vc R(\sigma,\sigma)=\vc q_*\}}\rangle = 1.
}
For instance, when $\vc R$ is given by \eqref{overlap_def}, then clearly $\vc R(\sigma,\sigma) = \vc1$ for all $\sigma\in\T_N$.
This means the presence of $\eta_j$ and $\eta'$ in \eqref{defining_Psi_parts} will be unimportant when we apply the functional $\Pi_M$ to any Gibbs measure on $\T_N$ (as in Section \ref{ass_section}). 
Even if $\vc q_*$ is not equal to $\vc 1$, the assumption of \eqref{constant_overlap_for_G} does simplify the expressions in \eqref{defining_Psi_parts}.
Indeed, by using the fact that $\E\exp(c\eta) = \exp(c^2/2)$, we obtain 
\begin{subequations}
\label{redefining_Psi_parts}
\eeqs{
\label{psi_1_redef}
\Pi_{M,1}(\vc\LL) &= \E\log\int_{\mathbf{T}_M}
%\E_\eta\Big\langle\exp\Big(\sum_{j=1}^M\kappa_j\big[X_j(\sigma)+\eta_j\sqrt{\xi^s(\vc 1)-\xi^s(\vc R(\sigma,\sigma))}\, \big]\Big)\Big\rangle\ \vc\tau_M(\dd\kappa), \\
\Big\langle\exp\Big(\sum_{j=1}^M\kappa_jX_j(\sigma)\Big)\Big\rangle\ \vc\tau_{M}(\dd\kappa)
+\sum_{s\in\SSS}\frac{M^s}{2}(\xi^s(\vc 1)-\xi^s(\vc q_*)),  \\
\label{psi_2_redef}
\Pi_{M,2}(\vc\LL) &= \E\log
%\E_\eta\big\langle \exp\big(\sqrt{M}\big[ Y(\sigma)+\eta\sqrt{\xi^s(\vc 1)-\xi^s(\vc R(\sigma,\sigma))}\, \big]\big)\big\rangle.
\big\langle \exp\big(\sqrt{M}Y(\sigma))\big\rangle
+\frac{M}{2}(\theta(\vc 1)-\theta(\vc q_*)). 
}
\end{subequations}
\end{remark}

\begin{remark} \label{Js_freedom}
Notice that $\Pi_{M,1}(\LL)$ does not change if we permute the $X_j$'s.
In this way, the functional depends on the choice of $\JJ^s$ only through its cardinality $M^s$, not on precisely which subset of $[M]$ it is.
\end{remark}

Even given Remark \ref{make_sense_remark}, it may still not be clear that $\Pi_M$ is well-defined, since different choices of $\vc R$ and $G$ may lead to the same law $\vc\LL$ for the array in \eqref{generic_array}.
This will naturally be resolved as follows.
%To state the result, we need to introduce some notation.
Let $\vc\LL^n$ denote the law of the finite sub-array $(\vc \RR_{\ell,\ell'})_{\ell,\ell'\in[n]}$; this is a probability measure on $\SSS$-tuples of symmetric $n\times n$ matrices whose entries lie in $[-1,1]$. %\textcolor{orange}{may be this should be $[-1,1]^{\SSS}$}.
Let $\vc\PP^n$ denote the set of all probability measures on this space.
By compactness, it is easy to metrize the topology of weak convergence on $\vc\PP^n$ by, say, a Wasserstein distance with respect to the Euclidean norm.
We can thus speak of continuity with respect to weak convergence.

%Nevertheless, observe that any Gibbs expectation $\langle\cdot\rangle$ is, by the strong law of large numbers, just a function of i.i.d.~samples $(\sigma^\ell)_{\ell\geq1}$.
%Moreover, the covariance structures \eqref{A3_cov} of $(X_j(\sigma^\ell))_{\ell\geq1}$ and $(Y(\sigma^\ell))_{\ell\geq1}$ depend only on the array $(\vc \RR_{\ell,\ell'})_{\ell,\ell'\geq1}$.
%It follows that the expectations appearing in $\Pi_M$ are completely determined by the law of this array under $\E\GG^{\otimes\infty}$.
%By definition, this law is precisely $\vc\LL$.

%We will need that $\Pi_M$ has sufficient regularity to have a continuous extension to weak limits of laws realized as above.

\begin{prop} \label{uniform_continuity}
For any $\eps>0$, there is $n$ large enough and some continuous function $\Pi_{M}^{(\eps)}:\vc\PP^n\to\R$ such that
\eq{
|\Pi_M(\vc\LL)-\Pi_{M}^{(\eps)}(\vc\LL^n)| \leq \eps \quad \text{for any $\vc\LL$ at which $\Pi_M$ is defined}.
}
\end{prop}

Upon taking $\eps\to0$, it is clear that $\Pi_M$ is well-defined, since $\Pi_{M}^{(\eps)}$ is defined independently of $\vc R$ and $G$; see \eqref{Psi_eps_def}. 
In practice, we will use Proposition \ref{uniform_continuity} via the following consequence.

\begin{cor} \label{extension_cor}
If $(\vc\LL_N)_{N\geq1}$ is any weakly convergent sequence of laws at which $\Pi_M$ is defined,
then $\lim_{N\to\infty}\Pi_M(\vc\LL_N)$ exists and depends only on the limit of $(\vc\LL_N)_{N\geq1}$.
\end{cor}

\begin{proof}
This is a standard application of uniform continuity.
Given any $\eps>0$, let $n$ and $\Pi_{M}^{(\eps)}$ be as in Proposition \ref{uniform_continuity}.
Since $\Pi_{M}^{(\eps)}$ is continuous on the compact space $\vc\PP^n$, it is necessarily uniformly continuous and thus Cauchy continuous.
By assumption, $(\vc \LL_N^n)_{N\geq1}$ is Cauchy, and so $(\Pi_{M}^{(\eps)}(\vc \LL_N^n))_{N\geq1}$ is Cauchy as well.
Consequently, for all $N$ and $N'$ large enough, the difference $|\Pi_M(\vc\LL_N)-\Pi_M(\vc\LL_{N'})|$ is at most
\eq{
|\Pi_M(\vc\LL_N)-\Pi_{M}^{(\eps)}(\vc\LL^n_N)|
+ |\Pi_{M}^{(\eps)}(\vc\LL^n_N)-\Pi_{M}^{(\eps)}(\vc\LL^n_{N'})|
+ |\Pi_{M}^{(\eps)}(\vc\LL^n_{N'})-\Pi_M(\vc\LL_{N'})|
\leq 3\eps.
}
That is, $(\Pi_M(\vc\LL_N))_{N\geq1}$ is Cauchy and thus convergent.

To see that $\lim_{N\to\infty}\Pi_M(\vc\LL_N)$ depends only on the limit of $(\vc\LL_N)_{N\geq1}$, consider two sequences $(\vc\LL_N)_{N\geq1}$ and $(\wt{\vc\LL}_N)_{N\geq2}$ which converge to the same law.
Then the sequence
\eq{
\vc\LL_1,\wt{\vc\LL}_1,\vc\LL_2,\wt{\vc\LL}_2,\vc\LL_3,\wt{\vc\LL}_3,\dots
}
also converges to this law, and so
\eq{
\Pi_M(\vc\LL_1),\Pi_M(\wt{\vc\LL}_1),\Pi_M(\vc\LL_2),\Pi_M(\wt{\vc\LL}_2),\Pi_M(\vc\LL_3),\Pi_M(\wt{\vc\LL}_3),\dots
}
is a convergent sequence by the existence argument given above.
In particular, the two subsequences $(\Pi_M(\vc\LL_N))_{N\geq1}$ and
$(\Pi_M(\wt{\vc\LL}_N))_{N\geq1}$ share the same limit.
\end{proof}

The argument for Proposition \ref{uniform_continuity} follows a general strategy that has appeared before, for instance \cite[Lem.~3]{panchenko14} or \cite[Thm.~1.3]{panchenko13a}.
One complication of note is that our overlap map $\vc R$ is not assumed to be an inner product.

\begin{proof}[Proof of Proposition \ref{uniform_continuity}]
We prove the desired statement for $\Pi_{M,1}$, as the argument for $\Pi_{M,2}$ is similar and in fact simpler.
We start with a truncation procedure.
For $a>0$, define 
\eq{
\log^{(a)}(x) \coloneqq (-a \vee (\log x \wedge a)), \qquad
\exp^{(a)}x \coloneqq \exp(-a\vee(x\wedge a)).
%\exp^{(a*)}x \coloneqq \one_{\{|x|\geq a\}}\exp x.
}
Note for later that
\eeq{ \label{exp_inequality}
|\exp x - \exp^{(a)}x| \leq \one_{\{x> a\}}\exp x + \one_{\{x<-a\}}.
}
For convenience, let us introduce the following notation for a three-fold average:
\eq{
\llangle f(\kappa,\eta,\sigma)\rrangle \coloneqq 
\int_{\mathbf{T}_M}\E_\eta\langle f(\kappa,\eta,\sigma) \rangle\ \vc\tau_M(\dd\kappa),
}
where here $\eta$ denotes the entire collection $(\eta_j)_{j\in[M]}$.
The variable of interest is then
\eq{
&Z^{(a)} \coloneqq %\Big\langle\int_{\mathbf{T}_M}\E_\eta
\Big\llangle\exp^{(a)}\Big(\sum_{j=1}^M\kappa_jX_j^\eta(\sigma)\Big)\Big\rrangle.
}
%\ \vc\tau_M(\dd\kappa)\Big\rangle, 
%\qquad
%Z^{(a*)} \coloneqq %\Big\langle\int_{\mathbf{T}_M}\E_\eta
%\Big\llangle\exp^{(a*)}\Big(\sum_{j=1}^M\kappa_jX_j^\eta(\sigma)\Big)\Big\rrangle.
%\ \vc\tau_M(\dd\kappa)\Big\rangle,
%}
When we have no truncation, we will just write $Z$ for $Z^{(\infty)}$.
Note that $\Pi_{M,1}(\vc\LL) = \E\log Z$.

Observe that by averaging over the $\eta_j$'s, we obtain the following expression for $Z$:
\eq{
Z=
\int_{\mathbf{T}_M\times\Sigma}\exp\Big(\sum_{j=1}^M\kappa_jX_j(\sigma)\Big)\exp\Big(\sum_{s\in\SSS}\frac{M^s}{2}\big(\xi^s(\vc1)-\xi^s(\vc R(\sigma,\sigma))\big)\Big)\ (\vc\tau_M\otimes G)(\dd\kappa,\dd\sigma).
}
So given any realization of the Gibbs measure $G$, the quantity $Z$ is simply the integral of $\exp g(z)$,
where $g$ is a Gaussian process, and the integral is over $z\in\mathbf{T}_M\times\Sigma$ with respect to a finite measure.
While this measure is random (depending on $G$), it is independent of the Gaussian process and thus regarded as fixed.
Consequently, standard Gaussian concentration (see the proof of \cite[Lem.~3]{panchenko07}) gives
\eeq{ \label{constant_C}
\P_X(|\log Z-\E_X\log Z|\geq a)
&\leq 2\exp(-a^2/(4C)), \quad \text{where} \\ 
C = \E\Big[\sum_{j=1}^M\kappa_j X_j^\eta(\sigma)^2\Big]
&= \sum_{s\in\SSS}M^s\xi^s(\vc 1).
}
Since $0\leq \E_X\log Z\leq\log \E_X Z = C/2$, we deduce from this concentration inequality that $|\log Z|$ is not too large:
\eq{
\P_X(|\log Z|\geq a)
&\leq 2\exp(-(a-C/2)^2/(4C)) \quad \text{for $a\geq C/2$} \\
&\leq \rlap{$2\exp(-a^2/(16C))$}\phantom{2\exp(-(a-C/2)^2/(4C))}\quad \text{for $a\geq C$}.
}
In particular, by integrating the tail we obtain the following moment bound:
\eq{
\E_X\log^2 Z \leq C^2 + \int_{C^2}^\infty \P_X(|\log Z|\geq\sqrt{a}\,)\ \dd a
\leq C^2 + 32C.
}
We have made this estimate in order to control the following truncation error: for $a\geq C$ we have
\eeq{ \label{first_trunction}
|\E\log Z - \E\log^{(a)} Z|
\leq \E|\E_X\log Z - \E_X\log^{(a)} Z|
&\leq \E|\one_{\{|\log Z|\geq a\}}\log Z| \\
&\leq \sqrt{\P(|\log Z|\geq a)\E\log^2 Z} \\
&\leq \sqrt{2}\exp(-a^2/(32C))\sqrt{C^2+32C}.\raisetag{2.5\baselineskip}
}
On the other hand, since $\log^{(a)}$ is $\e^a$-Lipschitz, we have
\eeq{ \label{second_truncation}
&\E|\log^{(a)} Z - \log^{(a)}Z^{(a)}|
\leq \e^a\E|Z-Z^{(a)}| \\
%&\stackref{exp_inequality}{\leq} \e^a\E(Z^{(a*)})\\
&\stackref{exp_inequality}{\leq} \e^a\E\Big\llangle\one_{\{\sum_{j=1}^M\kappa_jX_j^\eta(\sigma)> a\}}\exp\Big(\sum_{j=1}^M\kappa_jX_j^\eta(\sigma)\Big)+\one_{\{\sum_{j=1}^M\kappa_jX_j^\eta(\sigma)< -a\}}\Big\rrangle\\
&\stackrefp{exp_inequality}{\leq} \e^a\Big(\E\Big\llangle\one_{\{\sum_{j=1}^M\kappa_jX_j^\eta(\sigma)> a\}}\Big\rrangle
\E\Big\llangle\exp\Big(2\sum_{j=1}^M\kappa_jX_j^\eta(\sigma)\Big)\Big\rrangle\Big)^{1/2
} + \e^a\E\Big\llangle\one_{\{\sum_{j=1}^M\kappa_jX_j^\eta(\sigma)<- a\}}\Big\rrangle\\
&\stackrefp{exp_inequality}{\leq} \e^a\exp(-a^2/4C)\exp(C) + \e^a\exp(-a^2/2C), \raisetag{4.5\baselineskip}
}
where in the last line we have again used the fact $\sum_{j=1}^M\kappa_jX_j^\eta(\sigma)$ is Gaussian with mean zero and variance $C$.
At last, given any $\eps>0$, we choose $a$ sufficiently large that \eqref{first_trunction} and \eqref{second_truncation} combine to give
\eeq{ \label{after_truncations}
|\E\log Z - \E\log^{(a)}Z^{(a)}| \leq \eps/2.
}

Now recall our notation that $\sigma^1,\sigma^2,\dots$ are independent samples from the Gibbs measure $G$.
Also let $\eta^1,\eta^2,\dots$ be independent copies of $\eta$.
We then have the following identity for any integer $r\geq1$ (simply by definition of $\llangle\cdot\rrangle$ as an average):
\eq{
\llangle f(\kappa,\eta,\sigma)\rrangle^r
= \prod_{\ell=1}^r\int_{\mathbf{T}_M}\E_{\eta^\ell} \big\langle f(\kappa,\eta^\ell,\sigma^\ell)\big\rangle \  \vc\tau_M(\dd\kappa),
}
provided both sides make sense.
Applying this identity to the function of interest, we obtain
\eq{
(Z^{(a)})^r = \prod_{\ell=1}^r\int_{\mathbf{T}_M}\E_{\eta^\ell}\Big\langle\exp^{(a)}\Big(\sum_{j=1}^M\kappa_jX_j^{\eta_\ell}(\sigma^\ell)\Big)\Big\rangle \ \vc\tau_M(\dd\kappa).
}
Conditional on $G$, the $\E_X$-expectation of the right-hand side is just some non-random function $\vphi_{M,r}$ of the covariance matrix $\mathbf{K} = (\mathbf{K}_{(j,\ell),(j',\ell')})$ for $(X_j^{\eta_\ell}(\sigma^\ell))_{j\in[M],\ell\in[r]}$:
\eq{
\E_X[(Z^{(a)})^r] = \vphi_{M,r}(\mathbf{K}).
}
Since $\exp^{(a)}$ is bounded and continuous, so too is $\vphi_{M,r}$, as weak convergence of Gaussian distributions is equivalent to convergence of their covariance matrices.
Moreover, since nonnegative definite matrices form a closed subset of all symmetric matrices, the Tietze–Urysohn–Brouwer extension theorem allows us to extend $\vphi_{M,r}$ continuously to this larger space.
Now, by \eqref{A3_cov} we have $\mathbf{K}_{(j,\ell),(j',\ell')} = \one_{\{j=j'\}}\xi^s(\vc\RR_{\ell,\ell'})$ whenever $j\in\JJ^s$.
Consequently, $\mathbf{K}$ is a continuous function of the array $\vc\RR^r = (\vc\RR_{\ell,\ell'})_{\ell,\ell'\in[r]}$.
By composing this function with $\vphi_{M,r}$, we obtain a bounded and continuous function $\phi_{M,r}$ (defined on all symmetric $r\times r$ vector arrays) such that
\eq{
\E_X[(Z^{(a)})^r] = \phi_{M,r}\big(\vc\RR^r).
}

To complete the proof, we appeal to Stone--Weierstrass to find a polynomial $\sum_{r=1}^n \alpha_r x^r$ which is within $\eps/2$ of $\log x$ for all $x\in[\e^{-a},\e^a]$.
Since $Z^{(a)}$ always belongs to this interval, we have the following approximation:
\eeq{ \label{stone_approx}
\Big|\E_X\log^{(a)}Z^{(a)} - \sum_{r=1}^n \alpha_r\phi_{M,r}(\vc\RR^r)\Big| \leq \eps/2.
}
Once we average over the realizations of $G$, we obtain the following function of $\vc\LL^n=\mathsf{Law}(\vc\RR^n;G)$:
\eeq{ \label{Psi_eps_def}
\Pi_{M,1}^{(\eps)}(\vc\LL^n) \coloneqq \sum_{r=1}^n\alpha_r\int \phi_{M,r}(\vc\RR^r)\ \vc\LL^n(\dd\vc\RR)
=\sum_{r=1}^n\alpha_r\E\langle\phi_{M,r}(\vc\RR^r)\rangle.
}
This is the map claimed by the proposition.
Indeed, since each $\phi_{M,r}$ is bounded and continuous, $\Pi_{M,1}^{(\eps)}$ is continuous with respect to weak convergence.
And putting together \eqref{after_truncations} and \eqref{stone_approx}, we have
\eq{
|\E\log Z - \Pi_{M,1}^{(\eps)}(\vc\LL^n)| \leq \eps.
}
By the exact same argument, we can obtain the analogous approximating function for $\Pi_{M,2}$.
In that case, the relevant function $f$ is simply $\sqrt{M}Y^\eta(\sigma)$ (no dependence on $\kappa$), and the constant $C$ appearing in \eqref{constant_C} is $M\theta(\vc 1)$.
\end{proof}

%The clunky definition of $\Pi_M$ given above is just broad enough to include the overlap distributions we will encounter at finite $N$ as well those induced by the Ruelle probability cascades.
%Once $N$ is sent to infinity, however, the limiting distributions will satisfy the Ghirlanda--Guerra identities. Because of the synchronization theory developed in \cite{panchenko15}, these identities enable us to realize associated the overlap arrays in a special way; that is, the map $\vc R$ in (A1) takes on a very particular form.
%Nevertheless, technical issues surround the verification of (A3), and so in the next section we will restrict $\Pi_M$ to just those overlap distributions which both satisfy the Ghirlanda--Guerra identities and take finitely many values.
%These are exactly the overlap distributions generated by Ruelle probability cascades.

\subsection{Review of Poisson--Dirichlet cascades and Ruelle probability cascades} \label{cascade_review}
Let us adopt the conventions that $\N=\{1,2,\dots\}$ and $\N^0 = \{\varnothing\}$.
For each sequence of the form
\eeq{ \label{m_sequence}
0 = m_0 < m_1 < \cdots < m_{k-1} < m_k = 1,
}
there is a random probability measure on $\N^{k-1}$, called a \textit{Poisson--Dirichlet cascade}, which satisfies certain properties described below.
Since $\N^{k-1}$ is countable, the cascade is naturally identified with the random weights $(v_\alpha)_{\alpha\in\N^{k-1}}$ constituting its probability mass function.
A precise construction can be found in \cite[Sec.~14.2]{talagrand11b}; here we describe just three properties needed in the sequel. \\

\subsubsection{Overlap distribution}
For $\alpha = (\alpha_1,\dots,\alpha_{k-1})\in\N^{k-1}$, let $p(\alpha)$ denote the set of truncations of $\alpha$:
\eq{
p(\alpha) = \{\varnothing,(\alpha_1),(\alpha_1,\alpha_2),\dots,(\alpha_1,\dots,\alpha_{k-1})\}.
}
The similarity of two vectors $\alpha,\alpha'\in \N^{k-1}$ is measured by how many elements are shared by $p(\alpha)$ and $p(\alpha')$.
That is, if $\alpha = (\alpha_1,\dots,\alpha_{k-1})$ and $\alpha' = (\alpha_1',\dots,\alpha_{k-1}')$, then define the overlap
\eeq{ \label{overlap_notation}
r(\alpha,\alpha') \coloneqq |p(\alpha)\cap p(\alpha')| =\begin{cases} \inf\{r:\,  \alpha_r \neq \alpha_r'\} &\text{if $\alpha\neq\alpha'$}, \\
k &\text{if $\alpha=\alpha'$}.
\end{cases}
}
The most basic property of the Poisson--Dirichlet cascade is that if $\alpha^1$ and $\alpha^2$ are independently sampled according to the weights $(v_\alpha)_{\alpha\in\N^{k-1}}$, then $r(\alpha^1,\alpha^2)$ follows a distribution encoded by \eqref{m_sequence}.
Namely, if $\langle\cdot\rangle$ denotes expectation over these independent samples, and $\E(\cdot)$ denotes expectation over realizations of the cascade, then by \cite[Prop.~14.3.3]{talagrand11b} we have
\eeq{ \label{correct_overlap_distribution}
\E\langle\one_{\{r(\alpha^1,\alpha^2)=r\}}\rangle = m_r - m_{r-1}, \quad 1\leq r\leq k.
}
%This fact can actually be generalized as follows.

\subsubsection{Expectations of hierarchical functions}
Let $(z_\beta)_{\beta\in\N^0\cup\cdots\cup\N^{k-1}}$ be i.i.d.~random variables taking values in some metric space $T$.
Given any function $F\colon T^{k}\to\R$, we define (using a slight abuse of notation) its hierarchical form:
\eeq{ \label{zs_with_dependence}
F(\alpha) \coloneqq F(z_\varnothing,z_{(\alpha_1)},z_{(\alpha_1,\alpha_2)},\dots,z_{(\alpha_1,\dots,\alpha_{k-1})}), \quad \alpha\in\N^{k-1}.
}
Therefore, $F(\alpha)$ and $F(\alpha')$ are statistically dependent only via the variables $(z_\beta)_{\beta\in p(\alpha)\cap p(\alpha')}$.
We now describe a way of computing expectations of the form $\E\log\langle \exp F(\alpha)\rangle$, using only a single random variable for each level of overlap.
First define
\eeq{ \label{just_zs}
F_{k} \coloneqq F(z_0,z_1,\dots,z_{k-1}),
}
where $z_0,\dots,z_{k-1}$ are i.i.d.~$T$-valued random variables as before.
Now inductively define
\eeq{ \label{following_just_zs}
F_r \coloneqq \frac{1}{m_r}\log \E_r\exp(m_r F_{r+1}) \quad \text{for $r\in[k-1]$}, \qquad
F_0 \coloneqq \E_0(F_1),
}
where $\E_r(\cdot)$ denotes expectation over just $z_r,\dots,z_{k-1}$.
By \cite[Thm.~14.2.1]{talagrand11b}, we then have
\eeq{ \label{magical_iteration}
\E\log\langle\exp F(\alpha)\rangle = F_0. 
}
As a matter of interpretation, the identity \eqref{magical_iteration} has converted the hierarchical structure of the random variables $(z_\beta)_{\beta\in\N^0\cup\N^1\cup\cdots\cup\N^{k-1}}$ into an iterative procedure. \\

\subsubsection{Tilting by hierarchical functions}
The last property we need concerns the Poisson--Dirichlet cascade tilted by a function $F$ of the form \eqref{zs_with_dependence}.
That is, given any other function $U$ of the same form, we define
\eeq{ \label{tilted_Gibbs}
\langle U(\alpha^1)U(\alpha^2)\rangle_F \coloneqq \frac{\langle U(\alpha^1)U(\alpha^2) \exp F(\alpha^1)\exp F(\alpha^2)\rangle}{\langle\exp F(\alpha)\rangle^2},
}
where $\alpha^1$ and $\alpha^2$ are independent samples from the Poisson--Dirichlet cascade.
With $F_r$ as in \eqref{following_just_zs}, define
\eq{
W_r \coloneqq \exp(m_r(F_{r+1}-F_r)), \quad r\in[k-1].
}
With $U_k$ as in \eqref{just_zs} for the function $U$, \cite[Prop.~14.3.2]{talagrand11b} gives the following identity for any $r\in[k]$:
\eeq{ \label{big_generalization}
\E\langle\one_{\{r(\alpha^1,\alpha^2)=r\}}U(\alpha^1)U(\alpha^2)\rangle_F = (m_r-m_{r-1})\E\big[W_1\cdots W_{r-1}(\E_r[W_r\cdots W_{k-1}U_k])^2\big].
}
%where the case $r=k$ is interpreted as
%\eq{
%\E\langle\one_{\{\alpha^1=\alpha^2\}}U(\alpha^1)^2\rangle = \E(W_1\cdots W_{k-1}U_k^2).
%}
Note that $\E_r(W_r) = 1$ by \eqref{following_just_zs},
and that $W_{r}$ has no dependence on $z_{r'}$ for $r'>r$.
Consequently, for any $r'>r$ we have
\eq{ %\label{conditional_argument}
\E_r[W_r\cdots W_{r'}]
= \E_r\big[W_r\cdots W_{r'-1}\E_{r'}(W_{r'})\big]
=\E_r[W_r\cdots W_{r'-1}]
= \cdots = 1.
}
Therefore, when $U\equiv 1$, \eqref{big_generalization} provides a generalization of \eqref{correct_overlap_distribution}:
\eeq{ \label{still_have}
\E\langle\one_{\{r(\alpha^1,\alpha^2)=r\}}\rangle_F = m_r - m_{r-1}, \quad r\in[k].
}

\subsubsection{Ruelle probability cascades}
Notice that so far we have only dealt with the sequence of weights $(m_r)_{0\leq r\leq k}$ from \eqref{m_sequence}.
When one also defines a sequence of locations
\eeq{ \label{q_sequence}
0 = q_0 \leq q_1 \leq \cdots \leq q_k \leq q_{k+1}=1,
}
then one obtains a measure
\eeq{ \label{zeta_def}
\zeta = \sum_{r=1}^{k}(m_r-m_{r-1})\delta_{q_r}.
}
We will now construct a random measure $\GG = \GG_{m;q_1,\dots,q_k}$ on any separable, infinite-dimensional Hilbert space such that if $\sigma^1$ and $\sigma^2$ are two independent samples from this measure, then $\sigma^1\cdot\sigma^2$ is $\zeta$-distributed (in the averaged sense of \eqref{overlap_for_rpc} given below).
Let $(e_\beta)_{\beta\in\N^0\cup\cdots\cup\N^{k-1}}$ be a collection of orthonormal vectors in the Hilbert space, and for each $\alpha\in\N^{k-1}$ define
\eq{ %\label{h_alpha_def}
h_\alpha \coloneqq \sum_{\beta\in p(\alpha)}e_\beta\sqrt{q_{|\beta|+1}-q_{|\beta|}},
}
where $|\beta|=r$ for $\beta\in\N^r$.
Notice that 
\eeq{ \label{overlap_formula_rpc}
h_{\alpha}\cdot h_{\alpha'} = q_{r(\alpha,\alpha')}.
}
Therefore, if $h_\alpha$ is chosen with probability $v_\alpha$ according to the Poisson--Dirichlet cascade, then two independently chosen $h_{\alpha^1}$ and $h_{\alpha^2}$ will yield the following analogue of \eqref{correct_overlap_distribution}:
\eeq{ \label{overlap_for_rpc}
\E\langle \one_{\{h_{\alpha^1}\cdot h_{\alpha^2}\in[0,q]\}}\rangle = \zeta\big([0,q]\big) \quad \text{for any $q\in[0,1]$}.
}
We thus take $\GG_{m;q_1,\dots,q_k}$ to be a purely atomic measure with 
\eeq{ \label{rpc_def}
\GG_{m;q_1,\dots,q_k}(\{h_\alpha\}) = v_\alpha, \quad \text{where $(v_\alpha)_{\alpha\in\N^{k-1}}$ is the Poisson--Dirichlet cascade for \eqref{m_sequence}}.
}
This measure $\GG_{m;q_1,\dots,q_k}$ is called a \textit{Ruelle probability cascade}.

\subsection{Applying the Parisi prelimiting functional to Ruelle probability cascades} \label{restrict_prelimit}
Now we return to our consideration of the function $\Pi_M$ from Section \ref{prelimit_section}.
Here we study the outcome of applying $\Pi_M$ to overlap distributions obtained from Ruelle probability cascades.

Let $\zeta$ be any measure on $[0,1]$ with finite support;
in other words, $\zeta$ is of the form \eqref{zeta_def} for
some sequences $(m_r)_{1\leq r\leq k}$ and $(q_r)_{1\leq r\leq k}$ of the form \eqref{m_sequence} and \eqref{q_sequence}.
Let $\GG_\zeta$ be the associated Ruelle probability cascade following \eqref{rpc_def}, and
let $\Phi$ be any $\vc\lambda$-admissible map.
As a shorthand, we will write
\eeq{ \label{bold_q_sequence}
\vc q_r = \Phi(q_r), \quad 0\leq r\leq k+1,
}
since all quantities of interest will depend on $\Phi$ only through the values of $\vc q_1,\dots,\vc q_k$.
Using the map $\vc R\colon (h_\alpha,h_{\alpha'}) \mapsto \Phi(h_\alpha\cdot h_{\alpha'})$, we consider the law
\eeq{ \label{rpc_relevant_law}
\vc\LL(\zeta,\Phi) \coloneqq \mathsf{Law}(\vc\RR;\GG_{m;q_1,\dots,q_k})
}
from Section \ref{prelimit_section_3}.
Using the notation from Theorem \ref{ds_rep}, we equivalently have 
\eeq{ \label{equivalent_gibbs}
\vc\LL(\zeta,\Phi) = \mathsf{Law}(\RR;\GG_{m;q_1,\dots,q_k})\circ \Phi^{-1}.
}
Implicit in our notation is that the right-hand side of \eqref{rpc_relevant_law} is completely determined by $\zeta$ in \eqref{zeta_def}.
That is, even if a different collection of $m$'s and $q$'s give the same measure in \eqref{zeta_def}, $\mathsf{Law}(\RR;\GG_{m,q_1,\dots,q_k})$ would remain the same.
This is a consequence of Theorem \ref{representation_thm}, since overlap distributions arising from the cascades do indeed satisfy the Ghirlanda--Guerra identities (see \cite[Thm.~15.2.1]{talagrand11b} or \cite[Thm.~2.10]{panchenko13a}).
Let us make a formal statement to which we can refer later.

\begin{cor} \label{zeta_to_law}
For any fixed $\Phi$, the map $\zeta\mapsto\vc\LL(\zeta,\Phi)$ is well-defined and continuous with respect to weak convergence.
\end{cor}

In order to evaluate $\Pi_M$ at $\vc\LL(\zeta,\Phi)$, we still need to check Assumption \ref{map_assumption}, and for this we simply construct the desired Gaussian processes.
Let $(\eta_{j,\beta})_{j\in[M],\beta\in\N^0\cup\cdots\cup\N^{k-1}}$ and $(\eta_\beta)_{\beta\in\N^0\cup\cdots\cup\N^{k-1}}$ be independent standard normal random variables, and set
\begin{subequations}
\label{XY_def_rpc}
\eeqs{
X_j(h_\alpha) &= \sum_{\beta\in p(\alpha)}\eta_{j,\beta}\sqrt{\xi^s(\vc q_{|\beta|+1})-\one_{\{|\beta|>0\}}\xi^s(\vc q_{|\beta|})} \quad \text{for $j\in\JJ^s$}, \label{X_RPC}\\
Y(h_\alpha) &= \sum_{\beta\in p(\alpha)}\eta_{\beta}\sqrt{\theta(\vc q_{|\beta|+1})-\theta(\vc q_{|\beta|})}. \label{Y_RPC}
}
\end{subequations}
The desired covariance identities \eqref{A3_cov} trivially follow.
Therefore, we can specialize \eqref{Psi_def} to the present setting by defining
\eeq{ \label{restriction_def}
\PPP_{M,i}(\zeta,\Phi) &\coloneqq \Pi_{M,i}(\vc\LL(\zeta,\Phi)) \quad \text{for $i\in\{1,2\}$}, \quad \text{and} \\
\PPP_M(\zeta,\Phi) &\coloneqq \PPP_{M,1}(\zeta,\Phi)-\PPP_{M,2}(\zeta,\Phi).
}
The following statement is a precursor to the Lipschitz continuity claimed in Theorem \ref{lipschitz_continuity}.

\begin{prop} \label{lipschitz_continuity_restricted}
For any $\vc\lambda$-admissible pairs $(\zeta,\Phi)$ and $(\wt\zeta,\wt{\Phi})$ such that $\zeta$ and $\wt\zeta$ have finite support, we have
\eeq{ \label{lipschitz_continuity_restricted_eq}
\frac{|\PPP_M(\zeta,\Phi)-\PPP_M(\wt\zeta,\wt{\Phi})|}{M} \leq \frac{C_*}{2}\bigg(1+\sum_{s\in\SSS}\Big|\frac{M^s}{M}-\lambda^s\Big|\bigg)\DD\big((\zeta,\Phi),(\wt\zeta,\wt{\Phi})\big),
}
where $C_*$ is given in \eqref{lipschitz_continuity_eq}.
\end{prop}

Before proving Proposition \ref{lipschitz_continuity_restricted}, let us make the following preliminary calculation, which explains how the functional $(\zeta,\Phi)\mapsto \PPP(\zeta,\Phi)$ emerges from the cavity method, and how Theorem \ref{lipschitz_continuity} will follow from Proposition \ref{lipschitz_continuity_restricted}.

\begin{prop} \label{explain_appearance}
Assume that $M^s/M\to\lambda^s$ as $M\to\infty$, for each $s\in\SSS$.
Then for any $\vc\lambda$-admissible pair $(\zeta,\Phi)$ such that $\zeta$ has finite support, we have
\eeq{ \label{explain_appearance_eq}
\lim_{M\to\infty} \frac{\PPP_M(\zeta,\Phi)}{M} = \PPP(\zeta,\Phi).
}
\end{prop}

\begin{proof}
We will use the shorthands 
\eq{
u_r^s=\one_{\{r>0\}}\xi^s(\vc q_r)=\one_{\{r>0\}}\xi^s(\Phi(q_r)) \quad \text{and} \quad
w_r = \theta(\vc q_r)=\theta(\Phi(q_r)).
}

First we compute the right-hand side of \eqref{explain_appearance_eq} by recalling the definition of $\PPP$ from \eqref{parisi_functional_def}. %quantities defined above \eqref{parisi_functional_def}.
Since $\zeta\big([0,u]\big)=m_{r}$ for $u\in[q_{r},q_{r+1})$, the quantity from \eqref{ds_def} is equal to
\eeq{ \label{ds_identity}
d^s(q) &= \int_q^1\zeta\big([0,u]\big)(\xi^s\circ\Phi)'(u)\ \dd u \\
&= m_{r}[u^s_{r+1}-\xi^s(\Phi(q))]+\sum_{r'=r+1}^k m_{r'}(u^s_{r'+1}-u^s_{r'}) \quad \text{for all $q\in[q_{r},q_{r+1}]$}.
}
When $q = q_r$, we will use the notation $d_r^s \coloneqq d^s(q_r)$.
So
\eq{
d_{k+1}^s=0 \qquad \text{and} \qquad d^s_r =  \sum_{r'=r}^{k}m_{r'}(u^s_{r'+1}-u^s_{r'}) \quad \text{for $r\in[k]$.}
}
Since $\zeta\big([0,u]\big) = 0$ for all $u<q_1$, we have
\eeq{ \label{constant_up2_q1}
d^s(q) = d^s_1 \quad \text{for all $q\in[0,q_1]$}.
}
Now consider the first integral in \eqref{A_def}.
In light of \eqref{constant_up2_q1}, we have
\eq{
\int_0^{q_1} \frac{(\xi^s\circ \Phi)'(q)}{b^s-d^s(q)}\ \dd q
=\int_0^{q_1} \frac{(\xi^s\circ \Phi)'(q)}{b^s-d^s_1}\ \dd q
&= \frac{u^s_1 - \xi^s(\vc 0)}{b^s-d^s_1}.
}
Meanwhile, on the interval $[q_{r},q_{r+1}]$ with $1 \leq r\leq k$, from \eqref{ds_identity} we have
\eq{
\int_{q_r}^{q_{r+1}}\frac{(\xi^s\circ\Phi)'(q)}{b^s-d^s(q)}\ \dd q
&= \int_{q_{r}}^{q_{r+1}}\frac{(\xi^s\circ\Phi)'(q)}{b^s-d^s_{r+1} - m_{r}[u^s_{r+1}-\xi^s(\Phi(q))]}\ \dd q \\
&= \int^{u^s_{r+1}}_{u^s_r}\frac{1}{b^s-d^s_{r+1} - m_{r}[u^s_{r+1}-u]}\ \dd u
= \frac{1}{m_r}\log\frac{b^s-d^s_{r+1}}{b^s-d^s_{r}}.
}
The last integral to compute is
\eq{
\int_0^1 \zeta\big([0,q]\big)(\theta\circ\Phi)'(q)\ \dd q
= \sum_{r=1}^k m_r(w_{r+1}-w_r).
}
Putting together these computations and recalling the definition of $A$ from \eqref{A_def}, we have
\eeq{ \label{A_def_0}
A(\zeta,\Phi,\vc b)
&= \sum_{s\in\SSS} \frac{\lambda^s}{2}\bigg(b^s - 1 - \log b^s + \frac{ u_1^s}{b^s-d^s_1}+ \sum_{r=1}^{k}\frac{1}{m_r}\log\frac{b^s-d^s_{r+1}}{b^s-d^s_r}\bigg) \\
&\phantom{=}-\frac{1}{2}\sum_{r=1}^{k}m_r(w_{r+1}-w_r).
}
Finally, by definition we have $\PPP(\zeta,\Phi)=\inf_{\vc b}A(\zeta,\Phi,\vc b)$, where the infimum is over $\vc b$ such that $b^s>d^s(0)$ for each $s\in\SSS$.
Because of \eqref{constant_up2_q1}, this condition is equivalent to $b^s>d^s_1$.

Now we compute the left-hand side of \eqref{explain_appearance_eq}.
Notice from \eqref{overlap_formula_rpc} that no matter the choice of $\alpha$, we have $h_\alpha\cdot h_\alpha = q_k$.
That is, \eqref{constant_overlap_for_G} holds with $\vc q_* = \Phi(q_k)$, which we have been calling $\vc q_k$.
Therefore, instead of referring to the quantities from \eqref{defining_Psi_parts}, we can start from their equivalent forms in \eqref{redefining_Psi_parts}.
With the processes from \eqref{X_RPC}, the quantity $\PPP_{M,1}(\zeta,\Phi)$ from \eqref{psi_1_redef} is equal to
\eeq{ \label{psi_1}
\E\log\bigg\langle \int_{\mathbf{T}_M}\exp\Big(\sum_{s\in\SSS}\sum_{j\in\JJ^s}\kappa_j\sum_{\beta\in p(\alpha)}\eta_{j,\beta}\sqrt{u^s_{|\beta|+1}-u^s_{|\beta|}}\, \Big)\ \vc\tau_{M}(\dd\kappa)\bigg\rangle
+\sum_{s\in\SSS}\frac{M^s}{2}(u^s_{k+1}-u^s_k).
}
Meanwhile, with the processes from \eqref{Y_RPC}, the quantity $\PPP_{M,2}(\zeta,\Phi)$ from \eqref{psi_2_redef} is equal to
\eeq{ \label{psi_2}
\E\log\Big\langle\exp\Big(\sqrt{M}\sum_{\beta\in p(\alpha)}\eta_{\beta}\sqrt{w_{|\beta|+1}-w_{|\beta|}}\Big)\Big\rangle
+ \frac{M}{2}(w_{k+1}-w_k).
}
Each of these quantities can be rewritten using the formula \eqref{magical_iteration}.

Let us first consider $\PPP_{M,2}(\zeta,\Phi)$, as the computation is simpler and explicit in this case.
The $\eta_\beta$'s in \eqref{psi_2} play the role of the $z_\beta$'s in \eqref{zs_with_dependence}.
So let us define i.i.d.~standard normal random variables $(\eta_{r})_{0\leq r\leq k-1}$ to play the role of the $z_r$'s in \eqref{just_zs}.
That is, we begin with
\eq{
F_k = \sqrt{M}\sum_{r=0}^{k-1}\eta_{r}\sqrt{w_{r+1}-w_{r}} + \frac{M}{2}(w_{k+1}-w_k),
}
and then apply the formula \eqref{following_just_zs} inductively to arrive at $F_0$, which is equal to \eqref{psi_2} by \eqref{magical_iteration}.
Using the identities $\E\exp(c\eta_r) = \exp(c^2/2)$ and $\E(\eta_0)=0$, it is easy to verify that the result of this induction is
\eeq{ \label{psi_2_after_magical}
%F_0 = \frac{M}{2}\sum_{r=1}^{k-1}m_r(w_{r+1}-w_r), \quad \text{and so} \quad
\PPP_{M,2}(\zeta,\Phi) =\frac{M}{2}\sum_{r=1}^{k}m_r(w_{r+1}-w_r).
}
Next we consider the more complicated quantity $\PPP_{M,1}(\zeta,\Phi)$.
 Now the random vectors $(\eta_{j,\beta})_{j\in[M]}$ in \eqref{psi_1} play the role of the $z_\beta$'s in \eqref{zs_with_dependence}.
So let us define independent standard normal random variables $(\eta_{j,r})_{j\in[M], 0\leq r\leq k-1}$ to play the role of the $z_r$'s in \eqref{just_zs}.
 That is, the quantity in \eqref{just_zs} is given by
 \eq{
 F_k
 &\stackrefp{cavity_space_def}{=}\log\int_{\mathbf{T}_M}
 \exp\Big(\sum_{s\in\SSS}\sum_{j\in\JJ^s}\kappa_j\sum_{r=0}^{k-1}\eta_{j,r}\sqrt{u^s_{r+1}-u^s_{r}}\,\Big)\ \vc\tau_{M}(\dd\kappa) + \sum_{s\in\SSS}\frac{M^s}{2}(u^s_{k+1}-u^s_k)\\
 &\stackref{cavity_space_def}{=}\sum_{s\in\SSS}\bigg[\log \int_{ S_{M^s}}\exp\Big(\sum_{j\in\JJ^s}\kappa_j\sum_{r=0}^{k-1}\eta_{j,r}\sqrt{u^s_{r+1}-u^s_{r}}\,\Big)\ \mu_{M^s}(\dd\kappa)
 +\frac{M^s}{2}(u^s_{k+1}-u^s_k)\bigg],
 }
 and then $\PPP_{M,1}(\zeta,\Phi)$ is equal to $F_0$ as obtained inductively from \eqref{following_just_zs}.
But notice that we have written $F_k$ as a sum of $|\SSS|$ independent variables of the form
\eeq{ \label{single_species_before_iteration}
F_{k}^s \coloneqq \log \int_{ S_{M^s}} \exp\Big(\sum_{j\in\JJ^s}\kappa_j\sum_{r=0}^{k-1} \eta_{j,r}\sqrt{u_{r+1}^s-u_r^s}\, \Big)\ \mu_{M^s}(\dd{\kappa})
+\frac{M^s}{2}(u_{k+1}^s-u_k^s).
}
Therefore, applying \eqref{following_just_zs} to $F_k$ is equivalent to applying \eqref{following_just_zs} to each $F_k^s$ and then adding the results. 
That is, we have $F_r=\sum_{s\in\SSS}F_r^s$ by downward induction on $r$, where $F_r^s$ is defined from $F_{r+1}^s$ as in \eqref{following_just_zs}.
We write the final quantity $F_0^s$ as $\PPP_{M,1}^s(\zeta,\Phi)$ so that
%\begin{subequations}
%\label{after_ldp}
\eq{ %\label{after_ldp_1}
\PPP_{M,1}(\zeta,\Phi) =\sum_{s\in\SSS}\PPP_{M,1}^s(\zeta,\Phi).
}
While $\PPP_{M,1}^s(\zeta,\Phi)$ does not have an explicit expression as in \eqref{psi_2_after_magical}, we can invoke the large deviations calculation by Talagrand \cite[Prop.~3.1]{talagrand06b}, which says
\eeq{ \label{after_ldp_2}
\lim_{M^s\to\infty} \frac{\PPP_{M,1}^s(\zeta,\Phi)}{M^s}
&= \frac{1}{2}\inf_{b^s>d_{1}^s}\bigg[b^s-1-\log b^s + \frac{u^s_1}{b^s-d_1^s}+\sum_{r=1}^{k}\frac{1}{m_r}\log\frac{b^s-d_{r+1}^s}{b^s-d_{r}^s}\bigg].
}

\begin{remark} \label{even_spin_remark}
The identity \eqref{after_ldp_2} is most readily seen from (3.31) and (3.48) in \cite{talagrand06b}.
%In that paper, it was always the case that $u_0^s = 0$, since only even $p$-spin interactions were allowed and so the first-order derivative of $\xi$ vanished at the origin.
Furthermore, one sees from the same places in \cite{talagrand06b} that the presence of an external field $h_s$ adds a term of the form $h_s^2/(b^s-d^s_1)$ to the right-hand side of \eqref{after_ldp_2}.
\end{remark}

%\end{subequations}
Now sum the right-hand side of \eqref{after_ldp_2} over $s\in\SSS$ and compare with the first line of \eqref{A_def_0}.
Since the optimization in \eqref{after_ldp_2} is decoupled over $s\in\SSS$, the sum of infima is the infimum of the sum.
With the assumption that $M^s/M\to\lambda^s$ as $M\to\infty$, we thus have
\eeq{ \label{after_ldp}
\lim_{M\to\infty}\frac{\PPP_{M,1}(\zeta,\Phi)}{M} =
\inf_{\vc b}\sum_{s\in\SSS} \frac{\lambda^s}{2}\bigg(b^s - 1 - \log b^s + \frac{u_1^s}{b^s-d^s_1}+ \sum_{r=1}^{k}\frac{1}{m_r}\log\frac{b^s-d^s_{r+1}}{b^s-d^s_r}\bigg).
}
Finally, to account for the second line in \eqref{A_def_0}, subtract the quantity $\PPP_{M,2}(\zeta,\Phi)/M$ appearing in \eqref{psi_2_after_magical}, and we obtain \eqref{explain_appearance_eq}.
%
%Notice that when we multiply the objective function on the right-hand side of \eqref{after_ldp_2} by $\lambda^s$, then sum over $s\in\SSS$, and finally subtract $\PPP_{M,2}(\zeta,\Phi)$ from \eqref{psi_2_after_magical}, we obtain exactly the quantity $A(\zeta,\Phi,\vc b)$ from \eqref{A_def_0}.
%Furthermore, the optimization in \eqref{after_ldp_2} is decoupled over $s\in\SSS$, and $d_1^s$ is simply $d^s(0)$ in the notation of \eqref{ds_def}.
%Therefore, by using the assumption that $M^s/M\to\lambda^s$ as $M\to\infty$, we obtain \eqref{explain_appearance_eq}.
\end{proof}

%Now we adapt the approach of \cite[Lem.~14.11.1]{talagrand11b} toward proving Proposition \ref{lipschitz_continuity_restricted}, with suitable additions for the present setting. 
We saw in the proof of Proposition \ref{explain_appearance} that we can write $\PPP_{M,1}(\zeta,\Phi)$ as a function of the sequences $m = (m_r)_{0\leq r\leq k}$ and $\vc q = (\vc q_r)_{0\leq r\leq k+1}$ from \eqref{m_sequence} and \eqref{bold_q_sequence}.
That is, in a slight abuse of notation,
\eq{
\PPP_M(\zeta,\Phi) 
= \PPP_M(m;\vc q_1,\dots,\vc q_k) 
&= \sum_{s\in\SSS} \PPP_{M,1}^s(m;u_1^s,\dots,u_k^s)
- \PPP_{M,2}(m;w_1,\dots,w_k),
}
where $u_r^s= \one_{\{r>0\}}\xi^s(\vc q_r)$ and $w_r = \theta(\vc q_r)$.
Notice that we have omitted $\vc q_0=\vc 0$ and $\vc q_{k+1}=\vc1$, as these values are constant.
Our next observation is that adding duplicate copies of any $\vc q_r$ does not change the value of
the functions seen above.
This will ultimately allow us, in the proof of \eqref{lipschitz_continuity_restricted_eq}, 
to assume $\zeta$ and $\wt\zeta$ arise from the same $m$ sequence.

%Also observe that because of Corollary \ref{zeta_to_law}, whenever $\wt m = (\wt m_r)_{1\leq r\leq \wt k}$ and $(\wt{\vc q}_r)_{1\leq r\leq \wt k}$ yield the same measure as \eqref{zeta_def}, we will have
%\eq{
%\PPP_M(\wt m;\wt{\vc q}_1,\dots,\wt{\vc q}_k) = %\PPP_M(m;\vc q_1,\dots,\vc q_k)
%}

\begin{lemma} \label{m_lemma}
Consider any sequence of integers $0=n_0<n_1 < n_2 < \cdots < n_k$.
Let
$0 = \wt m_0 < \wt m_1 < \dots < \wt m_{n_k} = 1$
be such that $\wt m_{n_r} = m_r$ for each $r\in[k]$.
We then have %, for any $s\in\SSS$,
%\eq{
%\PPP_{M,1}^s(m;u_1^s,u_2^s,\dots,u_k^s)
%= \PPP_{M.1}(\wt m;\underbrace{u_1^s,\dots,u_1^s}_{\text{$n_1$ }},\underbrace{u_2^s,\dots, u_2^s}_{\text{$n_2-n_1$}},\cdots,\underbrace{u_k^s ,\dots,u_k^s}_{\text{$n_k-n_{k-1}$}}),
%}
%as well as
%\eq{
%\PPP_{M,2}(m;w_1,w_2,\dots,w_k)
%= \PPP_{M.1}(\wt m;\underbrace{w_1,\dots,w_1}_{\text{$n_1$ }},\underbrace{w_2,\dots, w_2}_{\text{$n_2-n_1$}},\cdots,\underbrace{w_k ,\dots,w_k}_{\text{$n_k-n_{k-1}$}}).
%}
\eeq{ \label{m_lemma_eq}
\PPP_M(m;\vc q_1,\dots,\vc q_k)
= \PPP_M(\wt m;\underbrace{\vc q_1,\dots,\vc q_1}_{\text{$n_1$ }},\underbrace{\vc q_2,\dots,\vc q_2}_{\text{$n_2-n_1$}},\cdots,\underbrace{\vc q_k,\dots,\vc q_k}_{\text{$n_k-n_{k-1}$}}).
}
\end{lemma}

\begin{proof}
It is not hard to determine \eqref{m_lemma_eq} directly from definition chasing, but it is even easier to simply appeal to Corollary \ref{zeta_to_law}.
Indeed, the right-hand side of \eqref{m_lemma_eq} is equal to $\PPP_M(\wt\zeta,\Phi)$, where
\eq{
\wt\zeta = \sum_{r=1}^k\sum_{n=n_{r-1}+1}^{n_r} (\wt m_n-\wt m_{n-1})\delta_{q_r}
= \sum_{r=1}^k(m_r-m_{r-1})\delta_{q_r} = \zeta.
}
Hence $\vc\LL(\zeta,\Phi)=\vc\LL(\wt\zeta,\Phi)$, and so by definition \eqref{restriction_def}, we are done.
%
%
%First consider $\PPP_{M,2}$.
%Applying summation by parts to \eqref{psi_2_after_magical}, we obtain
%\begin{align}
%\PPP_{M,2}(m;w_1,w_2,\dots,w_k) 
%&= \frac{M}{2}\Big[\theta(\vc 1)-\sum_{r=1}^{k}(m_{r}-m_{r-1})w_r\Big] \label{ibp_Psi2} \\
%&= \frac{M}{2}\Big[\theta(\vc 1)-\sum_{r=1}^k\sum_{n=n_{r-1}+1}^{n_r}(\wt m_n-\wt m_{n-1})w_r\Big] \notag \\
%&= \PPP_{M,2}(\wt m;\underbrace{w_1,\dots,w_1}_{n_1},\underbrace{w_2,\dots,w_2}_{n_2-n_1},\cdots,\underbrace{w_k,\dots,w_k}_{n_k-n_{k-1}}). \notag
%\end{align}
%That is, the value of $\PPP_{M,2}$ has not changed with the duplicate variables.
%
%It remains to show that the same is true for $\PPP_{M,1}$.
%Recall that $\PPP_{M,1}^s(m;u_1^s,\dots,u_k^s)$ can be obtained using the inductive procedure \eqref{following_just_zs}, with \eqref{single_species_before_iteration} as the initialization.
%But whenever $u^s_{r+1}=u^s_r$, the $r^\text{th}$ iteration in \eqref{following_just_zs} has no effect, by which me mean $F^s_r = F^s_{r+1}$.
%Therefore, when we use this procedure to compute
%$\PPP_{M,1}^s(\wt m;u_1^s,\dots,u_1^s,\cdots,u_k^s,\dots,u_k^s)$, the addition of duplicate variables simply introduces extra iterations affecting no change.
%Since $\wt m_{n_r}=m_r$, all other iterations will coincide with those used to obtain $\PPP_{M,1}^s(m;u_1^s,\dots,u_k^s)$ in the first place.
\end{proof}

The final preparation before proving Proposition \ref{lipschitz_continuity_restricted} is to control the variability of $\PPP_M$ with the $\vc q$ sequence.
The following lemma will be essential.
The quantity $\delta_r^s(u_1,\dots,u_k)$ seen in \eqref{ordered_derivative} is deserving of the title ``partial derivative of $\PPP_{M,1}^s$ with respect to $u_r$'', but because the $u_r$'s must stay ordered, we must be careful in how we state this.
The definition \eqref{derivative_limit} will soon clarify these subtleties.

\begin{lemma} \label{finite_M_lipschitz_lemma}
%We have
%\eeq{ \label{highest_partial}
%\frac{\partial F_{M^s,0}^s}{\partial u^s_{k+1}} = \frac{M^s}{2},
%}
Fix any sequence $0 = u_0 \leq u_1\leq\cdots\leq u_k\leq u_{k+1}=\xi^s(\vc1)$.
Let $(a_1,\dots,a_k)\in\R^k$ be such that $a_r\geq a_{r-1}$ whenever $u_r=u_{r-1}$, where $a_0 = a_{k+1}=0$.
We then have
\eeqc{ \label{ordered_derivative}
\lim_{\eps\searrow0}\frac{\PPP_{M,1}^s(m;u_1+\eps a_1 ,\dots,u_k+\eps a_k)-\PPP_{M,1}^s(m;u_1 ,\dots,u_k)}{\eps}
= \sum_{r=1}^k a_r\delta_r^s(u_1,\dots,u_k), \\
 \label{other_partials}
 \text{where} \quad
 -\frac{M^s}{2}(m_r-m_{r-1})
 \leq \delta_r^s(u_1,\dots,u_k)
 \leq 0.
}
\end{lemma}

\begin{proof}
%The first claim \eqref{highest_partial} is clear from \eqref{reinterpreting_last_iterate}.
%For the remaining claim \eqref{other_partials}, 
%Since $s$ remains constant throughout the proof, let us just write $u_r^s=u_r$.
The assumption on $(a_1,\dots,a_k)$ is so that for all sufficiently small $\eps>0$, we have
\eq{ %\label{all_perturbed}
u_0 \leq u_1 + \eps a_1 \leq u_2+\eps a_2 \leq \cdots \leq u_k+\eps a_k \leq u_{k+1}.
}
In other words, if all coordinates are perturbed simultaneously, then ordering is preserved.
But we will need to perturb the coordinates one at a time, hence the following claim.

\begin{claim} \label{permutation_claim}
There is some permutation $(\varrho(1),\dots,\varrho(k))$ of $(1,\dots,k)$ such that for all sufficiently small $\eps>0$ and any $j\in\{1,\dots,k\}$, we have
\eeq{ \label{only_some_perturbed}
u_0 \leq u_1 + \one_{\{\varrho(1)\leq j\}}\eps a_1 \leq u_2+\one_{\{\varrho(2)\leq j\}}\eps a_2 \leq \cdots \leq u_k+\one_{\{\varrho(2)\leq j\}}\eps a_k \leq u_{k+1}.
}
In other words, ordering is preserved even if only coordinates $\varrho^{-1}(1),\dots,\varrho^{-1}(j)$ have been perturbed.
\end{claim}

\begin{proofclaim}
We argue by induction on $k$, the base case of $k=1$ being trivial.
So assume $k\geq2$.
If $u_1 < u_2$, then first apply the inductive hypothesis to coordinates $2$ through $k$, and set $\varrho(1) = k$.
Indeed, even if $u_1$ is the last coordinate to be perturbed, we will have $u_1 < u_2 + a_2\eps$ for all $\eps$ sufficiently small.
Hence \eqref{only_some_perturbed} will be true for all $j\leq k-1$ by induction, and true for $j=k$ because $u_1 + a_1\eps < u_2 + a_2\eps$ for all $\eps$ sufficiently small.

Otherwise $u_1 = u_2$ (so we must have $a_1\leq a_2$), and we consider two separate cases.
If $a_1< 0$, then set $\varrho(1)=1$.
That is, we first perturb $u_1$ to arrive at $u_1+\eps a_1$, which is now strictly less than $u_2$, and so \eqref{only_some_perturbed} holds for $j=1$.
We then decide in which order to make the remaining perturbations by applying the inductive hypothesis to coordinates $2$ through $k$, which will ensure \eqref{only_some_perturbed} for all $j\geq2$.

If instead $a_1\geq0$, then again apply the inductive hypothesis to coordinates $2$ through $k$, and set $\varrho(1) = k$.
Indeed, even if $u_1$ is the last coordinate to be perturbed, the assumption $a_2\geq a_1\geq0$ means that
$u_1\leq u_2 + a_2\eps$ for all $\eps\geq0$.
So as before, \eqref{only_some_perturbed} will be true for all $j\leq k-1$ by induction, and true for $j=k$ because $u_1+a_1\eps\leq u_2+a_2\eps$ for all $\eps\geq0$.
\end{proofclaim}

Now fix the permutation $\varrho$ from Claim \ref{permutation_claim}, and fix $\eps>0$ small enough that \eqref{only_some_perturbed} holds for all $j\in[k]$.
We then write
\eeq{ \label{k_many_summands}
\PPP_{M,1}^s(m;u_1+\eps a_1 ,\dots,u_k+\eps a_k)-\PPP_{M,1}^s(m;u_1 ,\dots,u_k)
= \sum_{j=1}^k[f_j(\eps) - f_{j-1}(\eps)\big],
}
\eq{
\text{where} \quad f_j(\eps') &\coloneqq \PPP_{M,1}^s(m;u_1^{(j)}(\eps'),\dots,u_k^{(j)}(\eps')), \\ 
u_r^{(j)}(\eps') &\coloneqq u_r + \one_{\{\varrho(r)< j\}}\eps a_r + \one_{\{\varrho(r)=j\}}\eps' a_r, \quad
\eps'\in[0,\eps].
}
%\sum_{j=1}^k\big[\PPP_{M,1}^s(m;u_1+\one_{\{\varrho(1)\leq j\}}\eps a_1 ,\dots,u_k+\one_{\{\varrho(k)\leq j\}}\eps a_k)\\
%&\phantom{=\sum_{j=1}^k\big[}-\PPP_{M,1}^s(m;u_1+\one_{\{\varrho(1)< j\}}\eps a_1 ,\dots,u_k+\one_{\{\varrho(k)<j\}}\eps a_k)\big]
In words, $f_j$ is the result of perturbing coordinates $\varrho^{-1}(1),\dots,\varrho^{-1}(j)$, with a possibly smaller perturbation on the last coordinate in this list. 
Given $r\in[k]$, suppose $\varrho(r) = j$ so that the $j^\text{th}$ summand in \eqref{k_many_summands} is the first one in which $u_r$ is perturbed.
If $a_r = 0$, then $f_j = f_{j-1}$, and we need not consider this summand further.
If $a_r>0$, then we have
\eq{
u_{r-1} + \one_{\{\varrho(r-1)< j\}}a_{r-1}\eps 
\stackref{only_some_perturbed}{\leq} u_r < u_r + a_r\eps 
\stackref{only_some_perturbed}{\leq} u_{r+1}+\one_{\{\varrho(r+1)< j\}}a_{r+1}\eps.
}
Squeezing an additional term between $u_r$ and $u_r+a_r\eps$, we obtain that for all $\eps'\in(0,\eps)$,
\eeq{ \label{isolate_at_j}
u_{r-1} + \one_{\{\varrho(r-1)< j\}}a_{r-1}\eps
< u_r + a_r\eps' < u_{r+1}+\one_{\{\varrho(r+1)< j\}}a_{r+1}\eps.
}
By analogous reasoning, we obtain the same inequality when $a_r<0$.
We have thus reduced the problem to the following claim.

\begin{claim} \label{differentiability_claim}
Whenever $u_{r-1} < u_r < u_{r+1}$, we can differentiate $\PPP_{M,1}^s$ with respect to $u_r$.
The resulting derivative satisfies
\eeq{ \label{other_partials_pre}
-\frac{M^s}{2}(m_r-m_{r-1})\leq \frac{\partial\PPP_{M,1}^s(m;x_1,\dots,x_k)}{\partial x_r}\Big|_{(x_1=u_1,\dots,x_k=u_k)} \leq 0.
}
Furthermore, for any $\wt u = (\wt u_1\leq\cdots\leq\wt u_{k})$, the following limit exists:
\eeq{ \label{derivative_limit}
\delta_r^s(\wt u_1,\dots,\wt u_k) \coloneqq \lim_{u\to\wt u} \frac{\partial\PPP_{M,1}^s(m;x_1,\dots,x_k)}{\partial x_r}\Big|_{(x_1=u_1,\dots,x_k=u_k)},
}
where the limit is taken along any $u$ with $u_{r-1}<u_r<u_{r+1}$.
\end{claim}

Before proving the claim, let us use it to complete the proof of the lemma.
Consider the $j^\text{th}$ summand from \eqref{k_many_summands}, with the assumption that $\varrho(r)=j$ and $a_r \neq 0$ as discussed above.
By Claim \ref{differentiability_claim} and the inequality \eqref{isolate_at_j}, the function $\eps'\mapsto f_j(\eps')$
is differentiable on the open interval $(0,\eps)$.
As will be checked during the proof of Claim \ref{differentiability_claim}, this map is also continuous on the closed interval $[0,\eps]$, with $f_j(0)$ obviously equal to $f_{j-1}(\eps)$.
Therefore, by the mean value theorem, we have
\eq{ %\label{mvt_applied}
\frac{f_j(\eps) - f_{j-1}(\eps)}{\eps} = 
a_r\frac{\partial\PPP_{M,1}^s}{\partial x_r}\Big|_{(x_1=u_1^{(j)}(\eps'),\dots,x_k=u_k^{(j)}(\eps'))} \quad \text{for some $\eps'\in(0,\eps)$}.
%&\eps^{-1}\big[\PPP_{M,1}(m;\dots,u_{r-1} + \one_{\{\varrho(r-1)< j\}}a_{r-1}\eps, u_r + a_r\eps, u_{r+1}+\one_{\{\varrho(r+1)< j\}}a_{r+1}\eps,\dots)\\
%&\phantom{\eps^{-1}\big[}-\PPP_{M,1}(m;\dots,u_{r-1} + \one_{\{\varrho(r-1)< j\}}a_{r-1}\eps, u_r, u_{r+1}+\one_{\{\varrho(r+1)< j\}}a_{r+1}\eps,\dots)\big]
}
By \eqref{derivative_limit}, we then have
\eq{
\lim_{\eps\searrow0}\frac{f_j(\eps)-f_{j-1}(\eps)}{\eps} = a_r\delta_r^s(u_1,\dots,u_k).
}
Using this fact in \eqref{k_many_summands}, we are able to conclude \eqref{ordered_derivative}.
The inequality \eqref{other_partials} follows from \eqref{other_partials_pre}.

\begin{proof}[Proof of Claim \ref{differentiability_claim}]
Here we adapt the approach of \cite[Lem.~14.11.1]{talagrand11b}.
Recall that $\PPP_{M,1}^s=F_0^s$ is the result of applying \eqref{following_just_zs} with $F_k^s$ from \eqref{single_species_before_iteration} as the initialization.
But then \eqref{magical_iteration} implies
that $\PPP_{M,1}^s(m;u_1,\dots,u_k)$ is equal to
\eeq{ \label{using_magical_in_reverse}
\E\log\Big\langle\int_{ S_{M^s}}
\exp\Big(\sum_{j\in\JJ^s}\kappa_j\sum_{\beta\in p(\alpha)}\eta_{j,\beta}\sqrt{u_{|\beta|+1}-u_{|\beta|}}\,\Big)\ \mu_{M^s}(\dd\kappa)\Big\rangle
 + \frac{M^s}{2}(u_{k+1}-u_k),
}
where $u_0=0$, and $\langle\cdot\rangle$ denotes expectation according to the Poisson--Dirichlet cascade $(v_\alpha)_{\alpha\in\N^{k-1}}$ associated to \eqref{m_sequence}.
Let us simplify notation by writing
\eq{
Z \coloneqq \Big\langle\int_{ S_{M^s}}\exp\Big(\sum_{j\in\JJ^s}\kappa_jg_{j,\alpha}\Big)\ \mu_{M^s}(\dd\kappa)\Big\rangle
= \sum_{\alpha\in\N^{k-1}}v_\alpha\int_{ S_{M^s}}\exp\Big(\sum_{j\in\JJ^s}\kappa_jg_{j,\alpha}\Big)\ \mu_{M^s}(\dd\kappa),
}
where we have grouped the Gaussian variables into terms of the form
\eeq{ \label{gaussian_group}
g_{j,\alpha} \coloneqq \sum_{\beta\in p(\alpha)} \eta_{j,\beta}\sqrt{u_{|\beta|+1}-u_{|\beta|}}, \quad j\in\JJ^s, \alpha\in\N^{k-1}.
}
In this notation, differentiating \eqref{using_magical_in_reverse} with respect to $u_r$ results in
\eeq{ \label{having_taken_deriv}
\frac{\partial \PPP_{M,1}^s(m;u_1,\dots,u_k)}{\partial u_{r}}
= \E\Big[\frac{1}{Z}\cdot\frac{\partial Z}{\partial u_r}\Big]
- \one_{\{r=k\}}\frac{M^s}{2}.
}
Let us define
\eeq{ \label{Q_j_def}
Q_{j_1}(\alpha) \coloneqq \frac{1}{v_\alpha}\frac{\partial Z}{\partial g_{j_1,\alpha}} = \int_{ S_{M^s}}\kappa_{j_1}\exp\Big(\sum_{j\in\JJ^s}\kappa_{j}g_{j,\alpha}\Big)\ \mu_{M^s}(\dd\kappa)
 \quad \text{and} \quad
g_{j,\alpha}' \coloneqq \frac{\partial g_{j,\alpha}}{\partial u_r},
}
so that by the chain rule,
\eeq{ \label{pre_gibp}
\E\Big[\frac{1}{Z}\cdot\frac{\partial Z}{\partial u_r}\Big]
%&= \E\Big[\frac{1}{Z}\sum_{\alpha\in\N^{k-1}}v_{\alpha}\int_{ S_{M^s}}\Big(\sum_{j_1\in\JJ^s}\kappa_{j_1}g_{j_1,\alpha}'\Big)\exp\Big(\sum_{j\in\JJ^s}\kappa_{j}g_{j,\alpha}\Big)\ \mu_{M^s}(\dd\kappa)\Big] \\
&= \E\bigg[\sum_{(j_1,\alpha)\in\JJ^s\times\N^{k-1}}v_{\alpha} \E_{g}\Big[\frac{1}{Z}\cdot Q_{j_1}(\alpha)g_{j_1,\alpha}'\Big]\bigg],
}
where $\E_g(\cdot)$ denotes expectation over only the Gaussian random variables.
The right-hand side of \eqref{pre_gibp} sets up the following Gaussian integration by parts:
\eeq{ \label{post_gibp}
\E_{g}\Big[\, g_{j_1,\alpha^1}'\cdot\frac{Q_{j_1}(\alpha^1)}{Z}\Big]
&= \sum_{(j_2,\alpha^2)\in\JJ^s\times\N^{k-1}}\E_g(g_{j_1,\alpha^1}'g_{j_2,\alpha^2})\cdot\E_g\Big[\frac{\partial}{\partial g_{j_2,\alpha^2}}\frac{Q_{j_1}(\alpha^1)}{Z}\Big].
}
%\eeq{ \label{derivative_just_Z}
%\frac{\partial Z}{\partial u_r}
%= \sum_{(\alpha,j_1)\in\N^{k-1}\times\JJ^s} v_\alpha g_{j_1,\alpha}'Q_{j_1}(\alpha).
%%= \sum_{\alpha\in\N^{k-1}}v_{\alpha}\int_{ S_{M^s}}\Big(\sum_{j_1\in\JJ^s}\kappa_{j_1}g_{j_1,\alpha}'\Big)\exp\Big(\sum_{j\in\JJ^s}\kappa_{j}g_{j,\alpha}\Big)\ \mu_{M^s}(\dd\kappa).
%}
We will now consider two cases: $r<k$ and $r=k$.

If $1 \leq r \leq k-1$, then it is easily seen from \eqref{gaussian_group} that
\eeq{ \label{differentiate_gaussian}
g_{j,\alpha}'
= \frac{1}{2}\Big(\frac{\eta_{j,(\alpha_1,\dots,\alpha_{r-1})}}{\sqrt{u_r-u_{r-1}}}-\frac{\eta_{j,(\alpha_1,\dots,\alpha_{r})}}{\sqrt{u_{r+1}-u_{r}}}\Big),
}
%in which case \eqref{derivative_just_Z} leads to
%\begin{subequations}
%\label{gibp}
%\eeq{ \label{pre_gibp}
%\E\Big[\frac{1}{Z}\cdot\frac{\partial Z}{\partial u_r}\Big]
%%&= \E\Big[\frac{1}{Z}\sum_{\alpha\in\N^{k-1}}v_{\alpha}\int_{ S_{M^s}}\Big(\sum_{j_1\in\JJ^s}\kappa_{j_1}g_{j_1,\alpha}'\Big)\exp\Big(\sum_{j\in\JJ^s}\kappa_{j}g_{j,\alpha}\Big)\ \mu_{M^s}(\dd\kappa)\Big] \\
%&= \E\bigg[\sum_{(\alpha,j_1)\in\N^{k-1}\times\JJ^s}v_{\alpha} \E_{g}\Big[\, g_{j_1,\alpha}'\cdot\frac{Q_{j_1}(\alpha)}{Z}\Big]\bigg],
%}
%where $\E_g(\cdot)$ denotes expectation over only the Gaussian random variables.
%To condense notation, we will write
%\eeq{ \label{Q_j_def}
%Q_{j_1}(\alpha) \coloneqq \int_{ S_{M^s}}\kappa_{j_1}\exp\Big(\sum_{j\in\JJ^s}\kappa_{j}g_{j,\alpha}\Big)\ \mu_{M^s}(\dd\kappa).
%}
%The right-hand side of \eqref{pre_gibp} sets up the following Gaussian integration by parts:
%\eeq{ \label{post_gibp}
%\E_{g}\Big[\, g_{j_1,\alpha^1}'\cdot\frac{Q_{j_1}(\alpha^1)}{Z}\Big]
%&= \sum_{(\alpha^2,j_2)\in\N^{k-1}\times\JJ^s}\E_g(g_{j_1,\alpha^1}'g_{j_2,\alpha^2})\cdot\E_g\Big[\frac{\partial}{\partial g_{j_2,\alpha^2}}\frac{Q_{j_1}(\alpha^1)}{Z}\Big].
%}
%\end{subequations}
Now recall the quantity $r(\alpha^1,\alpha^2)$ from \eqref{overlap_notation}.
Since all $\eta_{j,\beta}$'s are mutually independent, it follows from definitions \eqref{gaussian_group} and \eqref{differentiate_gaussian} that
\eeq{ \label{prime_no_prime_cov}
\E_g(g_{j_1,\alpha^1}'g_{j_2,\alpha^2}) &= \frac{1}{2}\one_{\{j_1=j_2\}}\big(\E(\eta_{j_1,(\alpha^1_1,\dots,\alpha^1_{r-1})}\eta_{j_2,(\alpha^2_1,\dots,\alpha^2_{r-1})})-\E(\eta_{j_1,(\alpha^1_1,\dots,\alpha^1_{r})}\eta_{j_2,(\alpha^2_1,\dots,\alpha^2_{r})})\big) \\
&=\frac{1}{2}\one_{\{j_1=j_2\}}\begin{cases}
0 - 0 &\text{if $r(\alpha^1,\alpha^2) < r$},\\
1 - 1 &\text{if $r(\alpha^1,\alpha^2) > r$}, \\
1-0 &\text{if $r(\alpha^1,\alpha^2) = r$}.
\end{cases}
\raisetag{2\baselineskip}
}
Therefore, in \eqref{post_gibp} we need only consider $(j_2,\alpha^2)$ such that $j_2=j_1$ and $r(\alpha^1,\alpha^2)=r$.
Notice that the latter equality implies $\alpha^2\neq\alpha^1$ since $r< k$, and so the variable $g_{j,\alpha^2}$ does not appear in $Q_{j}(\alpha^1)$, which means
\eeq{ \label{differentiate_Q}
\frac{\partial}{\partial g_{j,\alpha^2}}\frac{Q_{j}(\alpha^1)}{Z}
= -\frac{Q_{j}(\alpha^1)}{Z^2}\cdot\frac{\partial Z}{\partial g_{j,\alpha^2}} 
\stackref{Q_j_def}{=} -v_{\alpha^2}\frac{Q_{j}(\alpha^1)Q_{j}(\alpha^2)}{Z^2}.
}
Using \eqref{prime_no_prime_cov} and \eqref{differentiate_Q} in \eqref{post_gibp}, and then \eqref{post_gibp} in \eqref{pre_gibp}, we arrive at %from \eqref{having_taken_deriv} to
\eeq{ \label{convert_to_tilt}
\frac{\partial \PPP_{M,1}^s}{\partial u_r}
&=
-\frac{1}{2}\sum_{j\in\JJ^s}\E\Big[\frac{1}{Z^2}\sum_{\alpha^1,\alpha^2\in\N^{k-1}}\one_{\{r(\alpha^1,\alpha^2)=r\}}v_{\alpha^1}v_{\alpha^2}Q_j(\alpha^1)Q_j(\alpha^2)\Big].
}
This concludes our consideration of the case $r<k$.

If instead $r=k$, then Gaussian integration by parts is still executed as in \eqref{post_gibp}, but \eqref{differentiate_gaussian} is replaced by
\eq{
g_{j,\alpha}' = \frac{\eta_{j,(\alpha_1,\dots,\alpha_{k-1})}}{2\sqrt{u_k-u_{k-1}}}.
}
Hence \eqref{prime_no_prime_cov} is replaced by
\eq{
\E_g(g_{j,\alpha^1}'g_{j_2,\alpha^2}) = \one_{\{j_1=j_2\}}\one_{\{\alpha^1=\alpha^2\}}/2,
}
which in turn implies \eqref{differentiate_Q} is replaced by
\eq{
\frac{\partial}{\partial g_{j,\alpha}}\frac{Q_j(\alpha)}{Z}
= - v_{\alpha}\frac{Q_j(\alpha)^2}{Z^2}+\frac{1}{Z}\cdot\frac{\partial Q_j(\alpha)}{\partial g_{j,\alpha}}.
}
This means the outcome of using \eqref{post_gibp} to compute \eqref{having_taken_deriv} is now
\eq{
\frac{\partial \PPP_{M,1}^s}{\partial u_k}
=  \frac{1}{2}\sum_{(j,\alpha)\in\JJ^s\times\N^{k-1}}\E\Big[-v_\alpha^2\frac{Q_{j}(\alpha)^2}{Z^2} + \frac{v_\alpha}{Z}\cdot\frac{\partial Q_{j}(\alpha)}{\partial g_{j,\alpha}}\Big]  -\frac{M^s}{2}.
}
But notice that the additional terms created by differentiating $Q_{j}(\alpha)$ cancel with the additional $-M^s/2$, since differentiating in \eqref{Q_j_def} leads to
\eq{
\sum_{(j,\alpha)\in\JJ^s\times\N^{k-1}} \frac{v_\alpha}{Z}\cdot\frac{\partial Q_{j}(\alpha)}{\partial g_{j,\alpha}}
= \sum_{\alpha\in\N^{k-1}}\frac{v_\alpha}{Z}\int_{ S_{M^s}}\Big(\sum_{j\in\JJ^s}\kappa_{j}^2\Big)\exp\Big(\sum_{j\in\JJ^s}\kappa_jg_{j,\alpha}\Big)\ \mu_{M^s}(\dd\kappa) = M^s.
}
Therefore, \eqref{convert_to_tilt} holds even in the case $r=k$.

In order to rewrite \eqref{convert_to_tilt}  using the notation of \eqref{tilted_Gibbs}, set
\eeq{ \label{F_alpha_def}
F(\alpha) = \log \int_{ S_{M^s}}\exp\Big(\sum_{j\in\JJ^s}\kappa_jg_{j,\alpha}\Big)\ \mu_{M^s}(\dd\kappa), \qquad
U^{(j)}(\alpha) = \frac{Q_j(\alpha)}{\exp F(\alpha)}.
}
Then \eqref{convert_to_tilt} can be rewritten as
\eeq{ \label{rewriting_as_tilt}
\frac{\partial \PPP_{M,1}^s}{\partial u_r}
&\stackrefp{big_generalization}{=}-\frac{1}{2}\sum_{j\in\JJ^s}\E\langle\one_{\{r(\alpha^1,\alpha^2)=r\}} U^{(j)}(\alpha^1)U^{(j)}(\alpha^2)\rangle_F\\
&\stackref{big_generalization}{=}-\frac{1}{2}(m_r-m_{r-1})\sum_{j\in\JJ^s}\E\big[W_1\cdots W_{r-1}(\E_r[W_r\cdots W_{k-1}U_k^{(j)}])^2\big] \leq 0.
}
On the other hand, by Jensen's inequality we have
\eeq{ \label{Uj_jensen}
\sum_{j\in\JJ^s}U^{(j)}(\alpha)^2 
&= \sum_{j\in\JJ^s}\Big(\frac{Q_{j}(\alpha)}{\exp F(\alpha)}\Big)^2 \\
&\leq \frac{1}{\exp F(\alpha)}\int_{ S_{M^s}}\Big(\sum_{j\in\JJ^s}\kappa_{j}^2\Big) \exp\Big(\sum_{j\in\JJ^s}\kappa_{j}g_{j,\alpha}\Big)\ \mu_{M^s}(\dd\kappa) = M^s.
}
Consequently, an application of Cauchy--Schwarz yields
%together with the fact that $\sum_{j\in\JJ^s}\kappa_j^2 = M^s$, yields
\eq{
\bigg|\E\bigg\langle\one_{\{r(\alpha^1,\alpha^2)=r\}}\sum_{j\in\JJ^s} U^{(j)}(\alpha^1)U^{(j)}(\alpha^2)\bigg\rangle_F\bigg|
&\leq\E\bigg\langle\one_{\{r(\alpha^1,\alpha^2)=r\}}\bigg|\sum_{j\in\JJ^s} U^{(j)}(\alpha^1)U^{(j)}(\alpha^2)\bigg|\bigg\rangle_F \\
&\leq M^s\E\langle\one_{\{r(\alpha^1,\alpha^2)=r\}}\rangle_F
\stackref{still_have}{=}M^s(m_r-m_{r-1}).
}
The proof of \eqref{other_partials_pre} is completed by using this inequality in the first line of \eqref{rewriting_as_tilt}.

Our last objective is to prove \eqref{derivative_limit}, as well as continuity of $\PPP_{M,1}^s$ jointly in all coordinates $u_1,\dots,u_k$.
It is clear from \eqref{gaussian_group} that $g_{j,\alpha}$ is continuous in $u_1,\dots,u_k$.
We claim that as a consequence, the quantities $Q_{j_1}(\alpha)$ from \eqref{Q_j_def} and $\e^{F(\alpha)}$ from \eqref{F_alpha_def} are almost surely (i.e.~for almost any realization of $(\eta_{j,\beta})_{j\in\JJ^s,\beta\in\N^0\cup\cdots\cup\N^{k-1}}$) continuous in $u_1,\dots,u_k$.
Indeed, observe that
\eq{
\exp(\pm\kappa_jg_{j,\alpha})
&= \prod_{\beta\in p(\alpha)} \exp(\pm\kappa_j\eta_{j,\beta}\sqrt{u_{|\beta|+1}-u_{|\beta|}}\, ) \\
&\leq \prod_{\beta\in p(\alpha)}\big[\exp\big(|\kappa_j|\eta_{j,\beta}\sqrt{\xi^s(\vc 1)}\, \big)+\exp\big(-|\kappa_j|\eta_{j,\beta}\sqrt{\xi^s(\vc 1)}\, \big)\big]
\coloneqq d_{j}(\alpha,\kappa),
}
where now $d_{j}(\alpha,\kappa)$ has no dependence on $u_1,\dots,u_k$.
From this inequality we have
\begin{subequations}
\label{F_dominated}
\eeq{
\e^{F(\alpha)} = \int_{S_{M^s}}\prod_{j\in\JJ^s}\exp(\kappa_jg_{j,\alpha})\ \mu_{M^s}(\dd\kappa)
\leq \int_{S_{M^s}}\prod_{j\in\JJ^s}d_j(\alpha,\kappa)\ \mu_{M^s}(\dd\kappa) \coloneqq D(\alpha,\kappa),
}
as well as
\eeq{ %\label{F_inv_dominated}
\e^{-F(\alpha)}&=\bigg(\int_{S_{M^s}}\prod_{j\in\JJ^s}\exp(\kappa_jg_{j,\alpha})\ \mu_{M^s}(\dd\kappa)\bigg)^{-1}\\
&\leq \int_{S_{M^s}}\prod_{j\in\JJ^s}\exp(-\kappa_jg_{j,\alpha})\ \mu_{M^s}(\dd\kappa)
\leq D(\alpha,\kappa).
}
\end{subequations} 
From the calculation
\eeq{ \label{expected_bound_calc}
\E_\eta\prod_{j\in\JJ^s}d_j(\alpha,\kappa)
=\prod_{j\in\JJ^s}\E_\eta d_j(\alpha,\kappa)
%= \prod_{j\in\JJ^s}\E_g\exp(\pm|\kappa_j|\eta_{j,\beta}\xi^s(\vc 1))
= \prod_{j\in\JJ^s}\big(\kappa_j\sqrt{\xi^s(\vc 1)}\big)^{2|p(\alpha)|}
\leq (M^s\xi(\vc 1))^{kM^s},
}
we conclude that $D(\alpha,\kappa)$ is finite with probability one.
%\eq{
%\int_{S_{M^s}}\prod_{j\in\JJ^s}d_j(\alpha,\kappa)\ %\mu_{M^s}(\dd\kappa) < \infty \quad \mathrm{a.s.}
%}
Therefore, our claim of continuity for $\e^{F(\alpha)}$ follows from dominated convergence with respect to the probability measure $\mu_{M^s}$ on $S_{M^s}$.
For $Q_{j_1}(\alpha)$, we need only make the additional observation that $|\kappa_j|\leq\sqrt{M^s}$, and then the same argument goes through.

Given the continuity of $\e^{F(\alpha)}$ with respect to $u_1,\dots,u_k$, we would like to conclude the same for
$\PPP_{M,1}^s = \E\log\langle\e^{F(\alpha)}\rangle-M^s(u_{k+1}-u_k)/2$.
The argument given above shows that $\langle \e^{F(\alpha)}\rangle$ is continuous, simply by replacing $\int_{M^s}(\cdot)\, \mu_{M^s}(\dd\kappa)$ with 
$\big\langle\int_{M^s}(\cdot)\, \mu_{M^s}(\dd\kappa)\big\rangle$.
Indeed, dominated convergence applies equally well to the latter, since the right-hand side of \eqref{expected_bound_calc} has no dependence on $\alpha$.
To conclude continuity for $\PPP_{M,1}^s$, observe that
\eq{
|\log\langle \e^{F(\alpha)}\rangle|
&= \log\langle \e^{F(\alpha)}\rangle\one_{\{\langle \e^{F(\alpha)}\rangle\geq1\}}
+
 \log\langle \e^{F(\alpha)}\rangle^{-1}\one_{\{\langle \e^{F(\alpha)}\rangle<1\}} \\
&\leq 
\log\langle \e^{F(\alpha)}\rangle\one_{\{\langle \e^{F(\alpha)}\rangle\geq1\}}
+
 \log\langle \e^{-F(\alpha)}\rangle\one_{\{\langle \e^{F(\alpha)}\rangle<1\}}
 \stackref{F_dominated}{\leq} \log\langle D(\alpha,\kappa)\rangle.
}
Since another application of Jensen's inequality gives
\eq{
\E_g\log\langle D(\alpha,\kappa)\rangle
\leq \log \E_g\langle D(\alpha,\kappa)\rangle
= \log \langle \E_g D(\alpha,\kappa)\rangle
\stackref{expected_bound_calc}{<}\infty,
}
it follows from dominated convergence that $\E\log\langle\e^{F(\alpha)}\rangle$ is continuous in $u_1,\dots,u_k$.
The same is clearly true for $\PPP_{M,1}^s$.

Finally, since we know $Q_{j_1}(\alpha)$ and $\e^{F(\alpha)}$ are almost surely continuous in $u_1,\dots,u_k$, the same must be true for $U^{(j)}(\alpha)$ defined in \eqref{F_alpha_def}.
Thanks to \eqref{Uj_jensen}, we can apply dominated convergence in \eqref{rewriting_as_tilt} with respect to $\E\langle\cdot\rangle$, in order to conclude that $\partial\PPP_{M,1}^s/\partial u_r$ is continuous in $u_1,\dots,u_k$, as desired.
\renewcommand{\qedsymbol}{$\square$ (Claim and Lemma)}
\end{proof}
\renewcommand{\qedsymbol}{}
\end{proof}
\renewcommand{\qedsymbol}{$\square$} %\vspace{-1\baselineskip}

\begin{proof}[Proof of Proposition \ref{lipschitz_continuity_restricted}]
We are only considering $\vc\lambda$-admissible pairs $(\zeta,\Phi),(\wt\zeta,\wt{\Phi})$ such that $\zeta$ and $\wt\zeta$ have finite support.
So let $(\zeta,\Phi)$ correspond to sequences $m = (m_r)_{0\leq r\leq k}$ and $(\vc q_r)_{0\leq r\leq k+1}$, while $(\wt\zeta,\wt{\Phi})$ corresponds to $\wt m = (\wt m_r)_{0\leq r\leq \wt k}$ and $(\wt{\vc q}_r)_{0\leq r\leq \wt k+1}$.
By replacing $m$ and $\wt m$ with their mutual refinement, and then creating duplicate $\vc q$'s and $\wt{\vc q}$'s as needed, we may assume $k=\wt k$ and $m = \wt m$ thanks to Lemma \ref{m_lemma}.
Then observe that
\eq{
Q_\zeta(z) = q_r \quad \text{and} \quad
Q_{\wt\zeta}(z) = \wt q_r \quad \text{for $z\in(m_{r-1},m_{r}]$, $1\leq r\leq k$}.
}
Since $\Phi(q_r) = \vc q_r$ and $\wt{\Phi}(\wt q_r) = \wt{\vc q_r}$, upon integrating over all possible $z$, we arrive at the identity
\eeq{ \label{DD_identity}
\DD\big((\zeta,\Phi),(\wt\zeta,\wt{\Phi})\big)
&\stackref{pseudometric_def}{=}
\int_0^1 \|\Phi(Q_{\zeta}(z))-\wt{\Phi}(Q_{\wt \zeta}(z))\|_1\ \dd z \\
&\stackrefp{pseudometric_def}{=}\sum_{r=1}^k(m_r-m_{r-1})\|\vc q_r-\wt{\vc q}_r\|_1.
}

Our goal now is to control the difference $\frac{1}{M}|\PPP_M(m;\vc q_1,\dots,\vc q_k)-\PPP_M(m;\wt{\vc q}_1,\dots,\wt{\vc q}_k)|$ in terms of $\DD\big((\zeta,\Phi),(\wt\zeta,\wt{\Phi})\big)$.
To do this, we interpolate between $\vc q$ and $\wt{\vc q}$ by defining
\eq{
\vc q_r(t) \coloneqq (1-t)\vc q_r + t\wt{\vc q}_r, \quad \text{and then} \quad u_r^s(t) \coloneqq  %\textcolor{red}{\one_{\{s>0\}}}
\xi^s(\vc q_r(t)), \quad
w_r(t)\coloneqq \theta(\vc q_r(t)), \quad t\in[0,1].
}
The quantity of interest is then $|\vphi(0)-\vphi(1)|$, where
\eq{
\vphi(t) &\coloneqq \frac{1}{M}\sum_{s\in\SSS}\PPP_{M,1}^s(m;u_1^s(t),\dots, u_k^s(t)) -\frac{1}{M}\PPP_{M,2}(m;w_1(t),\dots,w_k(t)) \\
&= \sum_{s\in\SSS}\frac{M^s}{M}\cdot\frac{\PPP_{M,1}^s(m;u_1^s(t),\dots, u_k^s(t))}{M^s} - \frac{\theta(\vc 1)}{2} +\frac{1}{2}\sum_{r=1}^k(m_r-m_{r-1})w_r(t),
}
where in the second line we have applied summation by parts to \eqref{psi_2_after_magical}.
%Applying summation by parts to \eqref{psi_2_after_magical}, we obtain
%\begin{align}
%\PPP_{M,2}(m;w_1,w_2,\dots,w_k) 
%&= \frac{M}{2}\Big[\theta(\vc 1)-\sum_{r=1}^{k}(m_{r}-m_{r-1})w_r\Big] \label{ibp_Psi2} \\
For ease of notation, let us denote the quantity from Lemma \ref{finite_M_lipschitz_lemma} by
\eq{
\delta^s_r(t) \coloneqq \frac{1}{M^s}\delta_r^s(u_1^s(t),\dots,u_k^s(t)),
%\frac{\partial \PPP_{M,1}^s(m;u_1,\dots,u_k)}{\partial u_r}\Big|_{(u_1=u^s_1(t),\dots,u_k=u^s_k(t))},
}
which by \eqref{other_partials} satisfies
\eeq{ \label{lipschitz_lemma_consequence}
-\frac{m_r-m_{r-1}}{2}\leq\delta^s_r(t)\leq 0.
}
Note that because $u_{r-1}^s(t) \leq u_{r}^s(t)$ for all $t\in[0,1]$, the time derivatives $a_r = \dd u_r^s(t)/\dd t$ must satisfy the hypothesis of Lemma \ref{finite_M_lipschitz_lemma}.
So by \eqref{ordered_derivative}, we have
\eq{
\vphi'(t)
&= \sum_{s\in\SSS}\frac{M^s}{M}\sum_{r=1}^k\delta^s_r(t)\frac{\dd u_r^s(t)}{\dd t}
+ \frac{1}{2}\sum_{r=1}^k(m_r-m_{r-1})\frac{\dd w_r(t)}{\dd t}. 
}
With further applications of the chain rule, it is elementary to calculate
\eq{
\frac{\dd u_r^s(t)}{\dd t} &= \sum_{s'\in\SSS}\frac{\partial \xi^{s}}{\partial q^{s'}}\Big|_{\vc q = \vc q_r(t)}(\wt q_r^{s'}-q_r^{s'}), \\
\frac{\dd w_r(t)}{\dd t} &= \sum_{s'\in\SSS}\frac{\partial\theta}{\partial q^{s'}}\Big|_{\vc q =\vc q_r(t)}(\wt q_r^{s'}-q_r^{s'})
\stackref{theta_def}{=}\sum_{s'\in\SSS}\bigg[\sum_{s\in\SSS}q_r^s(t)\lambda^s\frac{\partial \xi^{s}}{\partial q^{s'}}\Big|_{\vc q = \vc q_r(t)}\bigg](\wt q_r^{s'}-q_r^{s'}).
}
%+ \frac{1}{2}\sum_{r=1}^k(m_r-m_{r-1})\sum_{s'\in\SSS}\frac{\partial\theta}{\partial q^{s'}}\Big|_{\vc q =\vc q_r(t)}(\wt q_r^{s'}-q_r^{s'}) \\
Together, the two previous displays yield
\eq{
\vphi'(t) &= \sum_{r=1}^k\sum_{s\in\SSS}\Big[\frac{M^s}{M}\delta_r^s(t)+\frac{1}{2}q_r^s(t)\lambda^s(m_r-m_{r-1})\Big]\sum_{s'\in\SSS}\frac{\partial \xi^{s}}{\partial q^{s'}}\Big|_{\vc q = \vc q_r(t)}(\wt q_r^{s'}-q_r^{s'}).
}
%Once we replace the ratio $M^s/M$ by $\lambda^s$, we can use 
Notice that we can combine \eqref{lipschitz_lemma_consequence} with the fact that $0\leq q_r^s(t)\leq1$, in order to write
\eq{
-\lambda^s\frac{m_r-m_{r-1}}{2}\leq
\lambda^s\delta^s_r(t)+\frac{1}{2}\lambda^sq_r^s(t)(m_r-m_{r-1})
\leq \lambda^s\frac{m_r-m_{r-1}}{2}.
}
Recalling the definition of $C_*$ from \eqref{lipschitz_continuity_eq}, we thus have
\eq{ %\label{derivative_vphi_ineq}
|\vphi'(t)|
%&\leq \frac{C_0}{2}\Big(\max_{s\in\SSS}\Big|\frac{M^s}{M}-\lambda^s\Big|+\sum_{r=1}^k(m_r-m_{r-1})\sum_{s\in\SSS}\lambda^s\sum_{s'\in\SSS}|(\wt q_r^{s'}-q_r^{s'})|\Big) \\
&\stackrefp{DD_identity}{\leq} \frac{C_*}{2}\sum_{r=1}^k\sum_{s\in\SSS}\Big(\Big|\frac{M^s}{M}-\lambda^s\Big|+\lambda^s\Big)(m_r-m_{r-1})\sum_{s'\in\SSS}|\wt q_r^{s'}-q_r^{s'}| \\
&\stackrefp{DD_identity}{=} \frac{C_*}{2}\bigg(1+\sum_{s\in\SSS}\Big|\frac{M^s}{M}-\lambda^s\Big|\bigg)\sum_{r=1}^k(m_r-m_{r-1})\|\wt{\vc q}_r-\vc q_r\|_1 \\
&\stackref{DD_identity}{=} \frac{C_*}{2}\bigg(1+\sum_{s\in\SSS}\Big|\frac{M^s}{M}-\lambda^s\Big|\bigg)\DD\big((\zeta,\Phi),(\wt\zeta,\wt{\Phi})\big).
}
As this inequality holds for all $t\in[0,1]$, the same upper bound holds for $|\vphi(0)-\vphi(1)|$.
\end{proof}

\subsection{Extending the Parisi functional to general $\vc\lambda$-admissible pairs} \label{extension_section}
It was established in Proposition \ref{lipschitz_continuity_restricted} that $\PPP_M$ is Lipschitz continuous (in particular, uniformly continuous) when restricted to $\vc\lambda$-admissible pairs $(\zeta,\Phi)$ in which $\zeta$ has finite support.
Such pairs are in fact dense among all $\vc\lambda$-admissible pairs.\footnote{This follows from \eqref{need_for_density_1}, but it can also be seen as follows: weak convergence $\zeta_k\Rightarrow\zeta$ implies that for any $\vc\lambda$-admissible map $\Phi$ we have $\zeta_k \circ \Phi^{-1}\Rightarrow \zeta\circ \Phi^{-1}$, which is equivalent to $\DD\big((\zeta_k,\Phi),(\zeta,\Phi)\big)\to0$.}
Therefore, $\PPP_M$ admits a unique continuous extension to all $\vc\lambda$-admissible pairs.
To be precise, this extension is defined by
\eeq{ \label{extension_def}
\PPP_M(\zeta,\Phi) \coloneqq \lim_{k\to\infty}\PPP_M(\zeta_k,\Phi),
}
where $(\zeta_k)_{k\geq1}$ is any sequence of finitely supported measures converging weakly to $\zeta$.
Of course, Proposition \ref{lipschitz_continuity_restricted} immediately generalizes to this extension.

\begin{cor} \label{lipschitz_continuity_prelimit}
For any $\vc\lambda$-admissible pairs $(\zeta_1,\Phi_1)$ and $(\zeta_2,{\Phi}_2)$, we have
\eeq{ \label{lipschitz_continuity_prelimit_eq}
\frac{|\PPP_M(\zeta_1,\Phi_1)-\PPP_M(\zeta_2,{\Phi}_2)|}{M} \leq \frac{C_*}{2}\bigg(1+\sum_{s\in\SSS}\Big|\frac{M^s}{M}-\lambda^s\Big|\bigg)\DD\big((\zeta_1,\Phi_1),(\zeta_2,{\Phi}_2)\big),
}
where $C_*$ is given in \eqref{lipschitz_continuity_eq}.
\end{cor}

But in order for the limit in \eqref{extension_def} to be interchangeable with the limit $M\to\infty$, we will also need that $\PPP(\zeta,\Phi)=\lim_{k\to\infty}\PPP(\zeta_k,\Phi)$.
This will follow from the following result.

\begin{prop} \label{density_lemma}
Let $(\zeta,\Phi)$ be any $\vc\lambda$-admissible pair.
For any $\eps_1,\eps_2>0$, there is a measure $\wt\zeta$ on $[0,1]$ with finite support, such that
\eeqs{
%\int_0^1 \|Q_\zeta(z)-Q_{\wt\zeta}(z)\|\ \dd z &\leq \eps_1, \label{need_for_density_3} \\
%\int_0^1 \|\Phi(Q_{\zeta}(z))-{\Phi}(Q_{\wt \zeta}(z)))\|_1\ \dd z
\DD\big((\zeta,\Phi),(\wt\zeta,\Phi)\big)
&\leq \eps_1, \quad \text{and} \label{need_for_density_1} \\
|\PPP(\zeta,\Phi)-\PPP(\wt \zeta,{\Phi})| &\leq \eps_2. \label{need_for_density_2}
}
\end{prop}

Before proving this proposition, let us use it to quickly establish Theorem \ref{lipschitz_continuity}.

\begin{proof}[Proof of Theorem \ref{lipschitz_continuity}]
Let $(\zeta_1,\Phi_1)$ and $(\zeta_2,\Phi_2)$ be given.
Given any $\eps>0$, use Proposition \ref{density_lemma} to identify finitely supported measures $\wt\zeta_1$ and $\wt\zeta_2$ such that
\eeq{ \label{choosing_finite_approximations}
\DD\big((\zeta_i,\Phi_i),(\wt\zeta_i,{\Phi}_i)\big) \leq \eps \quad \text{and} \quad
|\PPP(\zeta_i,\Phi_i)-\PPP(\wt\zeta_i,{\Phi}_i)|\leq \eps \quad \text{for $i\in\{1,2\}$}.
}
The combination of Propositions \ref{lipschitz_continuity_restricted} and \ref{explain_appearance} gives
\eq{
|\PPP(\wt\zeta_1,\Phi_1)-\PPP(\wt\zeta_2,\Phi_2)|
\leq \frac{C_*}{2}\DD\big((\wt\zeta_1,\Phi_1),(\wt\zeta_2,\Phi_2)\big),
}
and so \eqref{choosing_finite_approximations} yields
\eq{
|\PPP(\zeta_1,{\Phi}_1)-\PPP(\zeta_2,{\Phi}_2)|
\leq 2\eps + \frac{C_*}{2}\Big[\DD\big((\zeta_1,\Phi_1),(\zeta_2,\Phi_2)\big)+2\eps\Big].
}
The proof is completed by letting $\eps$ tend to $0$.
\end{proof}

It is also easy to check that Proposition \ref{explain_appearance} continues to hold for the extended $\PPP_M$.

\begin{prop} \label{appearance_explained}
Assume that $M^s/M\to\lambda^s$ as $M\to\infty$, for each $s\in\SSS$.
For any $\vc\lambda$-admissible pair $(\zeta,\Phi)$, we have
\eeq{ \label{appearance_explained_eq}
\lim_{M\to\infty}\frac{\PPP_M(\zeta,{\Phi})}{M}=\PPP(\zeta,{\Phi}).
}
\end{prop}

\begin{proof}
Given any $\eps>0$, use Proposition \ref{density_lemma} to identify a finitely supported measure $\wt\zeta$ such that
$\DD\big((\zeta,\Phi),(\wt\zeta,{\Phi})\big) \leq \eps$.
Using the facts we have accumulated, we determine that
\eq{
\limsup_{M\to\infty}\frac{\PPP_M(\zeta,{\Phi})}{M}
&\stackref{lipschitz_continuity_prelimit_eq}{\leq} \lim_{M\to\infty}\frac{\PPP_M(\wt\zeta,\Phi)}{M}+\eps\frac{C_*}{2}  \\
&\stackref{explain_appearance_eq}{=} \PPP(\wt\zeta,\Phi)+\eps\frac{C_*}{2} \stackref{lipschitz_continuity_eq}{\leq} \PPP(\zeta,{\Phi})+\eps C_*.
}
An analogous chain of inequalities would also yield
\eq{
\liminf_{M\to\infty}\frac{\PPP_M(\zeta,{\Phi})}{M} \geq \PPP(\zeta,{\Phi}) - \eps C_*.
}
As $\eps$ is arbitrary, we can safely conclude \eqref{appearance_explained_eq}.
\end{proof}

The only remaining task of the section is to prove Proposition \ref{density_lemma}.
In preparation for the proof, let us make the following observation about quantile functions.

\begin{lemma} \label{quantile_claim}
Let $\zeta$ be any Borel probability measure on $[0,1]$.
Given $f:[0,1]\to[0,\infty)$, let $\zeta\circ f^{-1}$ denote the pushfoward of $\zeta$ under $f$.
If $f$ is left-continuous, non-decreasing, and satisfies $f(0)=0$, then
\eeq{ \label{push_quantiles}
f(Q_\zeta(z)) = Q_{\zeta\circ f^{-1}}(z) \quad \text{for all $z\in[0,1]$}.
}
\end{lemma}

\begin{proof}
We will prove \eqref{push_quantiles} by exhibiting inequalities in both directions.
On one hand, since $f$ is non-decreasing, we have $f^{-1}[0,f(q)] \supset [0,q]$ for any $q\in[0,1]$.
Consequently,
\eq{
(\zeta\circ f^{-1})\big([0,f(Q_\zeta(z))]\big)
\geq \zeta\big([0,Q_\zeta(z)]\big) \geq z,
}
which shows that $f(Q_\zeta(z)) \geq Q_{\zeta\circ f^{-1}}(z)$.

For the other direction, observe that for any $q\in[0,\infty)$, the monotonicity and left-continuity of $f$ together ensure $f^{-1}\big([0,q]\big)=[0,u_q]$ for some $u_q\in[0,1]$; in particular, $f(u_q) \leq q$.
So whenever $q$ satisfies $(\zeta\circ f^{-1})\big([0,q]\big) \geq z$, we must have
$\zeta\big([0,u_q]\big) \geq z$, which means  $Q_\zeta(z)\leq u_q$ and thus $f(Q_\zeta(z)) \leq f(u_q) \leq q$.
Since $Q_{\zeta\circ f^{-1}}(z)$ is one such $q$, we conclude that $f(Q_\zeta(z))\leq Q_{\zeta\circ f^{-1}}(z)$.
\end{proof}

Now, it is well known that for any Borel probability measure $\zeta$ on $[0,\infty)$, we have
\eeq{ \label{well_known}
\int_{[0,\infty)} u\ \zeta(\dd u) = \int_0^1 Q_{\zeta}(z)\ \dd z.
}
Lemma \ref{quantile_claim} leads to the following elementary extension of this fact.

\begin{lemma}
For any Borel probability measure $\zeta$ on $[0,1]$ and any left-continuous, non-decreasing function $f\colon[0,1]\to[0,\infty)$, we have
\eeq{ \label{first_claim_identity}
\int_{[0,1]} f(u)\ \zeta(\dd u) = \int_0^1 f(Q_{\zeta}(z))\ \dd z.
}
\end{lemma}

\begin{proof}
By the definition of pushforward, the left-hand side of \eqref{first_claim_identity} is simply the integral $\int_{[0,\infty)} u\, (\zeta\circ f^{-1})(\dd u)$, which is equal to $\int_0^1 Q_{\zeta\circ f^{-1}}(z)\, \dd z$ by \eqref{well_known}.
But $Q_{\zeta\circ f^{-1}}(z)= f(Q_\zeta(z))$ by \eqref{push_quantiles}, and so we recover the right-hand side of \eqref{first_claim_identity}.
\end{proof}

We are now ready to state the key identity to be used in the proof of Proposition \ref{density_lemma}.

\begin{lemma} \label{key_identity_claim}
For any Borel probability measure $\zeta$ on $[0,1]$, any Lipschitz continuous, non-decreasing function $f\colon[0,1]\to[0,\infty)$, and any $q\in[0,1]$, we have
\eeq{ \label{by_parts_with_quantiles}
\int_{q}^1 \zeta\big([0,u]\big)f'(u)\ \dd u
= f(1)-\zeta\big([0,q]\big)f(q)-\int_{\zeta([0,q])}^1 f(Q_\zeta(z))\ \dd z.
}
In particular,
\eeq{ \label{by_parts_with_quantiles_0}
\int_0^1 \zeta\big([0,u]\big)f'(u)\ \dd u
= f(1) - \int_{0}^1 f(Q_\zeta(z))\ \dd z.
}
\end{lemma}

\begin{proof}
The first step is to integrate by parts:
\eeq{ \label{after_first_step}
\int_{q}^1 \zeta\big([0,u]\big)f'(u)\ \dd u
= f(1)-\zeta\big([0,q]\big)f(q)-\int_{(q,1]}f(u)\ \zeta(\dd u).
}
If $\zeta\big((q,1]\big)=0$, then the right-hand side of \eqref{after_first_step} is clearly equal to the right-hand side of \eqref{by_parts_with_quantiles}, as the integral in each expression is 0.
Otherwise, we consider the probability measure $\zeta_q$ on $[0,1]$ obtained by
\eq{ %\label{rescaled_measure}
\zeta_q(\cdot) \coloneqq \frac{\zeta\big(\cdot\cap(q,1]\big)}{\zeta\big((q,1]\big)}.
}
In this notation, we have
\eeq{ \label{changing_notation}
\int_{(q,1]}f(u)\ \zeta(\dd u) 
%= \zeta\big((q,1]\big)\int_\R f(u)\ \zeta_q(\dd u)
&\stackrefp{first_claim_identity}{=} \zeta\big((q,1]\big)\int_{[0,1]} f(u)\ \zeta_q(\dd u) \\
&\stackref{first_claim_identity}{=}
\zeta\big((q,1]\big)\int_0^1 f(Q_{\zeta_q}(z))\ \dd z.
}
From the definition of $\zeta_q$, it is clear that
\eq{
Q_{\zeta_q}(z) =Q_\zeta\Big(\zeta\big([0,q]\big)+z\cdot\zeta\big((q,1]\big)\Big) \quad \text{for all $z\in[0,1]$}.
}
So by a suitable substitution of variables, we obtain
\eq{
\zeta\big((q,1]\big)\int_0^1 f(Q_{\zeta_q}(z))\ \dd z
= \int_{\zeta([0,q])}^1
f(Q_\zeta(z))\ \dd z.
}
Using this last equality in \eqref{changing_notation}, we can again rewrite \eqref{after_first_step} to be \eqref{by_parts_with_quantiles}.
The special case \eqref{by_parts_with_quantiles_0} follows from the observation that
\eq{
\zeta(\{0\})f(0) = \int_0^{\zeta(\{0\})} f(Q_\zeta(z))\ \dd z.
\tag*{\qedhere}
}
\end{proof}

\begin{proof}[Proof of Proposition \ref{density_lemma}]
Given any $\eps_1>0$, let $K$ be an integer so large that 
\eeq{ \label{K_choice_1}
\frac{1}{K}\sum_{s\in\SSS}\frac{1}{\lambda^s} \leq \eps_1.
}
Given $\zeta$, choose a sequence 
\eeq{ \label{tilde_qs}
0= q_0 \leq  q_1 < \cdots <  q_k \leq  q_{k+1}=1
}
in the following manner:
\begin{itemize}
\item If $\zeta(\{0\}) > 0$, then set $q_1 = 0$.
%\item If $\zeta\big([0,1/K]\big)>0$, then set $ q_1=1/K$.
\item For $j\in\{1,\dots,K\}$, if $\zeta\big((\frac{j-1}{K},\frac{j}{K}]\big)>0$, then include $ q=j/K$ as one of the elements $ q_r$ of \eqref{tilde_qs}, with $r\in[k]$.
\end{itemize}
Once \eqref{tilde_qs} has been formed, define $m_r = \zeta\big([0, q_r]\big)$ for $r\in[k]$, and write $ q_r = j_r/K$.
The condition that $\zeta$ assign positive mass to the interval $(\frac{j_r-1}{K},\frac{j_r}{K}]$ %(or the fully closed interval $[0,1/K]$ if $j_1=1$) 
ensures that
\eq{ %\label{need_from_m}
0=m_0 < m_1 < \cdots < m_k = 1.
}
Furthermore, since all zero-mass intervals are excluded, %we trivially have
%\eeq{ \label{zero_mass_excluded}
%\zeta\Big(\Big(\frac{j-1}{K},\frac{j}{K}\Big]\Big) = 0 \quad \text{whenever $j_{r-1} < j < j_r$ or $j>j_k$}.
%}
%Consequently, it must be the case that
we have
\eeq{ \label{why_K}
 q_r - 1/K \leq Q_{\zeta}(z) \leq  q_r \quad \text{whenever $z\in(m_{r-1},m_r]$, $1\leq r\leq k$}.
}
Equivalently, %whenever $q_r < q_{r+1}-1/K$ (i.e.~$j_r < j_{r+1}-1$ and $\zeta\big((\frac{j-1}{K},\frac{j}{K}]\big)=0$ for $j\in\{j_r+1,\dots,j_{r+1}-1\}$), 
the following implication is true:
\eeq{ \label{why_K_equivalent}
q_r \leq u \leq q_{r+1}- 1/K \quad \implies \quad \zeta\big([0,u]\big)=\zeta\big([0,q_r])=m_r.
}
Now take the approximating measure to be
\eq{
\wt\zeta = \sum_{r=1}^k(m_r-m_{r-1})q_r.
}
As usual, given ${\Phi}$ we will write $\vc q_r=\Phi(q_r)$ so that for $z\in(m_{r-1},m_r]$, we have
\eeq{ \label{why_K_2}
\|{\Phi}(Q_{\zeta}(z))-\vc q_r\|
\stackref{lambda_av_consequence}{\leq} | Q_{\zeta}(z)-q_r|\sum_{s\in\SSS}\frac{1}{\lambda^s}
\stackref{why_K}{\leq} \frac{1}{K}\sum_{s\in\SSS}\frac{1}{\lambda^s}
\stackref{K_choice_1}{\leq} \eps_1.
}
Since $Q_{\wt\zeta}(z) = q_r$ for $z\in(m_{r-1},m_{r}]$, this inequality leads to
\eeq{ \label{K_choice_consequence}
\int_0^1\|\Phi(Q_\zeta(z))-{\Phi}(Q_{\wt\zeta}(z))\|_1\ \dd z
&= \sum_{r=1}^k \int_{m_{r-1}}^{m_r} \|{\Phi}(Q_{ \zeta}(z))-\vc q_r\|_1\ \dd z
\stackref{why_K_2}{\leq} \eps_1.
}
This completes the proof of \eqref{need_for_density_1}.

It remains to show \eqref{need_for_density_2}. 
Let $C\geq1$ be a large enough constant that for all $s\in\SSS$, $\vc q,\vc u\in[0,1]^\SSS$, we have
\eeq{ \label{C_choice}
|\xi^s(\vc q) - \xi^s(\vc u)|
\leq C\|\vc q-\vc u\|_1 \quad \text{and} \quad
|\theta(\vc q) - \theta(\vc u)|
\leq C\|\vc q-\vc u\|_1.
}
In order to distinguish between \eqref{ds_def} applied to $(\wt\zeta,{\Phi})$ as opposed to $(\zeta,\Phi)$, we will write
\eq{
\wt d^s(q) \coloneqq \int_q^1\wt \zeta\big([0,u]\big)(\xi^s\circ{\Phi})'(u)\ \dd u.
}
Applying the identity \eqref{by_parts_with_quantiles_0}, we have
\eeq{ \label{comparing_ds_at_zero}
|d^s(0)-\wt d^s(0)| 
&\stackrefp{C_choice}{=} \Big|\int_0^1\xi^s(\Phi(Q_\zeta(z)))\ \dd z-\int_0^1\xi^s({\Phi}(Q_{\wt\zeta}(z)))\ \dd z\Big| \\
&\stackref{C_choice}{\leq}
C\DD\big((\zeta,\Phi),(\wt\zeta,{\Phi})\big).
}
Given $\eps_2>0$, choose ${\vc b}_*$ such that
\eeq{ \label{tilde_choice_b}
A(\zeta,{\Phi},{\vc b}_*) \leq \PPP(\zeta,{\Phi})+\eps_2,
}
and define 
\eeq{ \label{alpha_choice}
\alpha\coloneqq(7\eps_2/8)\wedge\inf_{s\in\SSS}( b^s_*-d^s(0))>0.
}
Then let $\eps\in(0,\alpha/7)$ be so small that
\eeq{ \label{eps_alpha_choice}
\frac{1}{\alpha-7\eps}-\frac{1}{\alpha} \leq \frac{\eps_2}{2C\sum_{s\in\SSS}1/\lambda^s},
}
as well as
\eeq{ \label{eps_alpha_choice_2}
\frac{\xi^s(\vc 0)}{\alpha - \eps}-\frac{\xi^s(\vc 0)}{\alpha} \leq \frac{\eps_2}{2} \quad \text{for all $s\in\SSS$}.
}
Finally, with $\eps_1 = \eps/C$, take $K$ as above so that whenever $|q-u|\leq1/K$, we have
\eeq{ \label{summarizing_choices}
|\xi^s({\Phi}(q))-\xi^s({\Phi}(u))|
\stackref{C_choice,lambda_av_consequence}{\leq}
C|q-u|\sum_{s\in\SSS}\frac{1}{\lambda^s}
\stackref{K_choice_1}{\leq} C\eps_1=\eps.
}
In addition, because of \eqref{need_for_density_1}, the inequality \eqref{comparing_ds_at_zero} now reads as
\eeq{ \label{comparing_ds_0}
|d^s(0)-\wt d^s(0)| \leq C\eps_1=\eps.
}

\begin{claim} \label{close_claim}
If $\vc b$ is such that $b^s-d^s(0)\geq\alpha$ for each $s\in\SSS$, then
\eeq{ \label{comparing_when_close}
|A(\zeta,\Phi,\vc b) - A(\wt\zeta,{\Phi},\vc b)| \leq \eps_2.
}
\end{claim}

\begin{proofclaim}
A simple calculus exercise shows that for any $x_0\in(0,\alpha)$, we have
\eeq{ \label{simple_calculus}
\sup_{y\geq \alpha,x\in[-x_0,x_0]}\Big|\frac{1}{y-x}-\frac{1}{y}\Big| 
= \sup_{y\geq \alpha}\Big(\frac{1}{y-x_0}-\frac{1}{y}\Big) 
= \frac{1}{\alpha-x_0}-\frac{1}{\alpha}.
}
For instance, when $x_0=\eps$ and $b^s-d^s(0)\geq\alpha$, the inequality \eqref{comparing_ds_0} allows us to write
\eeq{ \label{controlling_extra_term}
\Big|\frac{\xi^s(\vc 0)}{b^s-d^s(0)}-\frac{\xi^s(\vc 0)}{b^s-\wt d^s(0)}\Big|
\stackref{simple_calculus}{\leq} \frac{\xi^s(\vc 0)}{\alpha - \eps}-\frac{\xi^s(\vc 0)}{\alpha}
\stackref{eps_alpha_choice_2}{\leq} \frac{\eps_2}{2}.
}
In addition, it follows from \eqref{lambda_av_consequence} and \eqref{C_choice} that if ${\Phi}$ is differentiable at $q\in(0,1)$, then
\eq{
(\xi^s\circ {\Phi})'(q) \leq C\sum_{s\in\SSS}1/\lambda^s.
}
Since $b^s-d^s(q)\geq b^s-d^s(0)\geq\alpha$ for any $q\in[0,1]$ and $s\in\SSS$, the two previous displays and \eqref{eps_alpha_choice} lead us to conclude
\eeq{ \label{uniform_inequality}
\sup_{x\in[-7\eps,7\eps]}\Big|\frac{(\xi^s\circ {\Phi})'(q)}{b^s-d^s(q)-x}-\frac{(\xi^s\circ {\Phi})'(q)}{b^s-d^s(q)}\Big| \leq \red{\frac{\eps_2}{2}} \quad \text{whenever ${\Phi}'(q)$ exists}.
}
The goal now is to set ourselves up to use this inequality to prove \eqref{comparing_when_close}.
%Since integration by parts gives
%\eq{
%|d^s(0) - \wt d^s(0)|
%&\stackrefp{C_choice}{=} \Big|\int_0^1\zeta\big([0,u]\big)(\xi^s\circ \Phi)'(u)\ \dd u
%- \int_0^1\wt \zeta\big([0,u]\big)(\xi^s\circ \wt{\Phi})'(u)\ \dd u\Big| \\
%&\stackrefp{C_choice}{=} \Big|-\int_0^1 (\xi^s\circ \Phi)(u)\ \zeta(\dd u) +
%\int_0^1 (\xi^s\circ \wt{\Phi})(u)\ \wt\zeta(\dd u)\Big| \\
%&\stackrefp{C_choice}{=} \Big|\int_0^1 (\xi^s\circ \Phi)(Q_\zeta(z))\ \dd z
%- \int_0^1 (\xi^s\circ \wt{\Phi})(Q_{\wt\zeta}(z))\ \dd z\Big| \\
%&\stackref{C_choice}{\leq} C\int_0^1\|L(Q_\zeta(z))-\wt{\Phi}(Q_{\wt\zeta}(z))\|_1\ \dd z
%\stackref{K_choice_consequence}{\leq}\eps,
%}
%this assumption also gives $b^s > d^s(0)+2\eps$.
%
%We will repeatedly use \eqref{summarizing_choices} in the remainder of the proof.
%Recall from \eqref{why_K} that 
%By construction of the sequence \eqref{tilde_qs}, the following implication is true:
%\eq{
%q_r \leq u < q_{r+1}- 1/K \quad \implies \quad \zeta\big([0,u]\big)=\zeta\big([0,q_r])=m_r.
%}

Thanks to \eqref{why_K_equivalent}, we have the following for $q\in[q_{r}, q_{r+1}]$:
\eq{
d^s(q)
&= d^s(q_{r+1})
+\int_{q}^{(q_{r+1}-\frac{1}{K})\red{\vee q}}\zeta\big([0,u]\big)(\xi^s\circ{\Phi})'(u)\ \dd u \\
&\phantom{=d^s(q_{r+1})}\; +\int_{(q_{r+1}-\frac{1}{K})\red{\vee q}}^{q_{r+1}}\zeta\big([0,u]\big)(\xi^s\circ{\Phi})'(u)\ \dd u\\
&= d^s(q_{r+1}) + m_r\big(\xi^s({\Phi}((q_{r+1}-{1/K})\red{\vee q}))-\xi^s({\Phi}(q))\big) \\
&\phantom{=d^s(q_{r+1})}\;+\int_{(q_{r+1}-\frac{1}{K})\red{\vee q}}^{q_{r+1}}\zeta\big([0,u]\big)(\xi^s\circ{\Phi})'(u)\ \dd u.
}
Now, it is immediate from \eqref{summarizing_choices} that
\eq{
|\xi^s({\Phi}(q_{r+1}))-\xi^s({\Phi}((q_{r+1}-1/K)\red{\vee q}))| \leq \eps.
}
In addition, by using the trivial inequality $0\leq\zeta\big([0,u]\big)\leq 1$, we obtain
\eq{
0\leq \int_{(q_{r+1}-\frac{1}{K})\red{\vee q}}^{q_{r+1}}\zeta\big([0,u]\big)(\xi^s\circ{\Phi})'(u)\ \dd u
\leq \xi^s({\Phi}(q_{r+1}))-\xi^s({\Phi}((q_{r+1}-1/K)\red{\vee q})) \leq \eps.
}
Since we defined $\vc q_{r+1}$ to be ${\Phi}(q_{r+1})$, the three previous displays together show
\eeq{ \label{integrate_to_find_mr}
\big|d^s(q) - d^s( q_{r+1})-m_r\big(\xi^s(\vc q_{r+1})-\xi^s({\Phi}(q))\big)\big| \leq 2\eps
\quad \text{for $q\in[q_r,q_{r+1}]$.}
}
If $0\leq r\leq k-1$, then recall from \eqref{why_K} that $q_{r+1}-1/K \leq Q_{\zeta}(m_{r+1})\leq  q_{r+1}$.
Therefore, by yet another application of \eqref{summarizing_choices}, we have
\eq{
0 \leq \xi^s(\vc q_{r+1})-\xi^s({\Phi}(Q_\zeta(m_{r+1}))) \leq \eps.
}
Using $Q_{\zeta}(m_{r+1})$ as the value of $q$ in \eqref{integrate_to_find_mr}, we now obtain the following special case:
\eq{
|d^s(Q_{\zeta}(m_{r+1}))-d^s(q_{r+1})|
%\leq 2\eps + m_r\big(\xi^s(\vc q_{r+1})-\xi^s({\Phi}(Q_{ \zeta}(m_{r+1})))\big)
\leq 3\eps, \quad 0\leq r\leq k-1.
}
Since $Q_{\wt\zeta}(m_{r+1}) = q_{r+1}$, we can employ Lemma \ref{key_identity_claim} to make the following comparison:
\eq{
&|d^s(Q_{\zeta}(m_{r+1}))-\wt d^s(q_{r+1})| \\
&\stackrefp{by_parts_with_quantiles}{=}
\Big|\int_{Q_\zeta(m_{r+1})}^1\zeta\big([0,u]\big)(\xi^s\circ \Phi)'(u)\ \dd u - \int_{Q_{\wt\zeta}(m_{r+1})}^1\wt\zeta\big([0,u]\big)(\xi^s\circ{\Phi})'(u)\ \dd u\Big| \\
&\stackref{by_parts_with_quantiles}{\leq} \Big|\int_{m_{r+1}}^1(\xi^s\circ \Phi)(Q_\zeta(z))\ \dd z-\int_{m_{r+1}}^1(\xi^s\circ {\Phi})(Q_{\wt \zeta}(z))\ \dd z\Big| \\
&\phantom{\stackref{by_parts_with_quantiles}{\leq}}
+m_{r+1}\big|\xi^s\big({\Phi}(Q_{\zeta}(m_{r+1}))\big)-\xi^s(\vc q_{r+1})\big| \\
&\stackref{C_choice}{\leq} C\int_{m_{r+1}}^1\|\Phi(Q_\zeta(z))-{\Phi}(Q_{\wt \zeta}(z))\|_1\ \dd z + \eps
\stackrel{\mbox{\footnotesize\eqref{K_choice_consequence}}}{\leq} 2\eps.
}
The two previous displays combine to show that
\eeq{ \label{close_on_qs}
|d^s( q_{r+1}) - \wt d^s(q_{r+1})| \leq 5\eps, \quad 0\leq r\leq k-1.
}
Of course, the same inequality holds trivially when $r=k$, since $d^s(1) = 0 = \wt d^s(1)$.
Putting together \eqref{integrate_to_find_mr} and \eqref{close_on_qs}, we find
\eq{
\big|d^s(q) - \wt d^s(q_{r+1}) - m_r\big(\xi^s(\vc q_{r+1})-\xi^s({\Phi}(q))\big)\big| \leq 7\eps \quad \text{for all $q\in[q_r, q_{r+1}]$, $0\leq r\leq k$}.
}
It thus follows from \eqref{uniform_inequality} that whenever ${\Phi}'(q)$ exists and $q\in[q_r, q_{r+1}]$, we have
\eeq{ \label{follows_from_uniform_inequality}
\bigg|\frac{(\xi^s\circ {\Phi})'(q)}{b^s-d^s(q)}
-\frac{(\xi^s\circ {\Phi})'(q)}{b^s-\wt d^s(q_{r+1})- m_{r}\big(\xi^s(\vc q_{r+1})-\xi^s({\Phi}(q))\big)}\bigg| \leq \red{\frac{\eps_2}{2}}.
}
Upon integration, this inequality yields the following for $r\geq1$:
\begin{subequations}
\label{inequality_on_intervals}
\eeq{
\bigg|\int_{ q_{r}}^{ q_{r+1}}
\frac{(\xi^s\circ {\Phi})'(q)}{b^s-d^s(q)}\ \dd q
- \frac{1}{m_{r}}\log\frac{b^s - \wt d^s(q_{r+1})}{b^s - \wt d^s(q_{r+1})-m_{r}\big(\xi^s(\vc q_{r+1})-\xi^s(\vc q_{r})\big)}\bigg|
\leq \red{\frac{\eps_2}{2}}( q_{r+1}- q_r).
}
When $r=0$, we have $m_0=0$, and so our conclusion from \eqref{follows_from_uniform_inequality} is instead
\eeq{
\bigg|\int_0^{ q_{1}}
\frac{(\xi^s\circ \wt{\Phi})'(q)}{b^s-d^s(q)}\ \dd q - \frac{\xi^s(\vc q_1)-\xi^s(\vc 0)}{b^s-\wt d^s(q_1)}\bigg|
\leq \red{\frac{\eps_2}{2}} q_1.
}
\end{subequations}
Upon recognizing (as we did in \eqref{A_def_0}) that the discrete nature of $(\wt\zeta,\Phi)$ implies
\eq{
\frac{1}{m_r}\log\frac{b^s - \wt d^s(q_{r+1})}{b^s - \wt d^s(q_{r+1})-m_{r}\big(\xi^s(\vc q_{r+1})-\xi^s(\vc q_{r})\big)}
&= \int_{q_r}^{q_{r+1}}\frac{(\xi^s\circ \Phi)'(q)}{b^s-\wt d^s(q)}\ \dd q \quad \text{for $r\geq1$}, \\
\text{and} \quad \frac{\xi^s(\vc q_1)-\xi^s(\vc 0)}{b^s-\wt d^s(q_1)} &= \int_{0}^{q_1}\frac{(\xi^s\circ \Phi)'(q)}{b^s-\wt d^s(q)}\ \dd q,
}
we can conclude from \eqref{inequality_on_intervals}  that
\eeq{ \label{for_density_final_1}
\bigg|\int_0^1 \frac{(\xi^s\circ {\Phi})'(q)}{b^s-d^s(q)}\ \dd q
- \int_0^1 \frac{(\xi^s\circ \Phi)'(q)}{b^s- \wt d^s(q)}\ \dd q \bigg| \leq \red{\frac{\eps_2}{2}}.
}
Finally, we apply Lemma \ref{key_identity_claim} once more, specifically \eqref{by_parts_with_quantiles_0}, to see that
\eeq{ \label{for_density_final_2}
&\Big|\int_0^1 \zeta\big([0,q]\big)(\theta\circ \Phi)'(q)\ \dd q-\int_0^1 \wt\zeta\big([0,q]\big)(\theta\circ {\Phi})'(q)\ \dd q\Big| \\
&\stackrefp{C_choice}{=} \Big|\int_{0}^1 \theta(\Phi(Q_{\zeta}(z)))\ \dd z-\int_0^1 \theta({\Phi}(Q_{\wt\zeta}(z)))\ \dd z\Big| \\
&\stackref{C_choice}{\leq} C\int_0^1\|\Phi(Q_{\zeta}(z))-{\Phi}(Q_{\wt\zeta}(z))\|_1\ \dd z 
\stackref{K_choice_consequence}{\leq} \eps \leq \eps_2.
}
Once we recall the definition \eqref{A_def} of $A(\zeta,\Phi,\vc b)$, 
the desired inequality \eqref{comparing_when_close} 
follows from \eqref{controlling_extra_term}, \eqref{for_density_final_1}, and \eqref{for_density_final_2}.
\end{proofclaim}

We now finish the proof of \eqref{need_for_density_2}.
One inequality is immediate from Claim \ref{close_claim}.
Since ${\vc b}_*$ from \eqref{tilde_choice_b} trivially satisfies the hypotheses of the claim (see \eqref{alpha_choice}), we have
\eeq{ \label{PPP_one_way}
\PPP(\wt\zeta,\Phi) \leq A(\wt\zeta,\Phi,{\vc b}_*)
\stackref{comparing_when_close}{\leq} A(\zeta,{\Phi},{\vc b}_*)+\eps_2
\stackref{tilde_choice_b}{\leq} \PPP(\zeta,{\Phi})+2\eps_2.
}
On the other hand, take any $\vc b$ satisfying $b^s>\wt d^s(0)$ for each $s\in\SSS$, and such that $A(\wt\zeta,\Phi,\vc b) \leq \PPP(\wt\zeta,\Phi)+\eps_2$.
%If $\vc b$ satisfies the hypotheses of Claim \ref{close_claim}, then---by the same logic as for $\wt{\vc b}$---we obtain
%\eq{
%\PPP(\wt\zeta,\wt{\Phi}) \leq \PPP(\zeta,\Phi)+2\eps_2.
%}
Unfortunately, $\vc b$ may not satisfy the hypotheses of Claim \ref{close_claim}. Nonetheless, we must have
\eq{
b^s- d^s(0) \stackref{comparing_ds_0}{\geq} b^s - \wt d^s(0) - \eps > -\eps.
}
Therefore, by simply increasing each coordinate of $\vc b$ by $\eps+\alpha$, we can obtain a vector that does satisfy the hypotheses of Claim \ref{close_claim}.
Indeed, by our choice of $\eps<\alpha/7$ and $\alpha\leq7\eps_2/8$, we have $-\eps\geq-\alpha/7=\alpha-8\alpha/7\geq\alpha-\eps_2$.
%$\eps_2 \geq 8\alpha/7 \geq \alpha+\eps$.
Consequently, the previous display leads to
\eq{
(b^s+\eps_2)-d^s(0)> \alpha.
}
Furthermore, it is straightforward to differentiate \eqref{A_def} to obtain
\eq{
\frac{\partial}{\partial b^s} A(\wt\zeta,\Phi,\vc b) =
\frac{\lambda^s}{2}\Big[1-\frac{1}{b^s}\red{-\frac{\xi^s(\vc0)}{(b^s-d^s(0))^2}}-\int_0^1\frac{(\xi^s\circ\Phi)'(q)}{(b^s-\wt d^s(q))^2}\ \dd q\Big]
\leq \frac{\lambda^s}{2}.
}
Consequently, the small change we make to $\vc b$ creates a correspondingly small change in $A(\wt\zeta,\Phi,\vc b)$:
\eeq{ \label{A_with_shift}
A(\wt\zeta,\Phi,\vc b+\eps_2\vc1)
-A(\wt\zeta,\Phi,\vc b)
\leq \eps_2\sum_{s\in\SSS}\frac{\lambda^s}{2}=\frac{\eps_2}{2}.
}
Applying Claim \ref{close_claim} to $\vc b+\eps_2\vc 1$, we thus obtain
\eq{
\PPP(\zeta,{\Phi})
\leq A(\zeta,{\Phi},\vc b+\eps_2\vc1)
&\stackref{comparing_when_close}{\leq} A(\wt\zeta,\Phi,\vc b+\eps_2\vc 1)+\eps_2 \\
&\stackref{A_with_shift}{\leq} A(\wt\zeta,\Phi,\vc b)+\frac{3\eps_2}{2}
\leq \PPP(\wt\zeta,\Phi)+\frac{5\eps_2}{2}.
}
Combining this inequality with \eqref{PPP_one_way} results in
\eq{
|\PPP(\zeta,\Phi)-\PPP(\wt\zeta,{\Phi})|\leq 5\eps_2/2.
}
Of course, replacing $\eps_2$ by $2\eps_2/5$ yields \eqref{need_for_density_2}.
\end{proof}

\section{Upper bound} \label{upper_section}
In this section we prove the following result.

\begin{prop}\label{prop:upper:bound}
For any $\vc\lambda$-admissible pair $(\zeta,\Phi)$, we have
\eeq{ \label{eq:prop:upper:bound}
    \limsup_{N\to\infty}\E F_N \leq \PPP(\zeta,\Phi).
}
\end{prop}

The proof will require that we introduce in Section \ref{perturb_section} a perturbed version of the Hamiltonian $H_{N}$ from \eqref{mixed_H_def}.
This is to guarantee that Talagrand's positivity principle holds, a fact we show in Section \ref{positivity_section}.
In turn, this principle is critical to controlling how the free energy changes along a Guerra-type interpolation of the Hamiltonian, which we perform in Section \ref{interpolation_section}.
This interpolation ultimately proves Proposition \ref{prop:upper:bound}.

In defining $H_N^\pert$ via \eqref{Hpert_def}, we are able to ensure that the Ghirlanda--Guerra identities hold in the large-$N$ limit.
In addition to implying the positivity principle, these identities will be needed in Section \ref{final_lower_section} for the reasons discussed in Section \ref{sketch_section}.
Therefore, the definitions made in Section \ref{perturb_section} will be used throughout the rest of the paper.

\subsection{Perturbing the Hamiltonian} \label{perturb_section}
%Once $N$ is sent to infinity, it will become necessary---for a variety of reasons---that the limiting Gibbs measure satisfies the Ghirlanda--Guerra identities discussed in Section REF.
%By adopting the perturbation techniques found in \cite{panchenko13a} and further developed in \cite{panchenko15}, we can guarantee the validity of these identities.
%In this section, however, we simply define the perturbation that will ultimately prove successful.

We adopt the multi-species perturbation technique developed in \cite{panchenko15}.
For $\vc w\in[0,1]^\SSS$, define the following linear combination of the entries in the vector $\vc R(\sigma,\sigma')$ from \eqref{overlap_def}:
\eeq{ \label{warped_overlap_N}
R^{\vc w}(\sigma,\sigma') \coloneqq \sum_{s\in\SSS} \lambda^s(N)w^s R^s(\sigma,\sigma'), \quad \vc w\in[0,1]^\SSS.
}
Let $\WWW = \{\vc w_{1},\vc w_{2},\dots\}$ be a countable, dense subset of $[0,1]^\SSS$ which contains the standard basis vectors of $\R^\SSS$.
To avoid divide-by-zero pathologies, assume that $\vc w_q\neq\vc 0$ for all $q$.
For each $\vc w_q = (w_{q}^s)_{s\in\SSS}$ and $\vct i\in [N]^p$, we will write $w_q^{s(\vct i)} \coloneqq w_{q}^{s(i_1)}\cdots w_{q}^{s(i_p)}$.
Now consider a collection of i.i.d.~standard Gaussian random variables $\{g_{\vct i,q} :\, \vct i\in[N]^p, p\geq1,q\geq1\}$ that is independent of all $ g_{\vct i}$'s from \eqref{p_Hamiltonian}.
With $u_{p,q}$ allowed to be any number in $[0,3]$, we define
\eeq{ \label{Hpert_def}
H_{N}^\pert(\sigma) \coloneqq
\sum_{p,q\geq1}u_{p,q}H_{N,p,q}^\pert(\sigma), \quad \text{where} \quad
H_{N,p,q}^\pert(\sigma)\coloneqq \frac{2^{-(p+q)}}{N^{(p-1)/2}}\sum_{\vct i\in[N]^p}\sqrt{w_{q}^{s(\vct i)}}g_{\vct i,q}\sigma_{\vct i}.
}
%where $\{g_{\vct i,q} : \vct i\in[N]^p, p\geq1,q\geq1\}$ is a collection of i.i.d.~standard Gaussian random variables that is independent of all $ g_{\vct i}$'s from \eqref{p_Hamiltonian}, and each $u_{p,q}$ is allowed to be any number in $[0,3]$.
In the next section we will select the $u_{p,q}$ parameters randomly, in which case $\E_u$ will denote expectation with respect to the product measure $\P_u$ under which each $u_{p,q}$ is a uniform random variable on $[1,2]$, independent of all other variables.
We will continue to write $\E$ for expectation over all Gaussian processes (and Poisson--Dirichlet cascades whenever they are present) with fixed $u = (u_{p,q})_{p,q\geq1}$.

By direct calculation, we have
\eeq{ \label{pert_pq_cov}
\E[H_{N,p,q}^\pert(\sigma)H_{N,p,q}^\pert(\sigma')] 
= N4^{-(p+q)}(R^{\vc w_q}(\sigma,\sigma'))^p.
}
The covariance structure of $H_N^\pert$ is thus given by
\eeq{ \label{pert_cov}
\E[H_N^\pert(\sigma) H_N^\pert(\sigma')] = N\xi_N^\pert(\vc R(\sigma,\sigma')),
}
where $\xi_N^\pert$ is analogous to \eqref{xi_N_def}:
\eeq{ \label{xi_pert_def}
\xi_N^\pert(\vc x) \coloneqq \sum_{p,q\geq1} 
\frac{u_{p,q}^2}{4^{p+q}}\sum_{\vct s\in\SSS^{p}}\lambda^{\vct s}(N)w_{q}^{\vct s}x^{\vct s}.
}
A perturbed spin glass model is now constructed from the Hamiltonian
\eeq{ \label{barH_def}
\bar H_N(\sigma) \coloneqq H_N(\sigma) + c_NH_N^\pert(\sigma),
}
where $c_N$ is some constant.
To ensure that the perturbation does not change the limiting free energy, we will ultimately send to $c_N$ to $0$ as $N\to\infty$.
The following simple result is analogous to \cite[Lem.~12.2.1]{talagrand11b}.

\begin{lemma} \label{no_change_lemma}
Define the perturbed partition function and free energy:
\eeq{ \label{perturbed_quantities_def}
\bar Z_N \coloneqq \int_{\T_N}\exp(\bar H_N(\sigma))\ \tau_N(\dd\sigma), \qquad \bar F_N \coloneqq \frac{1}{N}\log\bar Z_N.
}
If $u_{p,q}\in[0,3]$ for all $p,q$, then we have
\eeq{ \label{perturbation_control}
\E F_N \leq \E \bar F_N \leq \E F_N + c_N^2/2.
}
\end{lemma}

\begin{proof}
Apply Lemma \ref{general_energy_lemma} with the following parameters:
\begin{itemize}
    \item In \eqref{sigma_tau_H}, take $(\Sigma,\tau) = (\T_N,\tau_N)$ and $H = H_N$.
    \item In \eqref{Hu_def}, take $(h_i)_{i\geq1} = (H_{N,p,q}^\pert)_{p,q\geq1}$, $c = c_N$ so that $h_u = H_N^\pert$, $H_u = \bar H_{N}$.
\end{itemize}
In this case, the constant $\varsigma^2(u)$ from \eqref{varsigma_assumption} satisfies
\eeq{ \label{xi_pert_bound}
\varsigma^2(u) = \frac{1}{N}\E[H_N^\pert(\sigma)^2]
\stackref{pert_cov}{=} \xi_N^\pert(\vc 1) \leq \sum_{p\geq1}\sum_{q\geq1}\frac{9}{4^{p+q}} = 1.
}
Therefore, \eqref{perturbation_control} follows from \eqref{general_perturbation_control}.
\end{proof}

\subsection{Multi-species positivity principle} \label{positivity_section}
Unlike the lower bound \eqref{parisi_lower}, Proposition \ref{prop:upper:bound} requires the convexity assumption \eqref{xi_convex}.
But notice that we only demand convexity of $\xi$ on $[0,1]^\SSS$ as oppposed to all of $[-1,1]^\SSS$.
This will be sufficient because of the following multi-species version of Talagrand's positivity principle \cite[Thm.~14.12.1]{talagrand11b} (see also \cite[Thm.~3.4]{panchenko13a}).

%Let $\WWW = \{\vc w^{(1)},\vc w^{(2)},\dots\}$ be a countable and dense subset of $[0,1]^\SSS$, which contains the standard basis of $\R^\SSS$.
%For $\vc u_q = (u_q^s)_{s\in\SSS}\in\WWW$ and $\vct i\in \llbrack 1,N\rrbrack^p$, we will write $u_{q}^{s(\vct i)} \coloneqq u_{q}^{s(i_1)}\cdots u_{q}^{s(i_p)}$.
%For each pair $p,q\geq1$, define
%\eq{
%H_{N}^\pert(\sigma) \coloneqq \sum_{p\geq1}\sum_{q\geq1}\frac{u_{p,q}}{2^{p+q}}H_{N,p,q}^\pert(\sigma),\quad\textnormal{where}\quad H_{N,p,q}^\pert(\sigma) \coloneqq \frac{1}{N^{(p-1)/2}}\sum_{\vct i\in\llbrack 1,N\rrbrack^p} \sqrt{u_{q}^{s(\vct i)}}g_{\vct i,q}\sigma_{\vct i}
%}
%and $\{g_{\vct i,q} : \vct i\in\llbrack1,N\rrbrack^p,p\geq1,q\geq1\}$ is a collection of i.i.d.~standard Gaussian random variables that is independent of all $ g_{\vct i}$'s from \eqref{p_Hamiltonian}.

\begin{lemma}\label{lem:positivity:principle}
For a non-random Hamiltonian $H$ on $\T_N$ satisfying
\eeq{ \label{integrable_to_satisfy}
\int_{\T_N}\exp |H(\sigma)|\ \tau_N(\dd\sigma) < \infty,
}
consider the perturbed Hamiltonian $\bar{H}(\sigma)\coloneqq H(\sigma)+c_N H^\pert_{N}(\sigma)$.
Denote the corresponding Gibbs measure by $\bar{G}_N$. 
If $c_N = N^{-\varpi}$ for some $\varpi<1/2$, then for any $\eps>0$, we have
\eeq{ \label{positivity_limit}
    \lim_{N\to\infty} \sup_{H(\sigma)}\E_{u} \E \bar{G}_N^{\otimes 2}\Big(\bigcup_{s\in\SSS}\{\RR^s(\sigma^1,\sigma^2)\leq -\eps\}\Big)=0,
}
where the supremum is over all measurable functions $H:\T_N\to\R$ satisfying \eqref{integrable_to_satisfy}.
\end{lemma}

\begin{proof}
As usual we will write $\langle\cdot\rangle$ to denote expectation according to $\bar G_N$, we take $(\sigma^\ell)_{\ell\geq1}$ to be i.i.d.~samples from $\bar G_N$, and $\RR^s_{\ell,\ell'}=R^s(\sigma^\ell,\sigma^{\ell'})$.
Since $\SSS$ is finite, it suffices to prove that for each $s\in\SSS$ we have
\eeq{ \label{eq:goal:lem:positivity:principle}
    \lim_{N\to\infty} \sup_{H(\sigma)}\E_{u} \E \bar{G}_N^{\otimes 2}\big(\{\RR^s_{1,2}\leq -\eps\}\big)=0.
}
The argument for \eqref{eq:goal:lem:positivity:principle}  relies on first establishing (a subset of) the Ghirlanda--Guerra identities.
Recall the notation $\vc\RR_{\ell,\ell'} = (\RR_{\ell,\ell'}^s)_{s\in\SSS}$ and $\vc\RR^n = (\vc\RR_{\ell,\ell'})_{\ell,\ell'\in[n]}$.

\begin{claim} \label{GG_single_claim}
For any bounded measurable function $f = f(\vc\RR^n)$ and any continuous function $\psi:[-1,1]\to\R$, define the quantity
\eeq{ \label{gg_restrict_s}
    \Delta^{s}(f,n,\psi)\coloneqq \Big| \E \langle f \psi(\RR^{s}_{1,n+1})\rangle-\frac{1}{n}\E\langle f \rangle \E \langle \psi(\RR^{s}_{1,2}) \rangle -\frac{1}{n}\sum_{\ell=2}^{n} \E \langle f\psi(\RR^{s}_{1,\ell})\rangle \Big|.
}
We then have
\eeq{ \label{eq:GG:for:positivity}
\lim_{N\to\infty}\sup_{H(\sigma)}\E_u\Delta^s(f,n,\psi) = 0.
}
\end{claim}

\begin{proofclaim}
Given any non-random $H$ satisfying \eqref{integrable_to_satisfy}, we will apply Theorem \ref{general_GG_thm} with the following parameters:
\begin{itemize}
    \item In \eqref{sigma_tau_H}, take $(\Sigma,\tau) = (\T_N,\tau_N)$.
    \item In \eqref{Hu_def}, take $(h_i)_{i\geq1} = (H_{N,p,q}^\pert)_{p,q\geq1}$, $c = c_N$ so that $h_u = H_N^\pert$, $H_u = \bar H$.
\end{itemize}
By \eqref{pert_cov}, the constant $\varsigma^2(u)$ from \eqref{varsigma_assumption} is equal to $\xi_N^\pert(\vc 1)$, which is at most $1$ by \eqref{xi_pert_bound}.
Consequently, Lemma \ref{general_concentration_lemma} and specifically \eqref{general_expectation_outcome} yields the following bound on the quantity defined in \eqref{vartheta_def}:
\eeq{ \label{vartheta_ineq}
\vartheta\leq 2\sqrt{\pi N^{1-2\varpi}} \leq 4N^{1/2-\varpi}.
}
Now fix $s\in\SSS$.
Recall that we chose the set $\WWW$ so that there is some $q$ for which $ \vc w_q$ has entries all equal to $0$ except for $1$ in the $s$-coordinate.
For this value of $q$, we have $R^{\vc w_q}(\sigma,\sigma') = \lambda^s(N)R^s(\sigma,\sigma')$, and so \eqref{pert_pq_cov} gives
\eq{
\frac{1}{N}\E[H_{N,p,q}^\pert(\sigma)H_{N,p,q}^\pert(\sigma')]
= 4^{-(p+q)}\lambda^s(N)R^s(\sigma,\sigma')^p.
}
Therefore, for all $N$ large enough that 
\eq{
N^{\varpi}\sqrt{\frac{4^{p+q}\vartheta}{\lambda^s(N)N}} \stackref{vartheta_ineq}{\leq} \frac{2^{p+q+1}}{\sqrt{\lambda^s(N)}}N^{\varpi/2-1/4} < 1,
} 
the inequality \eqref{general_GG_claim} in the present setting 
%(with $\psi$ equal to the polynomial $x\mapsto x^p$) 
reads as 
\eeq{ \label{verified_for_monomial}
\E_u\Delta^s(f,n,x\mapsto x^p) \leq 24\|f\|_\infty\frac{2^{p+q}}{\sqrt{\lambda^s(N)}}n^{-1}N^{\varpi-1/2}\big(1+2N^{1/4-\varpi/2}\big).
}
%where we have used \eqref{vartheta_ineq}.
Note that the right-hand side has no dependence on $H$, and tends to $0$ as $N\to\infty$ so long as $\varpi<1/2$.
That is, we have proved \eqref{eq:GG:for:positivity} when $\psi(x) = x^p$.

For general $\psi$, we approximate by polynomials.
Indeed, given any $\eps>0$, by Stone--Weierstrass we can find a polynomial $\wt\psi(x) = \sum_{p=1}^L a_px^p$ such that
\eq{
|\psi(x)-\wt\psi(x)| \leq \eps \quad \text{for all $x\in[-1,1]$}.
}
Simply from examining the definition \eqref{gg_restrict_s}, it follows that 
\eeq{ \label{approx_by_poly}
|\Delta^s(f,n,\psi)-\Delta^s(f,n,\wt\psi)| \leq 2\eps\|f\|_\infty.
}
Since $\E_u\Delta^s(f,n,\wt\psi)\leq\sum_{p=1}^L|a_p|\Delta^s(f,n,x\mapsto x^p)$, we have $\Delta^s(f,n,\wt\psi)\to0$ by \eqref{verified_for_monomial}, uniformly in $H$.
Consequently, \eqref{approx_by_poly} leads to
\eq{
\limsup_{N\to\infty}\sup_{H(\sigma)}\E_u\Delta(f,n,\psi) \leq 2\eps\|f\|_\infty.
}
As $\eps$ is arbitrary, we conclude \eqref{eq:GG:for:positivity}.
\end{proofclaim}

With \eqref{eq:GG:for:positivity} in hand, one can proceed exactly as in \cite[Thm.~3.4]{panchenko13a} to prove \eqref{eq:goal:lem:positivity:principle}.
\end{proof}

\subsection{Guerra interpolation: proof of Proposition \ref{prop:upper:bound}} \label{interpolation_section}
By Proposition \ref{density_lemma}, specifically \eqref{need_for_density_2}, it suffices to prove \eqref{eq:prop:upper:bound} when $\zeta$ has finite support.
So let us consider any $\vc\lambda$-admissible pair $(\zeta,\Phi)$ such that $\zeta$ has finite support.
That is, $\zeta$ is of the form \eqref{zeta_def} for
some sequences $(m_r)_{0\leq r\leq k}$ and $(q_r)_{0\leq r\leq k+1}$ of the form \eqref{m_sequence} and \eqref{q_sequence}.
Using the same shorthand as in Section \ref{restrict_prelimit}, we write
\eq{ %\label{bold_q_sequence}
\vc q_r = \Phi(q_r), \quad 0\leq r\leq k+1.
}

%Throughout this subsection, we fix $\tilde{\zeta}_0\in(0,m_1), \tilde{\zeta}_{k+1}\in (m_k,1)$, and consider
%\begin{equation}\label{eq:def:fop:new}
    %0\eqqcolon \tilde{\zeta}_{-1}<\tilde{\zeta}_0<\tilde{\zeta}_1<\cdots<\tilde{\zeta}_k<\tilde{\zeta}_{k+1}<\tilde{\zeta}_{k+2}\coloneqq 1,\quad\textnormal{where}\quad \tilde{\zeta}_j=m_j,1\leq j \leq k.
%\end{equation}
%We will later send $\tilde{\zeta}_0\to 0$ and $\tilde{\zeta}_{k+1}\to 1$. 
As defined in Section \ref{cascade_review}, let $(v_\alpha)_{\alpha \in \N^{k-1}}$ be the weights of the Poisson--Dirichlet cascade corresponding to the sequence \eqref{m_sequence}. 
%For $\alpha^1,\alpha^2\in \N^{k+2}$, we denote
%\begin{equation}
    %\alpha^1\wedge \alpha^2 \coloneqq \min\{0\leq j \leq k+2:\alpha^{1}_1=\alpha^2_{1},...,\alpha^{1}_j=\alpha^{2}_j,\alpha_{j+1}\neq \alpha^{2}_{j+1}\}.
%\end{equation}
%Since $\bx \to \xi^s(\bx)$ in \eqref{gamma_def} is coordinate-wise increasing function for $\bx\geq 0$, we can consider Gaussian processes $C^{s}(\alpha)$ for $s\in \llbrack 1,K \rrbrack$ and $D(\alpha)$, indexed by $\alpha \in \N^{k+2}$, with covariances
%\begin{equation}
    %\E C^{s}(\alpha^{1})C^{s}(\alpha^{2})=\xi^s(\bq_{\alpha^1\wedge \alpha^2}) \quad\textnormal{and}\quad \E D(\alpha^1)D(\alpha^2)=\theta(\bq_{\alpha^1 \wedge \alpha^2}).
%\end{equation}
Similar to \eqref{XY_def_rpc}, let $(X_i)_{i\in[N]}$ and $Y$ be centered Gaussian processes on $\N^{k-1}$ whose covariance structures are given by
\eeq{ \label{XY_def_again}
\E[X_{i}(\alpha)X_{i'}(\alpha')] &= 
\one_{\{i=i'\}}\xi^s(\vc q_{r(\alpha,\alpha')})
\quad \text{for $i\in\II^s$}, \\
\E[Y(\alpha)Y(\alpha')] &= 
\theta(\vc q_{r(\alpha,\alpha')}).
}
Assume that these processes are independent of each other, of the Poisson--Dirichlet cascade, and of the Gaussian disorder defining $H_N$ and $H_N^\pert$.
We then define the following interpolating Hamiltonian on $\T_N\times\N^{k-1}$:
\eq{ %\label{interpolating_with_alpha}
    \Hb_{N,t}(\sigma,\alpha) &\coloneqq
    \sqrt{1-t}\big(H_N(\sigma)+\sqrt{N}Y(\alpha)\big)
    +\sqrt{t}\sum_{i=1}^N\sigma_iX_i(\alpha)
    + c_NH_N^\pert(\sigma), \quad t\in[0,1].
    %\phantom{+\sqrt{1-t}}\llap{$\sqrt{t}$}\big(H_N(\sigma)+\sqrt{N}Y(\alpha)+\sqrt{N}\eta(\theta(\vc 1)-\theta(\vc q_k)\big)\\
    %&\phantom{\coloneqq}+\sqrt{1-t}\sum_{s\in\SSS}\sum_{i\in\II^s}\sigma_i\big(X_i(\alpha)+\eta_i(\xi^s(\vc 1)-\xi^s(\vc q_k))\big) + c_NH_N^\pert(\sigma), \quad t\in[0,1].
}
%To this Hamiltonian one can associate a free energy
%\eeq{ \label{associated_energy}
%\vphi_N(t) \coloneqq \frac{1}{N}\log \sum_{\alpha \in \N^{k-1}}v_{\alpha}\int_{\T_N}\exp \Hb_{N,t}(\sigma,\alpha)\ \tau_N(\dd\sigma).
%}
We assume here that $c_N = N^{-\varpi}$ for some $\varpi\in(0,1/4)$.
Let us consider the associated free energy,
\eq{ %\label{vphi_free}
\vphi_N(t) \coloneqq \log \int_{\T_N}\sum_{\alpha\in\N^{k-1}}v_\alpha\exp \Hb_{N,t}(\sigma,\alpha)\ \tau_N(\dd\sigma).
}
Upon defining 
\eeq{ \label{phi_free}
\phi_N(t)\coloneqq \E_u\E\vphi_N(t) +\frac{N}{2}\Big[t\sum_{s\in\SSS}\lambda^s(\xi^s(\vc 1)-\xi^s(\vc q_k))+(1-t)(\theta(\vc1)-\theta(\vc q_k))\Big],
}
we have the following estimate.

%function
%\eeq{ \label{derivative_formula}
%\phi_N(t) &\coloneqq \E_u\E\log \sum_{\alpha \in \N^{k-1}}v_{\alpha}\int_{\T_N}\exp \Hb_{N,t}(\sigma,\alpha)\ \tau_N(\dd\sigma)
%\\
%&\phantom{\coloneqq}+ \frac{t}{2}(\theta(\vc1)-\theta(\vc %q_k))+\frac{1-t}{2}\sum_{s\in\SSS}\lambda^s(\xi^s(\vc 1)-\xi^s(\vc q_k)). 
%}

%By Jensen's inequality, it is straightforward to see
%\begin{multline}\label{eq:bound:perturbed:interpolating}
   %0\leq \varphi_N(t)- \frac{1}{N}\E \log \sum_{\alpha \in \N^{k+2}}v_{\alpha}\int_{\T_N}\exp\left(H_{N,t}(\sigma,\alpha)\right)\tau_N(\dd\sigma)\\
   %\leq \frac{N^{-1/2}}{2}\sum_{s\in\SSS}\sum_{p,q\geq 1}\frac{\E u_{p,q}^2}{4^{p+q}}w^s_q\lambda^s(N)=o_N(1).
%\end{multline}
%Guerra's replica symmetry breaking bound \cite{guerra03} is given as follows.
\begin{claim}\label{lem:guerra:interpolation}
The following inequality holds:
\eeq{ \label{guerra_ineq}
\limsup_{N\to\infty}\frac{1}{N}\sup_{t\in [0,1]}\phi_N^\prime(t)\geq 0.
}
\end{claim}

\begin{proofclaim}
Define the Gibbs measure associated to the Hamiltonian $\Hb_{N,t}(\sigma,\alpha)$:
\eeq{ \label{two_var_gibbs}
   \G_{N,t}(\dd\sigma, \alpha) \coloneqq \frac{1}{\exp \vphi_N(t)}v_{\alpha}\exp \Hb_{N,t}(\sigma,\alpha)\ \tau_N(\dd\sigma).
}
Denote by $\langle \cdot \rangle_{t}$ the expectation according to $\G_{N,t}$.
By direct calculation we have
\eeq{ \label{derivative_formula}
\phi'_N(t)=\E_u\E \Big\langle \frac{\dd \Hb_{N,t}(\sigma,\alpha)}{\dd t}\Big\rangle_{t}
+ \frac{N}{2}\Big[\sum_{s\in\SSS}\lambda^s(\xi^s(\vc1)-\xi^s(\vc q_k))-\theta(\vc1)+\theta(\vc q_k)\Big].
}
By recalling the definitions of $\xi^s$ and $\theta$ from \eqref{gamma_def} and \eqref{theta_def}, it is trivial to check that \eqref{derivative_formula} can be rewritten as
\begin{subequations}
\label{derivative_manipulations}
\eeq{ \label{derivative_formula_2}
\phi'_N(t)=\E_u\E \Big\langle \frac{\dd \Hb_{N,t}(\sigma,\alpha)}{\dd t}\Big\rangle_{t}
+ N\CC(\vc 1,\vc q_k),
}
where
\eeq{ \label{CC_def}
\CC(\vc x,\vc y)\coloneqq \frac{1}{2}\big(\xi(\vc x)-\xi(\vc y)-(\vc x-\vc y)\cdot\nabla\xi(\vc y)\big), \quad \vc x,\vc y\in[-1,1]^\SSS.
}
Note that \eqref{xi_convex} implies 
\eeq{ \label{xi_convex_consequence}
\CC(\vc x, \vc y)\geq 0 \quad \text{for $\vc x, \vc y\in[0,1]^\SSS$}.
}
Next consider the Gibbs average in \eqref{derivative_formula_2}.
In light of \eqref{xi_N_def} and \eqref{XY_def_again}, Gaussian integration by parts (see \cite[Lem.~1.1]{panchenko13a}) shows that
\eeq{ \label{from_another_gibp}
\E \Big\langle \frac{\dd \Hb_{N,t}(\sigma,\alpha)}{\dd t}\Big\rangle_{t}
= -N\E\Big\langle \CC_N(\vc 1,\vc 1)-\CC_N(\vc R(\sigma^1,\sigma^2), \bq_{r(\alpha^{1},\alpha^{2})})\Big\rangle_{t},
}
where $(\sigma^1,\alpha^1),(\sigma^2,\alpha^2)$ are independent samples from $\G_{N,t}$, and
\eq{
\CC_N(\vc x,\vc y)\coloneqq \frac{1}{2}\Big(\xi_N(\vc x)+\theta(\vc y)-\sum_{s\in\SSS}\frac{N^s}{N}x^s\xi^s(\vc y)\Big).
}
Furthermore, by substituting $\xi_N\mapsto\xi$ and $N^s/N\mapsto\lambda^s$ in this definition, we recover the function $\CC$ from \eqref{CC_def} while incurring negligible change:
\eeq{\label{eq:substitute:covariance}
    &\lim_{N\to\infty}\sup_{\vc x,\vc y \in [-1,1]^\SSS}\Big| \CC_N(\vc x, \vc y)-\CC(\vc x,\vc y)\Big|=0.
}
\end{subequations}
%
%\eeq{ \label{eq:derivative:guerra}
%\phi'_N(t)%=\frac{1}{N}\E_u\E \Big\langle \frac{\dd \Hb_{N,t}(\sigma,\alpha)}{\dd t}\Big\rangle_{t}
%=\E_u\E\Big\langle \CC_N(\vc 1,\vc 1)-\CC_N(\vc R(\sigma^1,\sigma^2), \bq_{r(\alpha^{1},\alpha^{2})})\Big\rangle_{t} - \CC(\vc 1,\vc q_k),
%(\xi(\vc 1)-\xi(\vc q_k)-(\vc1-\vc q_k)\cdot\nabla\xi(\vc q_k)),
%}
%
Since $\CC(\vc x,\vc x)=0$, the cumulative outcome of \eqref{derivative_manipulations} is that
\eq{
    \frac{1}{N}\phi^\prime_N(t)\geq\E_u\E \big\langle \CC(\vc R(\sigma^1,\sigma^2), \bq_{r(\alpha^{1},\alpha^{2})}) \big\rangle_{t} -o(1),
}
where $o(1)$ tends to $0$ as $N\to\infty$, uniformly in $t$.
To set up an application of Lemma \ref{lem:positivity:principle}, we observe the trivial inequality
\eeq{ \label{trivial_with_theta_delta}
\frac{1}{N}\phi_N'(t) &\geq \delta(\eps)-\|\CC\|_\infty \E_u\E\G_{N,t}^{\otimes 2}\Big(\bigcup_{s\in\SSS}\{R^s(\sigma^1,\sigma^2)\leq -\eps\}\Big) - o(1), \quad \text{where} \quad \\
\delta(\eps) &\coloneqq \inf\{\CC(\vc x,\vc y):\, \vc x\in[-\eps,1]^\SSS, \vc y\in[0,1]^\SSS\}, \\
\|\CC\|_{\infty} &\coloneqq \sup\{\CC(\vc x, \vc y):\,  \vc x\in [-1,1]^\SSS, \vc y \in [0,1]^\SSS\}.
}
Next we write the marginal of $\G_{N,t}$ on $\T_N$ as a Gibbs measure of form required by Lemma \ref{lem:positivity:principle}.
Indeed, if we define
\eq{
H_t(\sigma) \coloneqq \log\sum_{\alpha\in\N^{k-1}}v_\alpha\exp\Big(\sqrt{1-t}\big(H_N(\sigma)+\sqrt{N}Y(\alpha)\big)+\sqrt{t}\sum_{i=1}^{N}\sigma_iX_i(\alpha)\Big),
}
then we have the marginal
\eq{
\bar\G_{N,t}(\dd\sigma) \coloneqq \sum_{\alpha\in\N^{k-1}}\G_{N,t}(\dd\sigma,\alpha)
\stackref{two_var_gibbs}{=} \frac{1}{\exp\vphi_N(t)}\exp\big(H_t(\sigma)+c_NH_N^\pert(\sigma)\big)\ \tau_N(\dd\sigma).
}
Although $H_t$ is random, this randomness is independent of $H_N^\pert$.
Therefore, if we denote by $\E_1$ and $\E_2$ the expectations over $H_t$ and $H_N^\pert$ respectively, then
\eq{
\E_u\E\G_{N,t}^{\otimes 2}\Big(\bigcup_{s\in\SSS}\{R^s(\sigma^1,\sigma^2)\leq -\eps\}\Big)
&= \E_1\E_u\E_2\bar\G_{N,t}^{\otimes 2}\Big(\bigcup_{s\in\SSS}\{R^s(\sigma^1,\sigma^2)\leq -\eps\}\Big) \\
&\leq \sup_{t\in[0,1]}\E_u\E_2\bar  \G_{N,t}^{\otimes2}\Big(\bigcup_{s\in\SSS}\{R^s(\sigma^1,\sigma^2)\leq -\eps\}\Big)
\stackref{positivity_limit}{=} o(1).
}
As the final line is uniform in $t$, applying this estimate to \eqref{trivial_with_theta_delta} results in
\eq{
\limsup_{N\to\infty}\frac{1}{N}\sup_{t\in[0,1]}\phi_N'(t) \geq \delta(\eps) \quad \text{for any $\eps>0$}.
}
%Finally, in light of the last equality of \eqref{eq:substitute:covariance}, the convexity assumption \eqref{xi_convex} implies $\CC(\vc x, \vc y)\geq 0$ for $\vc x, \vc y\geq 0$. 
Finally, because of \eqref{xi_convex_consequence}
we have $\delta(\eps)\to0$ as $\eps\to0$.
\end{proofclaim}

We now compute $\phi_N(0)$ and $\phi_N(1)$. 
When $t = 0$, the terms involving $\sigma$ are decoupled from those involving $\alpha$, and by simple algebra \eqref{phi_free} becomes
\eeq{ \label{first_fnc_1}
\phi_N(0) 
%&= \frac{1}{N}\log \sum_{\alpha\in\N^{k-1}}v_\alpha \int_{\T_N}\exp\Big(\bar H_N(\sigma)+\sqrt{N}Y(\alpha)\Big)\ \tau_N(\dd\sigma) \\
&= \E_u(\E\log \bar Z_N) + \E\log\sum_{\alpha\in\N^{k-1}}v_\alpha\exp\big(\sqrt{N}Y(\alpha)\big) + \frac{N}{2}\big(\theta(\vc 1)-\theta(\vc q_k)\big).
}
Notice that the last two terms on the right-hand side are exactly of the form \eqref{psi_2_redef}, except here $N$ replaces $M$.
We computed the expectation of this expression in \eqref{psi_2_after_magical}:
\eq{ %\label{eq:compute:phi:one}
    \E \log \sum_{\alpha\in \N^{k-1}}v_{\alpha}\exp\big(\sqrt{N}Y(\alpha)\big)
    + \frac{N}{2}\big(\theta(\vc 1)-\theta(\vc q_k)\big)
    =\frac{N}{2}\sum_{r=1}^k m_r\big(\theta(\vc q_{r+1})-\theta(\vc q_{r})\big).
}
Inserting this identity into \eqref{first_fnc_1} yields
\eeq{ \label{at_time_0}
\frac{1}{N}\phi_N(0) = \E_u(\E\bar F_N) + \frac{1}{2}\sum_{r=1}^km_r\big(\theta(\vc q_{r+1})-\theta(\vc q_{r})\big).
}
Meanwhile, for $t=1$ we have
\eq{
\phi_N(1) &= \E_u\E\log \sum_{\alpha\in\N^{k-1}}v_\alpha\int_{\T_N}\exp\Big(\sum_{i=1}^N\sigma_iX_i(\alpha)+ c_NH_N^\pert(\sigma)\Big)\ \tau_N(\dd\sigma) \\
&\phantom{=}+ \frac{N}{2}\sum_{s\in\SSS}\lambda^s(\xi^s(\vc 1)-\xi^s(\vc q_k)).
}
In order to remove the perturbation term, we apply Lemma \ref{general_energy_lemma} with the following parameters:
\begin{itemize}
    \item In \eqref{sigma_tau_H}, take $(\Sigma,\tau) = (\T_N\times\N^{k-1},\tau_N\otimes(v_\alpha))$
    and $H(\sigma,\alpha) = \sum_{i=1}^N \sigma_iX_i(\alpha)$.
    \item In \eqref{Hu_def}, take $(h_i)_{i\geq1} = (H_{N,p,q}^\pert)_{p,q\geq1}$, $c = c_N$ so that $h_u = H_N^\pert$, $H_u = \Hb_{N,1}$.
\end{itemize}
In this case, we have already seen in \eqref{xi_pert_bound} that the constant $\varsigma^2(u)$ from \eqref{varsigma_assumption} satisfies $\varsigma^2(u) \leq 1$.
Therefore, \eqref{general_perturbation_control} implies
\eq{ %\label{first_fnc_0}
\Big|\phi_N(1) -\E\log \int_{\T_N}\sum_{\alpha\in\N^{k-1}}v_\alpha\exp\Big(\sum_{i=1}^N\sigma_iX_i(\alpha)\Big)\ \tau_N(\dd\sigma)-\frac{N}{2}\sum_{s\in\SSS}\lambda^s(\xi^s(\vc 1)-\xi^s(\vc q_k))&\Big| \\
\leq \frac{c_N^2N}{2}&.
}
Notice that the last two terms on the left-hand side are exactly of the form \eqref{psi_1_redef}, with $(\T_N,\tau_N)$ replacing $(\mathbf{T}_M,\vc\tau_M)$.
On the assumption that $N^s/N\to\lambda^s$ as $N\to\infty$, we computed the limiting value of this expression in \eqref{after_ldp}:
\eq{
\lim_{N\to\infty}\frac{1}{N}\E\log \int_{\T_N}\sum_{\alpha\in\N^{k-1}}v_\alpha\exp\Big(\sum_{i=1}^N\sigma_iX_i(\alpha)\Big)\ \tau_N(\dd\sigma)+\frac{1}{2}\sum_{s\in\SSS}\lambda^s\big(\xi^s(\vc 1)-\xi^s(\vc q_k)&\big) \\
= \PPP(\zeta,\Phi) + \frac{1}{2}\sum_{r=1}^k m_r\big(\theta(\vc q_{r+1})-\theta(\vc q_{r})&\big). 
}
From the two previous displays and the assumption that $c_N\to0$ as $N\to\infty$, we obtain
\eeq{ \label{at_1_bound}
\lim_{N\to\infty}\frac{1}{N}\phi_N(1) = \PPP(\zeta,\Phi)
+ \frac{1}{2}\sum_{r=1}^k m_r\big(\theta(\vc q_{r+1})-\theta(\vc q_{r})\big).
}
We thus have
\eq{
\limsup_{N\to\infty} \E F_N
\stackref{perturbation_control}{=} \limsup_{N\to\infty} \E_u(\E\bar F_N)
&\stackref{at_time_0}{=} \limsup_{N\to\infty} \frac{1}{N}\phi_N(0) - \frac{1}{2}\sum_{r=1}^km_r\big(\theta(\vc q_{r+1})-\theta(\vc q_{r})\big) \\
&\stackref{guerra_ineq}{\leq}  \lim_{N\to\infty} \frac{1}{N}\phi_N(1) - \frac{1}{2}\sum_{r=1}^km_r\big(\theta(\vc q_{r+1})-\theta(\vc q_{r})\big) \\
&\stackref{at_1_bound}{=} \PPP(\zeta,\Phi). \tag*{\qedsymbol}
}

\section{Lower bound part I: redefining the model} 
\label{redefine_model}
For Theorem \ref{main_thm}, the only thing that is assumed about the size of each species is that $\lambda^s(N) \coloneqq{\Lambda^s(N)/N}$ converges to a constant $\lambda^s\in(0,1]$ as $N\to\infty$.
In what follows, we define an auxiliary model whose limiting free energy is no larger than that of the original model, and this auxiliary model is different only in the sizes of each species. 
%so let us write $\wt N^s$ for the new number of coordinates in $\T_N$ that will be assigned to species $s$.
That is, we prescribe a method to change the value of $\Lambda^s(N)$ for certain $N$, in order to suit the large-$M$ asymptotics of the cavity method pursued in Section \ref{ass_section}.
Specifically, Proposition \ref{M_ratios} will ensure that the hypothesis of Proposition \ref{appearance_explained} is true.
The latter result will be invoked at the very last moment in proving \eqref{parisi_lower}; see Section \ref{main_proofs}.

Let $(N_k)_{k\geq1}$ be an increasing sequence of integers such that $N_1=1$, and 
\eq{ %\label{pass_to_subsequence}
\liminf_{N\to\infty} \E F_N = \lim_{k\to\infty} \E F_{N_k}.
}
Mimicking the shorthand $\Lambda^s(N) = N^s$ from before, we write $\Lambda^s(N_k)=N_k^s$.
By possibly passing to a subsequence of $(N_k)_{k\geq1}$, we may assume both of the following statements:

\begin{enumerate}[label=\textup{(\roman*)}]

\item \label{statement_1}
For each $k\geq1$, the quantity $\Delta_k \coloneqq N_{k+1}-N_k$ is at least $N_{k+1}/2$.
That is, $N_{k+1}\geq 2N_k$.

\item \label{statement_2}
For each $s\in\SSS$, the sequence $(N_k^s)_{k\geq1}$ is strictly increasing. 
(This is possible because $\lambda^s > 0$).

\end{enumerate}
On the sequence $(N_k)_{k\geq1}$, we alter nothing from the original model.
That is, we assume the $\SSS$-tuple $((\Lambda^s(N))_{s\in\SSS}$ has been prescribed for any $N$ belonging to $\{N_1,N_2,\dots\}$, but not for any other $N$.
Therefore, we must declare the value of $\Lambda^s(N)$ for every $N$ not belonging to $\{N_1,N_2,\dots\}$, which we do inductively as follows.

Suppose $N_k \leq N < N_{k+1}$ and that $(\Lambda^s(N))_{s\in\SSS}$ has been defined in such a way that $N_k^s\leq \Lambda^s(N)\leq N_{k+1}^s$ for each $s\in\SSS$.
Let $\alpha^s(N)$ be the unique number in $[0,1]$ such that
\eeq{ \label{alpha_implicit_defn}
 (1-\alpha^s(N)) N_k^s + \alpha^s(N)N_{k+1}^s =  \Lambda^s(N).
}
Now identify $s_*\in\SSS$ such that $\alpha^{s_*}(N)$ is minimal (if there are multiple such $s_*$, then choose one according to some deterministic rule), and set
\eq{
\Lambda^s(N+1) = \begin{cases}
\Lambda^{s}(N)+1 &\text{if }s=s_*, \\
\Lambda^s(N) &\text{if }s\neq s_*.
\end{cases}
}
In this way, $\Lambda^s(\cdot)$ is non-decreasing, and we maintain the identity $N = \sum_{s\in\SSS}\Lambda^s(N)$.
%
%\begin{remark} \label{redefine_remark}
%It will simplify our notation to note that because of the weakly increasing property, there is a natural way to assign a species to each positive integer $i$ (i.e.~the assignment does not depend on $N$).
%Namely, $s(i)$ is the unique value of $s$ such that $\Lambda^1s(i)>\Lambda^1s(i-1)$.
%So for the time being we can write $s(i)$ without ambiguity instead of $s(i,N)$.
%When we wish to isolate coordinates belonging to a single species, we use the notation $\sigma(s) = (\sigma_i)_{i:\, s(i) = s}$. 
%We thus replace $\T_N$ as defined in \eqref{original_TN_def} by the following isomorphic space (where the isomorphism is simply a permutation of coordinates):
%\eeq{ \label{new_torus}
%\T_N \coloneqq \big\{\sigma\in \R^N :\, \|\sigma(s)\|_2^2 = {N^s} \text{ for each $s\in\SSS$}\big\}.
%}
%Correspondingly, $\tau_N$ is now the product measure on $\T_N$ formed by taking $\sigma(s)$ to be uniform on $ S_{{N^s}}$ for each $s$.
%\end{remark}
%
%To simplify notation, let us now forget the original model and take $\wt N^s$ as the new value of ${N^s}$.
The new model we have now defined is maintained throughout the rest of the paper, and the desired outcome is the following.

%sRecalling the notation $\lambda^s(N) = {N^s}/N$ and that $\lambda^s(N_k)\to\lambda^s$ as $k\to\infty$, we can now state the outcome of our redefining the model.

\begin{prop} \label{M_ratios}
In the redefined model, the following limit holds for every $s\in\SSS$:
\eeq{ \label{M_ratios_eq}
\lim_{M\to\infty} \limsup_{N\to\infty} \Big|\frac{\Lambda^s(N+M)-\Lambda^s(N)}{M}-\lambda^s\Big| = 0.
}
\end{prop}

\begin{proof}
%Since we have only defined $\alpha^s(N)$ for $N\notin\{N_1,N_2,\dots\}$, let us set $\alpha^s(N_k)=0$ for all $k$.
%We then 
Let $\lambda_k^s = \lambda^s(N_k)$, and define for convenience the following quantities:
\eq{
%\eps_k^s \coloneqq \lambda^s(N_k) - \lambda^s, \quad
%\eps_{k} \coloneqq \max_{s\in\SSS} |\eps_k^s|, \quad
\alpham(N) &\coloneqq \min_{s\in\SSS}\alpha^s(N), &
\lambdam &\coloneqq \min_{s\in\SSS}\lambda^s, &
\eps_{k}^s &\coloneqq \lambda_k^s - \lambda^s, \\
\alphaM(N) &\coloneqq \max_{s\in\SSS}\alpha^s(N), &
\lambdaM &\coloneqq \max_{s\in\SSS}\lambda^s, &
\eps_{k} &\coloneqq \max_{s\in\SSS} |\eps_{k}^s|.
}
Note that $\eps_{k}\to0$ as $k\to\infty$.
Given $N$, let $k$ be the unique integer such that $N_k \leq N < N_{k+1}$.

\begin{claim}
If $k$ is large enough that $2\eps_{k+1}+\eps_{k} < \lambdam$, then
\eeq{ \label{max_min_gap}
\alphaM(N)-\alpham(N) \leq \frac{2}{(\lambdam-2\eps_{k+1}-\eps_{k})N_{k+1}}.
}
\end{claim}

\begin{proofclaim}
The claim is clear when $N=N_k$, since $\alpha^s(N_k) = 0$ for all $s\in\SSS$.
So let us assume $N_k < N < N_{k+1}$.
Since $\Delta_k = N_{k+1}-N_k \geq N_{k+1}/2$, we have
\eq{
{N_{k+1}^s}-{N_k^s} 
&= \lambda_{k+1}^s\cdot N_{k+1} - \lambda_k^s\cdot N_k\\
&= \big(\lambda_{k+1}^s-\lambda_k^s\big)N_{k+1}
+\lambda_k^s\Delta_k \\
&= (\eps_{k+1}^s-\eps_{k}^s) N_{k+1} + (\lambda^s + \eps_k^s)\Delta_k
\geq (\lambda^s + 2\eps_{k+1}^s -\eps_{k}^s)\frac{N_{k+1}}{2}.
}
Because $\Lambda^s(N)-\Lambda^s(N-1)$ takes the value $0$ or $1$, this inequality implies the following:
\eeq{ \label{max_min_gap_pre}
0\leq \alpha^s(N)-\alpha^s(N-1) 
\stackref{alpha_implicit_defn}{=} \frac{\Lambda^s(N)-\Lambda^s(N-1)}{{N_{k+1}^s}-N_k^s} \leq \frac{1}{(\lambdam-2\eps_{k+1}-\eps_{k})\frac{N_{k+1}}{2}}.
}
Then, because $\alpha^s(N)-\alpha^s(N-1)$ can be positive only when $\alpha^s(N-1)=\alpham(N-1)$, from \eqref{max_min_gap_pre} we deduce
\eq{
\alphaM(N) \leq \max\Big\{\alphaM(N-1),\alpham(N-1)+\frac{2}{(\lambdam-2\eps_{k+1}-\eps_{k})N_{k+1}}\Big\}.
}
On the other hand, we trivially have $\alpham(N) \geq \alpham(N-1)$, and so
\eq{
\alphaM(N)-\alpham(N) \leq \max\Big\{\alphaM(N-1)-\alpham(N-1),\frac{2}{(\lambdam-2\eps_{k+1}-\eps_{k})N_{k+1}}\Big\}.
}
Therefore, \eqref{max_min_gap} is true by induction.
\end{proofclaim}

Writing $\delta^s(N) \coloneqq \alpha^s(N)-\alpham(N)\geq0$ and observing that $\Delta_k = \sum_{s\in\SSS}({N_{k+1}^s}-{N_k^s})$, we trivially have
\eq{
N - N_k &= \sum_{s\in\SSS} \alpha^s(N)[{N_{k+1}^s}-{N_k^s}] 
= \alpham(N)\Delta_k + \sum_{s\in\SSS} \delta^s(N)[{N_{k+1}^s}-{N_k^s}].
}
By rearranging terms, we find that
\eq{
0\leq \frac{N-N_k}{\Delta_k}-\alpham(N)
&= \sum_{s\in\SSS} \delta^s(N)\cdot\frac{{N_{k+1}^s}-{N_k^s}}{\Delta_k} \\
&\leq \max_{s\in\SSS}\delta^s(N) = \alphaM(N)-\alpham(N). %\leq \frac{1}{N_{k+1}(\lambdam-2\eps_{k+1}-\eps_{k})}.
}
This inequality, combined with \eqref{max_min_gap}, yields the following expression as $N\to\infty$:
\eeq{ \label{alpha_limit}
\alpha^s(N) = \frac{N-N_k}{\Delta_k} + O(N_{k+1}^{-1}) \quad \text{for all $s\in\SSS$, $N_k \leq N < N_{k+1}$}.
}
In particular, hypothesis \eqref{lambda_assumption} is maintained, as explained by the next claim.

\begin{claim} \label{N_ratios}
In the redefined model, the following limit still holds for every $s\in\SSS$:
\eeq{ \label{correct_ratios}
\lim_{N\to\infty} \lambda^s(N) = \lambda^s.
}
\end{claim}

\begin{proof}
We already know that $\lambda_k^s\to\lambda^s$ as $k\to\infty$, and so we need only worry about $N$ not belonging to the sequence $(N_k)_{k\geq1}$.
For $N_k< N< N_{k+1}$,  we have
\eq{
\lambda^s(N)
&\stackrefp{alpha_limit}{=} \frac{\Lambda^s(N)}{N}
\stackref{alpha_implicit_defn}{=} \lambda_k^s(1-\alpha^s(N))\frac{N_k}{N} + \lambda^s_{k+1}\alpha^s(N)\frac{N_{k+1}}{N} \\
&\stackref{alpha_limit}{=} (\lambda^s+o(1))\Big[\Big(\frac{N_{k+1}-N}{\Delta_k}+O(N_{k+1}^{-1})\Big)\frac{N_k}{N}
+\Big(\frac{N-N_k}{\Delta_k}+O(N_{k+1}^{-1})\Big)\frac{N_{k+1}}{N}\Big].
}
Now \eqref{correct_ratios} follows from the observation that
\renewcommand{\qedsymbol}{$\square$ (Claim)}
\eq{
\frac{N_{k+1}-N}{\Delta_k}\cdot\frac{N_k}{N}+
\frac{N-N_k}{\Delta_k}\cdot\frac{N_{k+1}}{N} = 1.
\tag*{\qedhere}
}
\end{proof}

Now we can conclude the proof of Proposition \ref{M_ratios}.
Suppose that $N$ and $M$ are positive integers such that $N_k\leq N < N+M< N_{k+1}$.
Note that
\eq{ %\label{Ms_def}
%\Lambda^1s^\cav \coloneqq 
\Lambda^s(N+M)-\Lambda^s(N) &= [\alpha^s(N+M)-\alpha^s(N)][{N_{k+1}^s}-{N_k^s}] \\
\implies \quad\frac{\Lambda^s(N+M)-\Lambda^s(N)}{M}
&=\frac{\alpha^s(N+M)-\alpha^s(N)}{M}\big[(\lambda^s+o(1))N_{k+1}-(\lambda^s+o(1))N_k\Big].
}
Keeping $M$ fixed and letting $N\to\infty$, by \eqref{alpha_limit} this expression becomes
\eq{
\frac{\Lambda^s(N+M)-\Lambda^s(N)}{M} 
= \frac{1}{M}\Big(\frac{M}{\Delta_k} + O(N_{k+1}^{-1})\Big)\big(\lambda^s+o(1)\big)\Delta_k
= \big(\lambda^s+o(1)\big)\Big(1+\frac{O(1)}{M}\Big).
}
This analysis goes through also if $N+M=N_{k+1}$, by simply replacing $\alpha^s(N+M)$ with $1$.

If instead $N_k \leq N < N_{k+1} \leq N+M$, then
we may assume $N+M<N_{k+2}$.
Indeed, given any $M$, by condition \hyperref[statement_2]{\ref*{statement_1}} there is $k$ large enough that $N+M\leq N_{k+2}$ whenever $N< N_{k+1}$. 
By repeating twice the analysis from above, we can recover the same limiting statement as before (with $M$ fixed and $N\to\infty$):
\eq{
&\frac{\Lambda^s(N+M)-\Lambda^s(N)}{M}  \\
&= \frac{{N_{k+1}^s}-\Lambda^s(N)}{N_{k+1}-N}\cdot\frac{N_{k+1}-N}{M} + \frac{\Lambda^s(N+M)-{N_{k+1}^s}}{N+M-N_{k+1}}\cdot\frac{N+M-N_{k+1}}{M} \\
&= \big(\lambda^s+o(1)\big)\Big[\Big(1 + \frac{O(1)}{N_{k+1}-N}\Big)\frac{N_{k+1}-N}{M} + \Big(1+\frac{O(1)}{N+M-N_{k+1}}\Big)\frac{N+M-N_{k+1}}{M}\Big] \\
&= \big(\lambda^s+o(1)\big)\Big(1 + \frac{O(1)}{M}\Big).
}
That is, there is some constant $C$ not depending on $M$, such that
\eq{
\limsup_{N\to\infty}\Big|\frac{\Lambda^s(N+M)-\Lambda^s(N)}{M}-\lambda^s\Big| \leq C M^{-1}.
}
Upon sending $M\to\infty$, we have proved \eqref{M_ratios_eq}.
\end{proof}

In addition to conferring Proposition \ref{M_ratios}, the redefined model has the convenient feature that $\Lambda^s(N)$ is non-decreasing in $N$.
Therefore, we may assume that each integer $i$ is assigned a species which does not change with $N$.
That is, $s(i)$ is the unique value of $s$ such that $\Lambda^s(i) > \Lambda^s(i-1)$; here $s(1)$ is the unique value of $s$ such that $\Lambda^s(1) = 1$.
This simplification will allow us to more easily couple the models on $\T_{N+M}$ and $\T_N$; see \eqref{partition_3_parts}.

\settocdepth{section}
\section{Lower bound part II: the Aizenman--Sims--Starr scheme}
\label{ass_section}
The goal of this section is to establish (a rigorous version of) the inequality \eqref{ass_thm_summary}, as discussed heuristically in Section \ref{sketch_section}.
While \eqref{cavity_begin} would be a perfectly good starting place for the A.S.S.~scheme, we will need in Section \ref{final_lower_section} the perturbed form of the Hamiltonian (again, this is to guarantee the Ghirlanda--Guerra identities once $N$ is sent to infinity).
That is, we must work with $H_N^\pert$ from \eqref{Hpert_def} rather than $H_N$ from \eqref{mixed_H_def}.
%Furthermore, we will not have the luxury of using the same perturbation parameters $u_{p,q}$ for every $N$.
As before, let us think of the perturbation parameters $(u_{p,q})_{p,q\geq1}$ as i.i.d.~uniform random variables on $[1,2]$ which are independent of everything else, and then we write $\E_u$ to denote expectation over all $u_{p,q}$.
With this modified viewpoint, we recall the definitions from \eqref{perturbed_quantities_def} and apply \eqref{elementary_observation} to the sequence $a_N = \E_u( \E\log \bar Z_N)$, resulting in
\eeq{ \label{cavity_begin_modified}
\liminf_{N\to\infty} \E_u(\E \bar F_N) \geq \frac{1}{M}\liminf_{N\to\infty}\E_u\Big(\E\log\frac{\bar Z_{N+M}}{\bar Z_N}\Big).
}
But in light of Lemma \ref{no_change_lemma}, the left-hand side is just $\liminf_{N\to\infty} \E F_N$ once again, provided that $c_N\to0$.
The goal of the A.S.S. scheme is to understand the right-hand side of \eqref{cavity_begin_modified}.

To make the relevant computations, we will need that the number of cavity coordinates assigned to each species does not depend on $N$.
So for the remainder of this section, we will fix $M$ and then choose an increasing sequence $(N_k)_{k\geq1}$ that both achieves the limit infimum in \eqref{cavity_begin_modified}, i.e.~
\begin{subequations}
\label{subsequence_conditions}
\eeq{ \label{liminf_satisfied}
\liminf_{N\to\infty}\E_u\Big(\E\log \frac{\bar Z_{N+M}}{\bar Z_N}\Big)
= \lim_{k\to\infty}\E_u\Big(\E\log \frac{\bar Z_{N_k+M}}{\bar Z_{N_k}}\Big),
}
and is such that the number of cavity coordinates in each species is constant.
That is, for each $s\in\SSS$, there is a constant $\Lambda^s_\cav(M)$ satisfying
\eeq{ \label{cavity_constant}
\Lambda^s(N_k+M) - \Lambda^s(N_k) = \Lambda^s_\cav(M) \quad \text{for all $k\geq1$}.
}
\end{subequations}
The second condition \eqref{cavity_constant} is possible because there are only finitely many possibilities for the value of this difference, namely the integers between $0$ and $M$.
Within any sequence $(N_k)_{k\geq1}$, one of these possibilities must occur infinitely many times.
To ease our notational burden, we will henceforth write $N$ instead of $N_k$, understanding that we work only along the sequence $(N_k)_{k\geq1}$ chosen to satisfy \eqref{subsequence_conditions}.

In a slight abuse of notation, we will abbreviate $\Lambda^s_\cav(M)$ as just $M^s$.
This quantity should not be confused with $N^s = \Lambda^s(N)$ from \eqref{original_TN_def}.
If we take $\JJ^s$ to be the set of $j\in[M]$ such that $s(N+j) = s$, then $|\JJ^s| = M^s$, and we can consider the space $(\mathbf{T}_M,\vc\tau_M)$ from \eqref{cavity_space_def}.
While $\JJ^s$ does depend on $N$, its cardinality does not because of \eqref{cavity_constant}.
Therefore, in light of Remark \ref{Js_freedom}, we do not concern ourselves with how $\JJ^s$ depends on $N$.
Ultimately we will send $M\to\infty$, so let us note for later that regardless of the sequence $(N_k)_{k\geq1}$ chosen for each $M$, it follows from Proposition \ref{M_ratios} that
\eeq{ \label{cavity_fraction_converges}
\lim_{M\to\infty}\frac{M^s}{M}=\lambda^s.
}
For the time being, though, we work with fixed $M$.

%To ease our notational burden, we will continue to write $N$ instead of $N_k$, understanding that we work only along the sequence $(N_k)_{k\geq1}$ chosen to satisfy \eqref{subsequence_conditions}.

%Using the partition $M=\sum_{s\in\SSS}M^s$ from \eqref{cavity_constant}, let us define a product space analogous to $\T_N$ from \eqref{original_TN_def}:
%\eq{
%\mathbf{T}_{M} \coloneqq \Motimes_{s\in\SSS} S_{M^s}, \quad \text{where} \quad
% S_n \coloneqq \{\sigma\in\R^n:\, \|\sigma\|_2^2 = n\}..
%}
%We equip $\mathbf{T}_M$ with the product measure $\vc\tau_M$ whose marginal on $ S_{M^s}$ is the normalized surface measure $\mu_{M^s}$.
%Furthermore, just as we write $[N]=\uplus_{s\in\SSS}\II^s$ to denote the segmentation of $\T_{N}$ into the various species, we will write $[M] = \uplus_{s\in\SSS}\JJ^s$, where 
%$\JJ^s$ is the integer interval $\{1+\sum_{t=1}^{s-1}M_t,\dots,\sum_{t=1}^sM_t\}$.

%Returning now to the task of comparing the models on $\T_{N}$ and $\T_{N+M}$, 
Let us define the following rescaled version of $H_N$: %and $H_N^\pert$:
\eeq{ \label{H_NM_def}
H_{M,N}(\sigma) &\coloneqq \sum_{p\geq1}\beta_p\Big(\frac{N}{N+M}\Big)^{(p-1)/2}H_{N}^{(p)}(\sigma), %\\
%H_{M,N}^\pert(\sigma) &\coloneqq
%\sum_{p\geq1}\frac{N^{(p-1)/2}}{(N+M)^{(p-1)/2}}H_{N}^{(p),\pert}(\sigma).
}
where $H_N^{(p)}$ was defined in \eqref{HNp_def}.
The rescaling in \eqref{H_NM_def} is such that \eqref{covariance_at_N} becomes
\eeq{ \label{part_1_covariance}
\E[H_{M,N}({\sigma}) H_{M,N}({\sigma}')] = (N+M)\xi_{N}\Big(\frac{N}{N+M}\vc R(\sigma,\sigma')\Big), \quad {\sigma},{\sigma}' \in \R^N,
}
where $\xi_N$ is the covariance function from \eqref{xi_N_def}, and $\vc R(\sigma,\sigma')$ is the overlap vector defined in \eqref{overlap_def}.
Mimicking the notation from \eqref{barH_def}, we will write
\eeq{ \label{barH_NM_def}
\bar H_{M,N}(\sigma) \coloneqq H_{M,N}(\sigma) + c_N H_{N}^\pert(\sigma).
}
We also define the Gibbs measure and partition function associated to this Hamiltonian:
\eq{ %\label{perturbed_gibbs_def}
\bar G_{M,N}(\dd{\sigma}) \coloneqq \frac{1}{\bar Z_{M,N}}\exp(\bar H_{M,N}({\sigma}))\ \tau_N(\dd{\sigma}), \qquad
\bar Z_{M,N} \coloneqq \int_{\T_N}\exp(\bar H_{M,N}({\sigma}))\ \tau_N(\dd{\sigma}).
}
We will write $\langle\cdot\rangle_{M,N}$ to denote expectation over $\T_N$ with respect to $\bar G_{M,N}$.
This Gibbs measure is random depending on the Gaussian disorder, and its law depends on the choice of $u = (u_{p,q})_{p,q\geq1}$ in \eqref{Hpert_def}.

Now let $\vc\LL_{M,N} = \vc\LL_{M,N}(u)$ denote the law of the random overlap array produced by i.i.d.~samples from $\bar G_{M,N}$.
That is, in the notation of Section \ref{prelimit_section_3}, we have $\vc\LL_{M,N} = \mathsf{Law}(\vc\RR;\bar G_{M,N})$, where $\vc R:\T_N\times\T_N\to[-1,1]^\SSS$ is the map defined in \eqref{overlap_def}. 
Note that $\vc R(\sigma,\sigma)=\vc 1$ for all $\sigma\in\T_N$.
Regarding Assumption \ref{map_assumption}, the existence of processes $X_j$ and $Y$ on $\T_N$ satisfying \eqref{A3_cov} is verified in Remarks \ref{X_existence} and \ref{Y_existence}.
Therefore, we can speak of the functional $\Pi_M(\vc\LL_{M,N})$ defined in \eqref{Psi_def}, which is given by
\eeq{ \label{Psi_M_applied}
\Pi_M(\vc\LL_{M,N}) = \E\log\int_{\mathbf{T}_M}\Big\langle\exp\Big(\sum_{j=1}^M\kappa_jX_j(\sigma)\Big)\Big\rangle_{M,N}\ \vc\tau_{M}(\dd\kappa)- \E\log\big\langle \exp\big(\sqrt{M} Y(\sigma)\big)\big\rangle_{M,N}.
}
The rest of Section \ref{ass_section} is committed to proving the following result.

%Next let $X=(X_j({\sigma}))_{j\in[M],{{\sigma}}\in\T_{N}}$ be a centered Gaussian process with covariance structure given by
%\begin{linenomath}\postdisplaypenalty=0
%\begin{align}
%\label{S_cov}
%\E[X_{j}({\sigma}^1)X_{j'}({\sigma}^2)] &= \one_{\{j=j'\}}\cdot\xi^s(\vc \RR_{1,2}) \quad \text{for $j\in\JJ^s$},
%\end{align}
%\end{linenomath}
%where $\xi^s$ is the function from \eqref{gamma_def}.
%Finally, let $Y=(Y(\sigma))_{{\sigma}\in\T_N}$ be a centered Gaussian process with covariances
%\eeq{ \label{Y_def}
%\E[Y(\sigma^1)Y(\sigma^2)] = \theta(\vc \RR_{1,2}),
%}
%where $\theta$ is defined \eqref{theta_def}.
%Here it is assumed that $X$, $Y$, and $\bar H_{M,N}$ are mutually independent.
%Regarding the existence of $X$ and $Y$, see Remarks \ref{X_existence} and \ref{Y_existence}.

\begin{thm} \label{ass_thm}
Assume $c_N = N^{-\varpi}$ for some $\varpi>0$.
%Fix a positive integer $M$.
Let $\gamma_M$ be the joint law of $M$ independent, standard Gaussian random variables.
For $\delta>0$, let $\Ab_{M,\delta}$ be the following product of annuli:
\eeq{ \label{prod_ann_def}
\Ab_{M,\delta} \coloneqq \Motimes_{s\in\SSS} A_{M^s,\delta}, \quad \text{where} \quad A_{m,\delta} \coloneqq \big\{\kappa\in\R^m:\, m\leq \|\kappa\|_2^2 \leq m(1+\delta)\big\}.
}
For any $\delta\in(0,1]$ and any sequence $(N_k)_{k\geq1}$ satisfying \eqref{subsequence_conditions}, we have
\eeq{ \label{ass_thm_eq}
\liminf_{N\to\infty} \E F_N \geq
&\frac{1}{M}\limsup_{k\to\infty}\E_u\Pi_M(\vc\LL_{M,N_k}(u))
%\bigg[\E\log\int_{\mathbf{T}_M}\Big\langle\exp\Big(\sum_{j=1}^M\kappa_jX_j(\sigma)\Big)\Big\rangle_{M,N_k} \vc\tau_{M}(\dd\kappa) &- \E\log\big\langle \exp\big(\sqrt{M}Y(\sigma)\big)\big\rangle_{M,N_k}\bigg] 
- C\delta + \frac{1}{M}\log\gamma_M(\Ab_{M,\delta}),
}
where $C$ is a constant depending only on the values of $\lambda^s$, $s\in\SSS$.
\end{thm}

\begin{proof}
As before, let us just write $N$ instead of $N_k$, with the understanding that we work only along the sequence chosen to satisfy \eqref{subsequence_conditions}.
Since $c_N\to0$ as $N\to\infty$, we already know
\eeq{ \label{ass_beginning}
\liminf_{N\to\infty} \E F_N
\stackref{perturbation_control}{=}\liminf_{N\to\infty} \E_u(\E\bar F_N)
\stackrel{\mbox{\footnotesize\eqref{cavity_begin_modified},\eqref{liminf_satisfied}}}{\geq} \frac{1}{M}\lim_{k\to\infty}\E_u\Big(\E\log \frac{\bar Z_{N_k+M}}{\bar Z_{N_k}}\Big),
}
and so we turn our attention to the rightmost expression.
By trivial algebra we can write
\eq{
\E\log\frac{\bar Z_{N+M}}{\bar Z_N} 
&= \underbrace{\E\log\frac{\int_{\Ab_{M,\delta}} J_{M,N}(\kappa)\ \gamma_M(\dd\kappa)}{\bar Z_{M,N}}}_{Q_1}
+ \underbrace{\E\log \frac{\bar Z_{N+M}}{\int_{\Ab_{M,\delta}} J_{M,N}(\kappa)\ \gamma_M(\dd\kappa)}}_{Q_2}
- \underbrace{\E\log \frac{\bar Z_N}{\bar Z_{M,N}}}_{Q_3},
}
where
\eeq{ \label{J_def}
J_{M,N}(\kappa) \coloneqq \int_{\T_N}\exp\Big(\bar H_{M,N}({\sigma})+\sum_{j=1}^M\kappa_jX_j({\sigma})\Big)\ \tau_N(\dd{\sigma}).
}
Because of \eqref{ass_beginning}, to prove \eqref{ass_thm_eq}, it suffices to show three bounds which are uniform in $u$:
\eeqs{
Q_1 &\geq \log\gamma_M(\Ab_{M,\delta}) + \E\log\int_{\mathbf{T}_M}\Big\langle\exp\Big(\sum_{j=1}^M\kappa_jX_j(\sigma)\Big) \Big\rangle_{M,N}\ \vc\tau_{M}(\dd\kappa), \label{to_show_Q1} \\
Q_2 &\geq - C\delta M - o_M(1), \label{to_show_Q2} \\
Q_3 &\leq \E\log\big\langle \exp\big(\sqrt{M}Y(\sigma)\big)\big\rangle_{M,N} + o_M(1). \label{to_show_Q3}
}
Here $o_M(1)$ denotes a quantity depending on $M$ (but not on $u$) that converges to 0 as $N\to\infty$.
Verifying these three inequalities is the task of the next three sections.
The value of $C$ may change from line to line.

\subsection{Control of $Q_1$: proof of \eqref{to_show_Q1}}
Consider the random variable in $Q_1$:
\eq{
\frac{1}{\bar Z_{M,N}}\int_{\Ab_{M,\delta}} J_{M,N}(\kappa)\ \gamma_M(\dd\kappa)
&= \int_{\Ab_{M,\delta}}\Big\langle\exp\Big(\sum_{j=1}^M\kappa_jX_j({\sigma})\Big)\Big\rangle_{M,N}\ \gamma_M(\dd\kappa).
}
By Tonelli's theorem, we can move the expectation $\langle\cdot\rangle_{M,N}$ out of the integral over $\Ab_{M,\delta}$, and then use the product structure of $(\Ab_{M,\delta},\gamma_M)$:
\eq{
\int_{\Ab_{M,\delta}}\Big\langle\exp\Big(\sum_{j=1}^M\kappa_jX_j({\sigma})\Big)\Big\rangle_{M,N}\ \gamma_M(\dd\kappa)
= \bigg\langle \int_{\Ab_{M,\delta}}\exp\Big(\sum_{j=1}^M\kappa_jX_j({\sigma})\Big)\ \gamma_M(\dd\kappa)&\bigg\rangle_{M,N} \\
=\bigg\langle \prod_{s\in\SSS}\int_{A_{M^s,\delta}}\exp\Big(\sum_{j\in\JJ^s}\kappa_jX_j({\sigma})\Big)\ \gamma_{M^s}(\dd\kappa)&\bigg\rangle_{M,N}.
}
Now we apply \cite[Lem.~2.2]{chen13}, which says
\eq{
\int_{A_{M^s,\delta}}\exp\Big(\sum_{j\in\JJ^s}\kappa_j X_j({\sigma})\Big)\ \gamma_{M^s}(\dd\kappa)
&\geq\gamma_{M^s}(A_{M^s,\delta})\int_{ S_{M^s}}\exp\Big(\sum_{j\in\JJ^s}\kappa_j X_j({\sigma})\Big)\ \mu_{M^s}(\dd\kappa),
}
where $\mu_{M^s}$ is the normalized surface measure on the sphere $ S_{M^s}$.
Upon inserting this inequality into the previous display, and then reversing the factorization using the fact that $\Motimes_{s\in\SSS}( S_{M^s},\mu_{M^s})=(\mathbf{T}_M,\vc\tau_M)$, we arrive at
\eq{
\frac{1}{\bar Z_{M,N}}\int_{\Ab_{M,\delta}} J_{M,N}(\kappa)\ \gamma_M(\dd\kappa)
&\geq \gamma_M(\Ab_{M,\delta})\bigg\langle\int_{\mathbf{T}_M}\exp\Big(\sum_{j=1}^M\kappa_jX_j(\sigma)\Big)\ \vc\tau_{M}(\dd\kappa)\bigg\rangle_{M,N} \\
&=\gamma_M(\Ab_{M,\delta})\int_{\mathbf{T}_M}\Big\langle\exp\Big(\sum_{j=1}^M\kappa_jX_j(\sigma)\Big)\Big\rangle_{M,N}\ \vc\tau_{M}(\dd\kappa),
}
where the equality is once again from Tonelli's theorem.
We obtain \eqref{to_show_Q1} by taking the expected logarithm of both sides.
%\eeq{ \label{control_Q1}
%Q_1 &\geq \E\log\Big\langle\prod_{s\in\SSS}\gamma_{M^s}(A_{M^s,\delta})\int_{ S_{M^s}}\exp\Big(\sum_{j\in\JJ^s}\kappa_j X_j({\sigma})\Big)\ \mu_{M^s}(\dd\kappa(s))\Big\rangle_{M,N} \\
%&=\log\gamma_M(\Ab_{M,\delta}) + \E\log\Big\langle\int_{\mathbf{T}_{M}}\exp\Big(\sum_{j=1}^M\kappa_jX_j({\sigma})\Big)\ \vc\tau_M(\dd\kappa)\Big\rangle_{M,N}. %\\
%%&=\log\gamma_M(\Ab_{M,\delta,N}) + \E\log\int_{\T_{M,N}}\Big\langle\exp\Big(\sum_{j=1}^M\kappa_jS_j({\sigma})\Big)\Big\rangle_{M,N}\ \tau_M(\dd\kappa).
%}
%We have thus identified two of the four terms on the right-hand side of \eqref{ass_thm_eq}, with the desired inequality holding pointwise in $u$.

\subsection{Control of $Q_2$: proof of \eqref{to_show_Q2}}
This step is done in two parts, corresponding to a decomposition of $Q_2$ into two terms:
\eeq{ \label{split_Q2}
&\log \frac{\bar Z_{N+M}}{\int_{\Ab_{M,\delta}} J_{M,N}(\kappa)\ \gamma_M(\dd\kappa)} \\
&= \log\frac{\int_{\Ab_{M,\delta}} J_{M,N}(\kappa)P_{M,N}(\kappa)\, \dd\kappa}{\int_{\Ab_{M,\delta}} J_{M,N}(\kappa)\ \gamma_M(\dd\kappa)} + \log \frac{ \bar Z_{N+M}}{\int_{\Ab_{M,\delta}} J_{M,N}(\kappa) P_{M,N}(\kappa)\, \dd\kappa},
}
where $P_{M,N}$ is a function arising out of the following computation.
Since $\lambda^s>0$ for each $s\in\SSS$, we may assume $N$ is large enough that ${N^s}\geq1$ (this will avoid some divide-by-zero pathologies).
To begin, let us consider an element $\rho$ of the sphere $ S_{n+1}$ written as $\rho=(\tilde{\sigma},\kappa)$, where $\tilde{\sigma}\in\R^n$ and $\kappa\in\R$.
It is well-known that if $\rho$ is sampled uniformly (i.e.~according to $\mu_{n+1}$), then the density of $\kappa$ with respect to Lebesgue measure on $[-\sqrt{n+1},\sqrt{n+1}]$ is proportional to $(1-\kappa^2/(n+1))^{n/2-1}$.
Therefore, we have the identity
\eq{
&\int_{ S_{n+1}} f(\rho)\ \mu_{n+1}(\dd\rho) \\
&= \int_{ S_n}\int_{-\sqrt{n+1}}^{\sqrt{n+1}}f\Big(\sqrt{\frac{n+1-\kappa^2}{n}}{\sigma},\kappa\Big) \frac{\Gamma(\frac{n+1}{2})}{\Gamma(\frac{n}{2})\sqrt{(n+1)\pi}}\Big(1-\frac{\kappa^2}{n+1}\Big)^{\frac{n}{2}-1}\ \dd\kappa\, \mu_n(\dd{\sigma}),
}
which holds so long as $f$ is nonnegative or belongs to $L^1(\mu_{n+1})$.
If we define 
\eq{
B_{m,n}\coloneqq[-\sqrt{n+1},\sqrt{n+1}\, ]\times\cdots\times[-\sqrt{n+m},\sqrt{n+m}\, ],
}
then applying this identity inductively leads to
\eeq{ \label{sphere_induction}
\int_{ S_{n+m}} f(\rho)\ \mu_{n+m}(\dd\rho)
= \int_{ S_n}\int_{B_{m,n}} f\big(\psi_{m,n}({\sigma},\kappa)\big)p_{m,n}(\kappa)\ \dd\kappa\, \mu_n(\dd{\sigma}),
}
where the maps $\psi_{m,n}\colon S_n\times B_{m,n}\to S_{n+m}$ and $p_{m,n}\colon B_{m,n}\to\R $ are given by
\eeqs{
\psi_{m,n}({\sigma},\kappa) &\coloneqq (a_{m,n}^{(m)}(\kappa){\sigma},a_{m,n}^{(m-1)}(\kappa)\kappa_1,\dots,a_{m,n}^{(1)}(\kappa)\kappa_{m-1},\kappa_m), \label{coc_quantities_1} \\
a_{m,n}^{(\ell)}(\kappa) &\coloneqq \prod_{j=m-\ell+1}^{m}\sqrt{\frac{n+j-\kappa_{j}^2}{n+j-1}}, \quad 1\leq\ell\leq m, \label{coc_quantities_2} \\
p_{m,n}(\kappa) &\coloneqq \prod_{j=1}^{m} \frac{\Gamma(\frac{n+j}{2})}{\Gamma(\frac{n+j-1}{2})\sqrt{(n+j)\pi}}\Big(1 - \frac{\kappa_{j}^2}{n+j}\Big)^{\frac{n+j-1}{2}-1}. \label{coc_quantities_3}
%&= \pi^{-m/2}\frac{\Gamma(\frac{n+m}{2})}{\Gamma(\frac{n}{2})}\sqrt{\frac{n!}{(n+m)!}}\prod_{j=0}^{m-1}\Big(1 - \frac{\kappa_{j}^2}{n+j+1}\Big)^{\frac{n+j}{2}-1}
}

Next consider $\rho$ belonging to the product space $\T_{N+M}$,
%\eeq{ \label{the_product_space}
%\T_{N+M}
%&\stackref{original_TN_def}{=}
%\Motimes_{s\in\SSS} S_{\Lambda^s(N+M)}
%\stackref{cavity_constant}{=}\Motimes_{s\in\SSS} S_{N^s+M^s},
%}
and let us write $\rho = (\tilde\sigma,\tilde\kappa)$ with $\tilde\sigma\in\R^N$ and $\tilde\kappa\in\R^M$.
Recall the partitions $[N]=\uplus_{s\in\SSS}\II^s$ and $[M]=\uplus_{s\in\SSS}\JJ^s$, where $|\II^s| = N^s$ and $|\JJ^s| = M^s$.
These sets allow us to distinguish the various species:
\eq{
\tilde\sigma(s) \coloneqq (\tilde\sigma_i)_{i\in\II^s} \in \R^{N^s}, \quad
\tilde\kappa(s) \coloneqq (\tilde\kappa_j)_{j\in\JJ^s} \in\R^{M^s}, \quad
\rho(s) \coloneqq (\tilde\sigma(s),\tilde\kappa(s))\in S_{N^s+M^s}.
}
%In the same way as before, we decompose each factor as 
%\eeq{ \label{torus_decomposition}
%\rho(s) = (\tilde{\sigma}(s),\tilde{\kappa}(s))\in\R^{N^s}\times\R^{M^s}.
%}
Note that $\tilde{\sigma}(s)$ does not in general belong to $S_{N^s}$ (we only know $\|\tilde\sigma(s)\|_2^2\leq\|\rho(s)\|_2^2 = N^s+M^s$), hence the decoration by a tilde.
Therefore, we wish to perform the change of variables \eqref{sphere_induction} for each species $s\in\SSS$.
To this end, define the set
\eq{
\mathbf{B}_{M,N} &\coloneqq \Motimes_{s\in\SSS} B_{M^s,N^s} \subset \R^M,
}
and let
$\Psi_{M,N}\colon \T_N\times \mathbf{B}_{M,N}\to\T_{N+M}$ be the unique map such that 
the following diagram commutes for each $s\in\SSS$:
\eq{
\xymatrix{
(\sigma,\kappa) \ar[d]_-{\displaystyle\Psi_{M,N}} \ar[r]^-{} & (\sigma(s),\kappa(s)) \ar[d]^-{\displaystyle\psi_{M^s,N^s}} \\
\rho=(\tilde\sigma,\tilde\kappa) \ar[r]_-{} & (\tilde\sigma(s),\tilde\kappa(s))
}
}
%that if $\Psi_{M,N}({\sigma},\kappa)=\rho$, then %(\tilde{\sigma},\tilde{\kappa})$, then 
%\eeq{ \label{psi_def}
%\psi_{M^s,{N^s}}({\sigma}(s),\kappa(s))=\rho(s) \quad \text{for each $s\in\SSS$.}
%}
%The calculation that follows is what allows us to account for this inconsistency.
Thanks to the product structure of $\T_{N+M}$, $\T_N$, and $\mathbf{B}_{M,N}$, generalizing \eqref{sphere_induction} results in
%In view of \eqref{the_product_space}, generalizing the identity \eqref{sphere_induction} to $\T_{N+M}$ 
%results in
\eeq{ \label{product_cov}
\int_{\T_{N+M}}f(\rho)\ \tau_{N+M}(\dd\rho)
= \int_{\T_N}\int_{\mathbf{B}_{M,N}} f\big(\Psi_{M,N}({\sigma},\kappa)\big) P_{M,N}(\kappa)\ \dd\kappa\, \tau_N(\dd{\sigma}),
}
where
\eq{
P_{M,N}(\kappa) &\coloneqq \prod_{s\in\SSS} p_{M^s,{N^s}}(\kappa(s)), \quad \kappa\in\mathbf{B}_{M,N}.
}
Now observe that by applying Stirling's approximation to \eqref{coc_quantities_3}, we have the following limit for any fixed $m$ and $\kappa\in\R^m$:
\eq{
\lim_{n\to\infty} p_{m,n}(\kappa) = \frac{1}{(2\pi)^{m/2}}\exp\Big(-\frac{\|\kappa\|_2^2}{2}\Big).
} 
By the definition of $P_{M,N}$, this statement leads to
\eq{
\lim_{N\to\infty} P_{M,N}(\kappa) = \frac{1}{(2\pi)^{M/2}}\exp\Big(-\frac{\|\kappa\|_2^2}{2}\Big) \eqqcolon P_M(\kappa).
}
Note that $P_M$ is precisely the density function for the Gaussian measure $\gamma_M$.
We thus claim that the first term on the right-hand side of \eqref{split_Q2} satisfies
\eeq{ \label{claimed_0_lb}
\E\log\frac{\int_{\Ab_{M,\delta}} J_{M,N}(\kappa)P_{M,N}(\kappa)\, \dd\kappa}{\int_{\Ab_{M,\delta}} J_{M,N}(\kappa)\ \gamma_M(\dd\kappa)} \geq o_M(1).
}
Indeed, Jensen's inequality gives the following deterministic lower bound:
\eq{
\log \frac{\int_{\Ab_{M,\delta}} J_{M,N}(\kappa)P_{M,N}(\kappa)\, \dd\kappa}{\int_{\Ab_{M,\delta}} J_{M,N}(\kappa) P_M(\kappa)\, \dd\kappa}
&\geq \frac{\int_{\Ab_{M,\delta}} J_{M,N}(\kappa)P_M(\kappa) \log \frac{P_{M,N}(\kappa)}{P_M(\kappa)}\, \dd\kappa}{\int_{\Ab_{M,\delta}} J_{M,N}(\kappa) P_M(\kappa)\, \dd\kappa} \\\
&\geq \inf_{\kappa\in\Ab_{M,\delta}} \log\frac{P_{M,N}(\kappa)}{P_M(\kappa)}.
}
Since the convergence $P_{M,N}(\kappa)\to P_M(\kappa)$ is uniform on compact sets, and $P_M(\kappa)$ is bounded away from zero on the compact set $\Ab_{M,\delta}$, we have that
\eq{
\inf_{\kappa\in\Ab_{M,\delta}} \log\frac{P_{M,N}(\kappa)}{P_M(\kappa)} = o_M(1),
}
thus proving \eqref{claimed_0_lb}. 

Meanwhile, the second term on the right-hand side of \eqref{split_Q2} is controlled as follows.
The numerator in the logarithm is equal to
\eq{ %\label{product_cov_result}
\bar Z_{N+M} &\stackrefpp{perturbed_quantities_def}{product_cov}{=} \int_{\T_{N+M}}\exp(\bar H_{N+M}(\rho))\ \tau_{N+M}(\dd\rho) \\
&\stackref{product_cov}{=} \int_{\T_N}\int_{\mathbf{B}_{M,N}} \exp( \bar H_{N+M}\big(\Psi_{M,N}({\sigma},\kappa)\big)P_{M,N}(\kappa)\ \dd\kappa\, \tau_N(\dd{\sigma}).
}
Assuming $N$ is large enough that $\mathbf{B}_{M,N}$ contains $\Ab_{M,\delta}$, we now have the lower bound
\eeq{ \label{numerator_lower}
\bar Z_{N+M}
\geq \int_{\T_N}\int_{\Ab_{M,\delta}} \exp( \bar H_{N+M}\big(\Psi_{M,N}({\sigma},\kappa)\big)P_{M,N}(\kappa)\ \dd\kappa\, \tau_N(\dd\sigma).
}
Next we consider the denominator, which is
\eeq{
\label{recalling_denominator}
&\int_{\Ab_{M,\delta}} J_{M,N}(\kappa) P_{M,N}(\kappa)\, \dd\kappa \\
&\stackref{J_def}{=} \int_{\Ab_{M,\delta}} \int_{\T_N}\exp\Big( \bar H_{M,N}({\sigma})+\sum_{j=1}^M\kappa_jX_j({\sigma})\Big)P_{M,N}(\kappa)\ \tau_N(\dd{\sigma})\, \dd\kappa.
}
In view of \eqref{numerator_lower} and \eqref{recalling_denominator}, we are lead to compare $\bar H_{N+M}$ and $\bar H_{M,N}$ as follows.

Let us first consider the unperturbed versions of these Hamiltonians.
From \eqref{H_NM_def} we have
\eq{
H_{M,N}(\sigma)=\sum_{p\geq1}\frac{\beta_p}{(N+M)^{(p-1)/2}}\sum_{\vct i\in[N]^p}\sqrt{\Delta^2_{s(\vct i)}}g_{\vct i}\sigma_{\vct i}, \quad \sigma\in\T_N.
}
Recall that $H_{N+M}$ is very similar and simply contains more terms:
\eq{ %\label{rewriting_H}
H_{N+M}(\rho) = \sum_{p\geq1}\frac{\beta_p}{(N+M)^{(p-1)/2}}\sum_{\vct i\in[N+M]^p}\sqrt{\Delta^2_{s(\vct i)}}g_{\vct i}\rho_{\vct i}, \quad \rho\in\T_{N+M}.
}
Extending $H_{M,N}$ to all of $\R^N$, we use the identification $\rho= (\tilde{\sigma},\tilde{\kappa})$ to write
\eeq{ \label{partition_3_parts}
H_{N+M}(\rho)
%&\stackrel{\mathrm{dist}}{=}
=
 H_{M,N}(\tilde{\sigma})+\sum_{p\geq1}\frac{\beta_p}{(N+M)^{(p-1)/2}}\sum_{\vct i\in[N+M]^p\setminus[N]^p}\sqrt{\Delta^2_{s(\vct i)}}g_{\vct i}\rho_{\vct i}.
}
%- \frac{N^{(p-1)/2}}{(N+M)^{(p-1)/2}}H_{N,p}^\pert(\tilde{\sigma})\bigg].
We next separate the sum over $\vct i \in [N+M]^p\setminus[N]^p$ into two parts.
The first part will consist of those terms with exactly one cavity coordinate (i.e.~$\vct i$ contains exactly one coordinate larger than $N$).
Among such terms, let $\tilde\kappa_j \wt X_j(\tilde{\sigma})$ denote the sum of those whose cavity coordinate is $\tilde\kappa_j$ (here we have already summed over $p$).
The second part will collect all remaining terms, each of which contains at least two cavity coordinates; we call this part $D(\rho)$.
In summary, we have
\eeq{ \label{partition_2_parts}
\sum_{p\geq1}\frac{\beta_p}{(N+M)^{(p-1)/2}}\sum_{\vct i\in[N+M]^p\setminus[N]^p}\sqrt{\Delta^2_{s(\vct i)}}g_{\vct i}\rho_{\vct i}= \sum_{j=1}^M \tilde\kappa_j \wt X_j(\tilde{\sigma}) + D(\rho).
}
Note that $H_{M,N}$, $\wt X_j$, and $D$ are mutually independent with respect to the Gaussian disorder.
As is verified by a straightforward calculation, $\wt X_j$ is a centered Gaussian process with
\eeq{
\label{tilde_S_cov}
\E[\wt X_j(\tilde{\sigma})\wt X_{j'}(\tilde{\sigma}')] &= \one_{\{j=j'\}}\cdot\xi^s_{N}\Big(\frac{N}{N+M}\vc R(\tilde\sigma,\tilde\sigma')\Big) \quad \text{for $j\in\JJ^s$},
}
where $\vc R(\cdot,\cdot)$ is the overlap vector from \eqref{overlap_def}, and $\xi^s_{N}$ is the finite-volume version of $\xi^s$ from \eqref{gamma_def}:
\eq{
\xi^s_{N}(\vc x) 
&\coloneqq   \frac{1}{\lambda^{s}(N)}\frac{\partial \xi_{N}}{\partial x^s}(\vc x) 
=\sum_{p\geq1}p\beta_p^2\sum_{\vct t\in\SSS^{p-1}}
\Delta^2_{(\vct t,s)}
\lambda^{\vct t}(N)x^{\vct t}.
}
Also by direct calculation, the remainder term $D(\rho)$ satisfies
\eeq{ \label{T_remainder_bd}
\E[D(\rho)^2] 
&\le \frac{1}{N+M}\sum_{s_1,s_2\in\SSS}\|\tilde{\kappa}(s_1)\|_2^2\cdot\|\tilde{\kappa}(s_2)\|_2^2 \sum_{p\geq1} p(p-1) \beta_p^2\sum_{\vct t\in\SSS^{p-2}}\Delta^2_{\vct (t,s_1,s_2)}\lambda^{\vct t}(N+M) \\
&\hspace{-3.3ex}\stackref{lambda_assumption,decay_condition}{\leq} \frac{C\|\tilde{\kappa}\|_2^4}{N+M}
%&\leq \frac{\|\tilde{\kappa}\|_2^4}{N+M}\sum_{s_1,s_2\in\SSS} \frac{1}{\lambda_{s_1}(N+M)\lambda_{s_2}(N+M)}\frac{\partial^2\xi_{N+M}}{\partial x_{s_2}\partial x_{s_1}}(\vc 1)
\le \frac{C\|\tilde{\kappa}\|_2^4}{N}. \raisetag{2\baselineskip}
}

\begin{remark} \label{X_existence}
If we applied the same two-part decomposition as in \eqref{partition_2_parts}, but for the sum
\eq{
\sum_{p\geq1}\frac{\beta_p}{N^{(p-1)/2}}\sum_{\vct i\in[N+M]^p\setminus[N]^p}\sqrt{\frac{\Delta^2_{s(\vct i)}\lambda^{s(\vct i)}}{\lambda^{s(\vct i)}(N)}}g_{\vct i}\rho_{\vct i}
= \sum_{j=1}^M \tilde\kappa_j\wh X_j(\tilde\sigma) + \wh D(\rho),
}
then the covariance structure \eqref{tilde_S_cov} would be replaced by
\eq{
\E[\wh X_j(\tilde{\sigma})\wh X_{j'}(\tilde{\sigma}')]
&\stackrefp{gamma_def}{=}
\one_{\{j=j'\}}\cdot\frac{1}{\lambda^s(N)}\sum_{p\geq1}p\beta_p^2\sum_{\vct t\in\SSS^{p-1}}
\Big(\frac{\Delta^2_{(\vct t,s)}\lambda^{\vct t}\lambda^s}{\lambda^{\vct t}(N)\lambda^s(N)}\Big)
\lambda^{\vct t}(N)(\vc R(\tilde\sigma,\tilde\sigma'))^{\vct t} \\
&\stackref{gamma_def}{=} \one_{\{j=j'\}}\Big(\frac{\lambda^s}{\lambda^s(N)}\Big)^2\xi^s(\vc R(\tilde\sigma,\tilde\sigma')).
}
Therefore, the process $X_j(\tilde\sigma) = (\lambda^s(N)/\lambda^s)\wh X_j(\tilde\sigma)$, $j\in\JJ^s$ would have the covariance structure declared in \eqref{A3_cov}, but for $\tilde\sigma$ belonging to the projection of $\T_{N+M}$ onto the first $N$ coordinates.
Since this projection contains a copy of $\T_N$, the process $(X_j(\sigma))_{\sigma\in\T_N,j\in[M]}$ from \eqref{A3_cov} does exist.
\end{remark}

Now let $\wt H_{M,N}$ be an independent copy of $H_{M,N}$.
We define an interpolating Hamiltonian on $\T_N\times \Ab_{M,\delta}$, consisting of four parts:
\eq{
\Hb_t(\sigma,\kappa) \coloneqq \Hb_{t,1}(\sigma,\kappa) + \Hb_{t,2}(\sigma,\kappa) + \Hb_{t,3}(\sigma,\kappa) +
\Hb_{t,4}(\sigma,\kappa),
}
where, if we write $\Psi_{M,N}(\sigma,\kappa)=\rho=(\tilde\sigma,\tilde\kappa)$, then
\eq{
\Hb_{1,t}(\sigma,\kappa) &\coloneqq \sqrt{1-t} H_{M,N}({\sigma})+\sqrt{t} \wt H_{M,N}(\tilde{\sigma}), \\
\Hb_{2,t}(\sigma,\kappa) &\coloneqq \sqrt{1-t}\sum_{j=1}^M \kappa_j X_j(\sigma)+\sqrt{t}\sum_{j=1}^M \tilde\kappa_j \wt X_j(\tilde{\sigma}) , \\
\Hb_{3,t}(\sigma,\kappa) &\coloneqq \sqrt{t}\, D(\rho), \\
\Hb_{4,t}(\sigma,\kappa) &\coloneqq \sqrt{1-t}\, c_N H_{N}^\pert(\sigma)+\sqrt{t}\, c_{N+M}H_{N+M}^\pert(\rho).
}
Here we assume that $H_{M,N}$, $\wt H_{M,N}$, $X_j$, $\wt X_j$, $D$, $H_N^\pert$, and $H_{N+M}^\pert$ are  mutually independent.
The quantity of interest is the interpolating free energy
\eeq{ \label{interpolating_free_energy}
\phi(t) \coloneqq \E\log \int_{\Ab_{M,\delta}}\int_{\T_N} \exp\big(\Hb_t(\sigma,\kappa)\big)P_{M,N}(\kappa)\ \tau_N(\dd\sigma)\, \dd\kappa, \quad 0\leq t\leq 1.
}
At the initial time $t=0$, we have the expression from \eqref{recalling_denominator}:
\eq{
\phi(0) = \E\log\int_{\Ab_{M,\delta}} J_{M,N}(\kappa)P_{M,N}(\kappa)\, \dd\kappa.
}
At the terminal time $t=1$, by \eqref{partition_3_parts}, \eqref{partition_2_parts}, and Fubini's theorem, we recover the right-hand side of \eqref{numerator_lower}:
\eq{
\phi(1) 
&\stackrefp{numerator_lower}{=} \E\log\int_{\Ab_{M,\delta}}\int_{\T_N}\exp\Big( \bar H_{N+M}\big(\Psi_{M,N}(\sigma,\kappa)\big)\Big) P_{M,N}(\kappa)\ \tau_N(\dd\sigma)\, \dd\kappa \\
&\stackref{numerator_lower}{\leq} \E\log \bar Z_{N+M},
}
where the inequality holds for all large $N$.
Therefore, the final term in \eqref{split_Q2} satisfies
\eeq{ \label{why_derivative}
\E \log \frac{\bar Z_{N+M}}{\int_{\Ab_{M,\delta}} J_{M,N}(\kappa) P_{M,N}(\kappa)\, \dd\kappa} \geq \phi(1)-\phi(0)
&\geq -\sup_{t\in(0,1)} |\phi'(t)|.
}
%To compare the two values, we make the following derivative calculation.
To calculate the derivative of $\phi$, let $\langle\cdot\rangle_t$ denote expectation with respect to the Gibbs measure induced by $\Hb_t$ (where the reference measure on $\T_N\times\Ab_{M,\delta}$ is $\tau_N\otimes P_{M,N}(\kappa)\, \dd\kappa$, as in \eqref{interpolating_free_energy}), and observe that
\eq{
\phi'(t) = \E\Big\langle \frac{\dd \Hb_t(\sigma,\kappa)}{\dd t}\Big\rangle_t.
}
Then using Gaussian integration by parts (see \cite[Lem.~1.1]{panchenko13a}), we have
\begin{subequations}
\label{t_deriv_expression}
\eeq{
\phi'(t) &= \E\big\langle \CC\big((\sigma^1,\kappa^1),(\sigma^1,\kappa^1)\big) - \CC\big((\sigma^1,\kappa^1),(\sigma^2,\kappa^2)\big)\big\rangle_t, 
}
where $(\sigma^1,\kappa^1)$ and $(\sigma^2,\kappa^2)$ are regarded as independent samples from the Gibbs measure, and $\CC$ is defined by
\eeq{
\CC\big((\sigma,\kappa),(\sigma',\kappa')\big) &\coloneqq 
\E\Big[\frac{\dd\Hb_t(\sigma,\kappa)}{\dd t}\Hb_t(\sigma',\kappa')\Big].
}
By the mutual independence of $\Hb_{1,t}$, $\Hb_{2,t}$, $\Hb_{3,t}$, and $\Hb_{4,t}$, all cross terms in the product $(\dd \Hb_t/\dd t)\Hb_t$ vanish in expectation, leaving us with
\eeq{ \label{3_j}
\E\Big[\frac{\dd\Hb_t(\sigma,\kappa)}{\dd t}\Hb_t(\sigma',\kappa')\Big]
= \sum_{k=1}^4 \E\Big[\frac{\dd\Hb_{k,t}(\sigma,\kappa)}{\dd t}\Hb_{k,t}(\sigma',\kappa')\Big].
}
\end{subequations}
We now handle each of the four summands separately.

Concerning $k=1$, we observe that 
\eeq{ \label{j1_start}
\E\Big[\frac{\dd\Hb_{1,t}(\sigma,\kappa)}{\dd t}\Hb_{1,t}(\sigma',\kappa')\Big]
= \frac{1}{2}\E[ \wt H_{M,N}(\tilde{\sigma}) \wt H_{M,N}(\tilde{\sigma}')]
- \frac{1}{2}\E[ H_{M,N}(\sigma) H_{M,N}(\sigma')&] \\
\stackref{part_1_covariance}{=} \frac{N+M}{2}\Big[ \xi_{N}\Big(\frac{N}{N+M}\vc R(\tilde\sigma,\tilde\sigma')\Big) 
- \xi_{N}\Big(\frac{N}{N+M}\vc R(\sigma,\sigma')\Big)&\Big].
}
%where $\wt{\vc R}_{1,2}$ is the overlap vector from \eqref{overlap_def} but with $\sigma^1$ and $\sigma^2$ replaced by $\tilde{\sigma}^1$ and $\tilde{\sigma}^2$.
Since $(\tilde{\sigma}(s),\tilde{\kappa}(s)) = \rho(s)$ belongs to $ S_{N^s+M^s}$, we have
\eq{ %\label{pre_C_choice}
|R^s(\tilde\sigma,\tilde\sigma')| \leq \frac{\|\tilde\sigma(s)\|_2\|\tilde\sigma'(s)\|_2}{N^s}
\leq \frac{N^s+M^s}{{N^s}}. %= 1 + \frac{M^s}{N^s}.
}
Since $N/(N+M) \leq (N^s+M^s)/N^s$, it follows that
\eeq{ \label{overlap_containment}
\frac{N}{N+M}\vc R(\tilde\sigma,\tilde\sigma') \in [-1,1]^\SSS \quad \text{for all $(\sigma,\kappa),(\sigma',\kappa')\in\T_N\times\Ab_{M,\delta}$}.
}
Therefore, by \eqref{lambda_assumption} and \eqref{decay_condition}, we have the following trivial bound for all large $N$:
\eeq{ \label{j1_next}
&\Big|\xi_{N}\Big(\frac{N}{N+M}\vc R(\tilde\sigma,\tilde\sigma')\Big) - \xi_{N}\Big(\frac{N}{N+M}{\vc R}(\sigma,\sigma')\Big)\Big| \\
&\leq\sup_{\vc x\in[-1,1]^\SSS}\|\nabla\xi_{N}(\vc x)\|_2\cdot \frac{N\|{\vc R}(\tilde\sigma,\tilde\sigma')-{\vc R}(\sigma,\sigma')\|_2}{N+M}
\leq C\|{\vc R}(\tilde\sigma,\tilde\sigma')-{\vc R}(\sigma,\sigma')\|_2. \raisetag{2.5\baselineskip}
}
Because $(\tilde{\sigma}(s),\tilde{\kappa}(s)) = \rho(s)$ is taken equal to $\psi_{M^s,{N^s}}(\sigma(s),\kappa(s))$, it follows from the definition \eqref{coc_quantities_1} %and \eqref{psi_def}
that the two overlap vectors $\vc R(\tilde\sigma,\tilde\sigma')$ and $\vc R(\sigma,\sigma')$ are related by
\eeq{ \label{overlap_relationship}
R^s(\tilde\sigma,\tilde\sigma') = a_{M^s,{N^s}}^{(M^s)}(\kappa(s))\cdot a_{M^s,{N^s}}^{(M^s)}(\kappa'(s))\cdot R^s(\sigma,\sigma'), \quad s\in\SSS.
}
In the following claim, we take the convention that $a^{(0)}_{m,n} \equiv 1$.
\begin{claim}
For all $\kappa,\kappa'\in \mathbf{A}_{M,\delta}$, $\ell\in\{0,1,\dots,M^s\}$, and $N$ sufficiently large, we have
\eeqs{
\label{a_close_to_1}
\big|a^{(\ell)}_{M^s,N^s}(\kappa(s)) - 1\big| &\leq CMN^{-1},
\quad \text{and} \\ \label{aa_close_to_1}
\big|a^{(\ell)}_{M^s,N^s}(\kappa(s))a^{(\ell)}_{M^s,N^s}(\kappa'(s)) - 1\big| &\leq CMN^{-1}.
}
In the special case $\ell=M^s$, we have
\eeq{ \label{aa_close_to_1_special}
\big|a^{(M^s)}_{M^s,N^s}(\kappa(s))a^{(M^s)}_{M^s,N^s}(\kappa'(s)) - 1\big| \leq 2\delta M^s/N^s + CM^2N^{-2}.
}
\end{claim}

\begin{proofclaim}
By definition \eqref{prod_ann_def}, $\kappa\in \mathbf{A}_{M,\delta}$ means that 
\eeq{ \label{2_norm_restriction}
M^s \leq \|\kappa(s)\|_2^2 \leq (1+\delta)M^s.
}
Recall from \eqref{coc_quantities_2} that for $x\in B_{M^s,N^s}$, we have
\eq{
a_{M^s,{N^s}}^{(\ell)}(\kappa) &= \prod_{j=M^s-\ell+1}^{M^s} \sqrt{1 + \frac{1-{x_j^2}}{{N^s}+j-1}}.
}
So let us write $\kappa(s) = (x_1,\dots,x_{M^s})$, and set $\varkappa_j = (1-x_j^2)/(N^s+j-1)$.
For any $\JJ\subset[M^s]$, we have
\eq{
\prod_{j\in\JJ}(1+\varkappa_j)
= 1 + \sum_{j\in\JJ}\varkappa_j+\sum_{j_1<j_2}\varkappa_{j_1}\varkappa_{j_2} + \sum_{j_1<j_2<j_2}\varkappa_{j_1}\varkappa_{j_2}\varkappa_{j_3} + \cdots,
}
where the right-hand terminates after a suitable number of terms.
Subtracting the two leading terms, we find that
\eq{
\Big|\prod_{j\in\JJ}(1+\varkappa_j) - 1-\sum_{j\in\JJ}\varkappa_j\Big|
&\leq \Big(\sum_{j\in\JJ}|\varkappa_j|\Big)^2
+ \Big(\sum_{j\in\JJ}|\varkappa_j|\Big)^3 + \cdots 
\leq \frac{\big(\sum_{j\in\JJ}|\varkappa_j|\big)^2}{1-\sum_{j\in\JJ}|\varkappa_j|},
}
assuming that $\sum_{j\in\JJ}|\varkappa_j|<1$.
Now observe that
\eeq{ \label{second_deriv_bound}
\sum_{j\in\JJ}|\varkappa_j| \leq
\frac{1}{N^s}\sum_{j=1}^{M^s}(1+x_j^2)
\stackref{2_norm_restriction}{\leq} \frac{(2+\delta)M^s}{N^s} \leq CMN^{-1}.
}
It follows from the two previous displays that for all $N$ sufficiently large, we have
\eq{
\Big|\prod_{j\in\JJ}(1+\varkappa_j) - 1-\sum_{j\in\JJ}\varkappa_j\Big|
\leq \frac{(CMN^{-1})^2}{1-CMN^{-1}}
\leq CM^2N^{-2}.
}
By the mean value theorem (applied to $x\mapsto\sqrt{1+x}$), we conclude that
\eeq{ \label{with_two_terms}
\Big|a_{M^s,N^s}^{(\ell)}(\kappa(s))-1-\sum_{j=M^s-\ell+1}^{M^s}\frac{1-x_j^2}{N^s+j-1}\Big|
%&= \bigg|\bigg(\prod_{j=M^s-\ell+1}^{M^s}\Big(1 + \frac{1-{x_j^2}}{{N^s}+j-1}\Big)\bigg)^{1/2}-1\bigg| \\
&\leq CM^2N^{-2}.
}
The first inequality \eqref{a_close_to_1} follows from \eqref{with_two_terms} and \eqref{second_deriv_bound}.
The second inequality \eqref{aa_close_to_1} follows from \eqref{a_close_to_1}, thanks to the identity
\eeq{ \label{thanks_identity}
xy - 1 = (x-1)(y-1) + (x-1) + (y-1).
}
In the special case $\ell=M^s$, we have 
\eq{
\Big|\sum_{j=1}^{M^s}\frac{1-x_j^2}{N^s+j-1}\Big|
&= \Big|\sum_{j=1}^{M^s}\frac{1-x_j^2}{N^s}+\sum_{j=1}^{M^s}\Big(\frac{1-x_j^2}{N^s+j-1}-\frac{1-x_j^2}{N^s}\Big)\Big| \\
&=\Big|\frac{M^s}{N^s} - \frac{\|\kappa(s)\|_2^2}{N^s} + \sum_{j=1}^{M^s}\Big(\frac{1-x_j^2}{N^s+j-1}-\frac{1-x_j^2}{N^s}\Big)\Big| \stackref{2_norm_restriction}{\leq} \frac{\delta M^s}{N^s} + CMN^{-2}.
}
Therefore, \eqref{with_two_terms} says
\eq{
\big|a^{(M^s)}_{M^s,N^s}(\kappa(s)) -1\big| \leq \frac{\delta M^s}{N^s}+ CM^2N^{-2},
}
and then \eqref{aa_close_to_1_special} follows from \eqref{thanks_identity}.
\end{proofclaim}

We deduce the following for all large $N$:
\eeq{ \label{R_difference}
\|\vc R(\tilde\sigma,\tilde\sigma')-{\vc R}(\sigma,\sigma')\|_2 
&\stackref{overlap_relationship}{\leq}\sqrt{\sum_{s\in\SSS} \Big[a_{M^s,{N^s}}^{(M^s)}(\kappa(s))a_{M^s,{N^s}}^{(M^s)}(\kappa'(s))-1\Big]^2} \\
&\stackref{aa_close_to_1_special}{\leq} \sqrt{\sum_{s\in\SSS} \Big(\frac{2\delta M^s}{{N^s}} + CM^2N^{-2}\Big)^2}
\leq C\delta MN^{-1}.
}
Using this estimate in \eqref{j1_next} and inserting the resulting bound into \eqref{j1_start}, we arrive at the following:
\eeq{ \label{j_equal_1}
\sup_{(\sigma,\kappa),(\sigma',\kappa')\in\T_N\times\Ab_{M,\delta}}
\bigg|\E\Big[\frac{\dd\Hb_{1,t}(\sigma,\kappa)}{\dd t}\Hb_{1,t}(\sigma',\kappa')\Big]\bigg|
&\leq C\delta M.
}
This concludes the consideration of $k=1$ in \eqref{3_j}.

We next handle the $k=2$ case, for which a straightforward calculation gives
\eq{
&\E\Big[\frac{\dd\Hb_{2,t}(\sigma,\kappa)}{\dd t}\Hb_{2,t}(\sigma',\kappa')\Big]
= \frac{1}{2}\sum_{j=1}^M \Big[\tilde\kappa_j\tilde\kappa_j'\E[\wt X_j(\tilde{\sigma})\wt X_j(\tilde{\sigma}')] - \kappa_j\kappa_j'\E[X_j(\sigma)X_j(\sigma')]\Big].
}
The $j^\text{th}$ summand on the right-hand side can be computed by recalling  \eqref{tilde_S_cov} and \eqref{A3_cov} to compute the expectations, and then applying \eqref{coc_quantities_1} to express $\tilde\kappa_j,\tilde\kappa_j'$ in terms of $\kappa_j,\kappa_j'$.
When $j\in\JJ^s$, the resulting expression is equal to the following for some $\ell\in\{0,1,\dots,M^s-1\}$:
\eeq{ \label{pre_another_triangle}
\kappa_j\kappa_j'\Big[a_{M^s,{N^s}}^{(\ell)}(\kappa(s))\cdot a_{M^s,{N^s}}^{(\ell)}(\kappa'(s))\cdot\xi^s_{N}\Big(\frac{N}{N+M}\vc R(\tilde\sigma,\tilde\sigma')\Big) - \xi^s(\vc R(\sigma,\sigma'))\Big].
}
By the triangle inequality and \eqref{overlap_containment}, we have
\eeq{ \label{another_triangle}
&\Big|a_{M^s,{N^s}}^{(\ell)}(\kappa(s))\cdot a_{M^s,{N^s}}^{(\ell)}(\kappa'(s))\cdot\xi^s_{N}\Big(\frac{N}{N+M}\vc R(\tilde\sigma,\tilde\sigma')\Big) - \xi^s(\vc R(\sigma,\sigma'))\Big| \\
&\leq\sup_{\vc x\in[-1,1]^\SSS} |\xi^s_{N}(\vc x)|\cdot\Big|a_{M^s,{N^s}}^{(\ell)}(\kappa(s))\cdot a_{M^s,{N^s}}^{(\ell)}(\kappa'(s)) - 1\Big| \\
&\phantom{\leq}\quad+ \sup_{\vc x\in[-1,1]^\SSS}\|\nabla\xi^s_{N}(\vc x)\|_2\cdot\Big\|\frac{N}{N+M}\vc R(\tilde\sigma,\tilde\sigma') - {\vc R}(\sigma,\sigma')\Big\|_2 \\
&\phantom{\leq}\quad+ |\xi^s_{N}(\vc R(\sigma,\sigma')) - \xi^s(\vc R(\sigma,\sigma'))|.
}
The first term on the right-hand side is controlled by \eqref{aa_close_to_1}:
\eq{
\sup_{\vc x\in[-1,1]^\SSS} |\xi^s_{N}(\vc x)|\cdot\Big|a_{M^s,{N^s}}^{(\ell)}(\kappa(s))\cdot a_{M^s,{N^s}}^{(\ell)}(\kappa'(s)) - 1\Big|
\leq C M N^{-1}.
}
For the second term, we apply the triangle inequality and then invoke two of our previous inequalities:
\eq{
&\sup_{\vc x\in[-1,1]^\SSS}\|\nabla\xi^s_{N}(\vc x)\|_2\cdot\Big\|\frac{N}{N+M}\vc R(\tilde\sigma,\tilde\sigma') - \vc R(\sigma,\sigma')\Big\|_2 \\
&\stackrefp{overlap_containment,R_difference}{\leq} C\Big[\frac{M}{N+M}\|\vc R(\tilde\sigma,\tilde\sigma')\|_2 + \|\vc R(\tilde\sigma,\tilde\sigma') - \vc R(\sigma,\sigma')\|_2\Big] \\
&\stackref{overlap_containment,R_difference}{\leq} CMN^{-1} + C\delta MN^{-1} \leq C M N^{-1}.
}
Since $\vc R(\sigma,\sigma')\in[-1,1]^\SSS$, the final term in \eqref{another_triangle} is easily seen to tend to zero by the fact that $\lambda^s(N)\to\lambda^s$.
Indeed, by \eqref{lambda_assumption} and \eqref{decay_condition}, we can employ dominated convergence to conclude
\eq{
|\xi^s_{N}(\vc R(\sigma,\sigma')) - \xi^s(\vc R(\sigma,\sigma'))|
\leq \sum_{p\geq1}p\beta_p^2\sum_{\vc t\in\SSS^{p-1}}\Delta^2_{(\vct t,s)}|\lambda^{\vct t}(N)-\lambda^{\vct t}|
= o(1).
}
Here $o(1)$ denotes a quantity which tends to $0$ as $N\to\infty$, uniformly in all variables.
Now that the right-hand side of \eqref{another_triangle} is completely controlled by the three previous displays, we return to \eqref{pre_another_triangle}.
Since $\|\kappa\|_2^2 \leq (1+\delta)M$ for all $\kappa\in \Ab_{M,\delta}$, we find that
\eeq{ \label{j_equal_2}
\sup_{(\sigma,\kappa),(\sigma',\kappa')\in\T_N\times\Ab_{M,\delta}}
\bigg|\E\Big[\frac{\dd\Hb_{2,t}(\sigma,\kappa)}{\dd t}\Hb_{2,t}(\sigma',\kappa')\Big]\bigg|
= CMN^{-1} + o(1).
}
This concludes the consideration of $k=2$ in \eqref{3_j}.

Meanwhile, the $k=3$ term in \eqref{3_j} satisfies
\eeq{ \label{j3_setup}
&\E\Big[\frac{\dd\Hb_{3,t}(\sigma,\kappa)}{\dd t}\Hb_{3,t}(\sigma',\kappa')\Big]
= \frac{1}{2}\E[D(\rho)D(\rho')]
\stackref{T_remainder_bd}{\leq} \frac{C\|\tilde{\kappa}\|_2^2\cdot\|\tilde{\kappa}'\|_2^2}{N}.
}
Now let us recall the relationship between $\kappa\in\Ab_{M,\delta}$ and $\tilde{\kappa}$ once more: If $j\in\JJ^s$, then there is some $\ell$, $0\in\{0,1,\dots,M^s-1\}$ such that
\eeq{ \label{kappa_tilde_no_tilde}
|\tilde\kappa_j| \stackref{coc_quantities_1}{=} a^{(\ell)}_{M^s,{N^s}}(\kappa(s))|\kappa_j|
&\stackref{a_close_to_1}{\leq}
C|\kappa_j|.
}
Using this fact and \eqref{2_norm_restriction} in \eqref{j3_setup}, we find
\eeq{ \label{j_equal_3}
\sup_{(\sigma,\kappa),(\sigma',\kappa')\in\T_N\times\Ab_{M,\delta}}
\bigg|\E\Big[\frac{\dd\Hb_{3,t}(\sigma,\kappa)}{\dd t}\Hb_{3,t}(\sigma',\kappa')\Big]\bigg| \leq CM N^{-1}.
}

Finally, the $k=4$ term in \eqref{3_j} is the most delicate and satisfies
\eeq{ \label{j4_setup}
&\E\Big[\frac{\dd\Hb_{4,t}(\sigma,\kappa)}{\dd t}\Hb_{4,t}(\sigma',\kappa')\Big] \\
&\stackrefp{pert_cov}{=} \frac{c_{N+M}^2}{2}\E[H_{N+M}^\pert(\rho)H_{N+M}^\pert(\rho')] - \frac{c_N^2}{2}\E[H_N^\pert(\sigma)H_N^\pert(\sigma')] \\
&\stackref{pert_cov}{=} (N+M)\frac{c_{N+M}^2}{2}\xi_{N+M}^\pert(\vc R(\rho,\rho')) - N\frac{c_N^2}{2}\xi_N^\pert(\vc R(\sigma,\sigma')).
}
Here the overlap vector $\vc R(\rho,\rho') = (R^s(\rho,\rho'))_{s\in\SSS}$ is given by
\eq{
R^s(\rho,\rho') &= \frac{1}{N^s+M^s}\bigg(\sum_{i\in\II^s}\tilde\sigma_i\tilde\sigma_i' +\sum_{j\in\JJ^s}\tilde\kappa_j\tilde\kappa_j'\bigg) \\
&= \frac{N^s}{N^s+M^s} R^s(\tilde\sigma,\tilde\sigma') + \frac{1}{N^s+M^s}\sum_{j\in\JJ^s}\tilde\kappa_j\tilde\kappa_j'.
}
By the triangle inequality, we immediately have
\eq{
&\|\vc R(\rho,\rho') - \vc R(\tilde\sigma,\tilde\sigma')\|_2 \\
&\leq \|\vc R(\tilde\sigma,\tilde\sigma')\|_2\cdot\max_{s\in\SSS}\Big(1 - \frac{N^s}{N^s+M^s}\Big)
+ \sqrt{\sum_{s\in\SSS}\bigg(\frac{1}{N^s+M^s}\sum_{j\in\JJ^s}\tilde\kappa_j\tilde\kappa_j'\bigg)^2}.
}
Given \eqref{kappa_tilde_no_tilde} and the fact that $\|\kappa(s)\|_2^2 \leq (1+\delta)M^s$ for $\kappa\in\Ab_{M,\delta}$, we can conclude from the two previous displays that
\eeq{ \label{first_part_of_triangle}
\|\vc R(\rho,\rho') - \vc R(\tilde\sigma,\tilde\sigma')\|_2
%&\leq \|\wt{\vc R}_{1,2}^+ - \wt{\vc R}_{1,2}\|_2
%+ \|\wt{\vc R}_{1,2} - \vc \RR_{1,2}\|_2 \\
&= \|\vc R(\tilde\sigma,\tilde\sigma')\|_2\cdot CMN^{-1}
+ CMN^{-1} \stackref{overlap_containment}{\leq} CMN^{-1}.
}
Combining \eqref{first_part_of_triangle} with \eqref{R_difference}, we arrive at
\eeq{ \label{full_triangle}
\|\vc R(\rho,\rho') - \vc R(\sigma,\sigma')\|_2
\leq CMN^{-1}.
}
In particular, since $\vc R(\sigma,\sigma')\in[-1,1]^\SSS$, we may assume $N$ is sufficiently large that
$\vc R(\rho,\rho') \in [-2,2]^\SSS$ regardless of $\rho$ and $\rho'$.
Since $\xi_{N+M}^\pert(c\vc 1) < \infty$ for all $c\in(-4,4)$ (see \eqref{xi_pert_def}), this will be enough to bound all quantities involving $\xi_{N+M}^\pert$ by a constant.
We can now control the final expression in \eqref{j4_setup} as follows:
\eq{
&\Big|(N+M){c_{N+M}^2}\xi_{N+M}^\pert(\vc R(\rho,\rho')) - N{c_N^2}\xi_N^\pert(\vc R(\sigma,\sigma'))\Big| \\
%&\leq \frac{c_{N+M}^2}{2}|\xi_{N+M}^\pert(\wt{\vc R}_{1,2}^+) - \xi_{N+M}^\pert({\vc R}_{1,2})| \\
%&\phantom{\leq}+ \Big|\frac{c_N^2-c_{N+M}^2}{2}\Big\|\xi_{N+M}^\pert({\vc R}_{1,2})| \\
%&\phantom{\leq}+ \frac{c_N^2}{2}|\xi_{N+M}^\pert({\vc R}_{1,2}) - \xi_{N}^\pert({\vc R}_{1,2})|.
&\leq (N+M)c_{N+M}^2\|\vc R(\rho,\rho') - \vc R(\sigma,\sigma')\|_2\sup_{\vc x\in[-2,2]^\SSS}\|\nabla\xi^\pert_{N+M}(\vc x)\|_2 \\
&\phantom{\leq}+ \Big|{(N+M)c_{N+M}^2-Nc_{N}^2}\Big|\sup_{\vc x\in[-1,1]^\SSS}|\xi^\pert_{N+M}(\vc x)| \\
&\phantom{\leq}+ N{c_N^2}\sup_{\vc x\in[-1,1]^\SSS}|\xi_{N+M}^\pert(\vc x) - \xi_{N}^\pert(\vc x)|.
}
Upon inserting $c_N = N^{-\varpi}$ and using \eqref{full_triangle}, we find that the first product on the right-hand side is at most $CMN^{-2\varpi}$.
Considering the difference $(N+M)^{1-2\varpi} - N^{1-2\varpi}$, we see that the second product is also bounded from above by $CMN^{-2\varpi}$.
For the third and final product, since we have assumed that each $u_{p,q}$ in \eqref{xi_pert_def} does not depend on $N$, the supremum satisfies
\eq{
\sup_{\vc x\in[-1,1]^\SSS}|\xi_{N+M}^\pert(\vc x) - \xi_{N}^\pert(\vc x)|
&\stackref{xi_pert_bound}{\leq} 2\max_{s\in\SSS}|\lambda^s(N+M)-\lambda^s(N)| \\
&\stackrefp{xi_pert_bound}{\leq}
2\max_{s\in\SSS}\Big|\frac{M^sN-N^sM}{(N+M)N}\Big|
\leq CMN^{-1},
}
thereby making the third product at most $CMN^{-2\varpi}$.
We have thus argued that \eqref{j4_setup} can be rewritten
\eeq{ \label{j_equal_4}
\sup_{(\sigma,\kappa),(\sigma',\kappa')\in\T_N\times\mathbf{A}_{M,\delta}}\bigg|\E\Big[\frac{\dd\Hb_{4,t}(\sigma,\kappa)}{\dd t}\Hb_{4,t}(\sigma',\kappa')\Big]\bigg|
\leq CMN^{-2\varpi}.
}

Returning to \eqref{t_deriv_expression}, the inequalities \eqref{j_equal_1}, \eqref{j_equal_2}, \eqref{j_equal_3}, and \eqref{j_equal_4} yield the following bound as $N\to\infty$:
\eq{
|\phi'(t)| \leq C\delta M + CMN^{-1} + o(1) +  CMN^{-2\varpi} \quad \text{for all $t\in[0,1]$}.
}
Therefore, \eqref{why_derivative} becomes
\eq{
\E \log \frac{\bar Z_{N+M}}{\int_{\Ab_{M,\delta}} J_{M,N}(\kappa) P_{M,N}(\kappa)\, \dd\kappa} \geq -C\delta M - CMN^{-1} - o(1) - CMN^{-2\varpi},
}
which together with \eqref{claimed_0_lb} and \eqref{split_Q2} results in \eqref{to_show_Q2}.
%\eeq{ \label{control_Q2}
%\E\log \frac{\bar Z_{N+M}}{\int_{\Ab_{M,\delta}} J_{M,N}(\kappa)\ \gamma_M(\dd\kappa)} \geq -C\delta M-o_M(1). 
%}

\subsection{Control of $Q_3$: proof of \eqref{to_show_Q3}}
In this final step, we will show
\eeq{ \label{control_Q3}
\Big| \E\log \frac{\bar Z_N}{\bar Z_{M,N}} - \E\log\big\langle\exp\big(\sqrt{M}Y(\sigma)\big)\big\rangle_{M,N}\Big| = o_M(1).
%\quad \xrightarrow[N\to\infty]{\text{uniformly in $u$}}\quad 0.
}
In particular, \eqref{to_show_Q3} will hold, and so Theorem \ref{ass_thm} will be proved.
%Once \eqref{control_Q3} is established, it will combine with \eqref{control_Q1} and \eqref{control_Q2} to prove Theorem \ref{ass_thm}.
%
%We begin by separating out the perturbative terms in \eqref{barH_NM_def}:
%\eq{
%\wh H_{M,N}(\sigma) = \wc H_{M,N}(\sigma) + %c_{N+M}H_{M,N}^\pert(\sigma), 
%}
%where
%\eq{
%\wc H_{M,N}(\sigma)&\coloneqq\sum_{p\geq1}\frac{\beta_p}{(N+M)^{(p-1)/2}}\sum_{\vct i\in[N]^p}\sqrt{\Delta^2_{s(\vct i)}}g_{\vct i}\rho_{\vct i}.
%}
%\eq{
%H_{M,N}^\pert(\sigma) \coloneqq\sum_{p\geq1}\frac{1}{(N+M)^{(p-1)/2}}\sum_{\vct i\in[N]^p}\sum_{q\geq1}\frac{u_{p,q}}{2^{p+q}}H^\pert_{N+M,\vct i,q}(\rho).
%}
%Upon recalling the definitions from \eqref{barH_def}, \eqref{_H}, and \eqref{barH_NM_def}, observe that
%Then 
To begin, note the following equality in distribution, which is immediate from the definition \eqref{H_NM_def} of $H_{M,N}$:
\eq{
H_{N}
&\stackrel{\text{dist}}{=} H_{M,N}+
\sum_{p\geq1}\beta_p\sqrt{1 - \frac{N^{p-1}}{(N+M)^{p-1}}}\wt H_{N}^{(p)},
}
where $\wt H_{N}^{(p)}$ is an independent copy of $H_{N}^{(p)}$.
Let us write
\eeq{ \label{tilde_Y_def}
\wt Y(\sigma) &\coloneqq \frac{1}{\sqrt{M}}\sum_{p\geq1}\beta_p\sqrt{1 - \frac{N^{p-1}}{(N+M)^{p-1}}}\wt H_{N}^{(p)}(\sigma), \quad \sigma\in\T_N.
}
%\eq{
%E(\sigma) &\stackrefp{VW_def}{\coloneqq} \sum_{p\geq1}\sum_{\vct i\in[N]^p} \Big(\frac{W_{N,\vct i}(\sigma)}{N^{(p-1)/2}}-\frac{W_{N+M,\vct i}(\sigma)}{(N+M)^{(p-1)/2}}\Big) \\
%&\stackref{VW_def}{=} \sum_{p\geq1}\frac{c_N(N+M)^{(p-1)/2}-c_{N+M}N^{(p-1)/2}}{[N(N+M)]^{(p-1)/2}}\sum_{\vct i\in[N]^p} \sum_{q\geq1} \frac{u_{p,q}}{2^{p+q}}\sqrt{w_{q}^{s(\vct i)}}g_{\vct i,q}\sigma_{\vct i},
%}
Now define an interpolating Hamiltonian:
\eq{
\Hb_{t} \coloneqq H_{M,N} + \sqrt{M}(\sqrt{1-t}\, Y + \sqrt{t}\, \wt Y) + c_N H_N^\pert, \quad t\in[0,1].
}
%Here we assume that $H_{M,N}^\pert$, $H_N^\pert$, $Y$, and $\wt Y$ are mutually independent.
Notice that $\Hb_0 = \bar H_{M,N}+\sqrt{M}Y$ and $\Hb_1 \stackrel{\text{dist}}{=} \bar H_N$; so upon setting
\eq{
\phi(t) &\coloneqq \E\log\bigg(\frac{1}{\bar Z_{M,N}}\int_{\T_N}\exp(\Hb_t(\sigma))\ \tau_N(\dd\sigma)\bigg) \\
&= \E\log\Big\langle\exp\Big[\sqrt{M}\big(\sqrt{1-t}\, Y(\sigma) + \sqrt{t}\, \wt Y(\sigma)\big)\Big]\Big\rangle_{M,N},
}
we have
\eq{
\phi(0) = \E\log\big\langle\exp\big(\sqrt{M}Y(\sigma)\big)\big\rangle_{M,N} \quad \text{and} \quad
\phi(1) = \E\log \frac{\bar Z_N}{\bar Z_{M,N}}.
}
As before, differentiation followed by Gaussian integration by parts (see \cite[Lem.~1.1]{panchenko13a}) yields
\eeq{ \label{second_gibp}
\phi'(t) = \E\Big\langle\frac{\dd \Hb_t(\sigma)}{\dd t}\Big\rangle_t
= \E\big\langle \CC(\sigma^1,\sigma^1)-\CC(\sigma^1,\sigma^2)\big\rangle_t, \quad \text{where} \quad
\CC(\sigma,\sigma')
\coloneqq \E\Big[\frac{\dd\Hb_t(\sigma)}{\dd t}\Hb_t(\sigma')\Big].
}
Here $\langle\cdot\rangle_t$ denotes expectation with respect to the Gibbs measure on $\T_N$ associated to $\Hb_t$, and $\sigma^1,\sigma^2$ are independent samples from said measure.
By the independence of $Y$ and $\wt Y$, we have
\eeq{ \label{after_second_gibp}
\E\Big[\frac{\dd\Hb_t(\sigma)}{\dd t}\Hb_t(\sigma')\Big]
&= \frac{M}{2}\big(\E[\wt Y(\sigma)\wt Y(\sigma')] - \E[Y(\sigma)Y(\sigma')]\big). %\\
%&\phantom{=}+\frac{1}{2}\big(c_N^2\E[H_N^\pert(\sigma^1)H_N^\pert(\sigma^2)]-c_{N+M}^2\E[H_{M,N}^\pert(\sigma^1)H_{M,N}^\pert(\sigma^2)]\big).
%\\
%&= \frac{\sqrt{M}}{2}\big(\E[\wt Y(\sigma^1)\wt Y(\sigma^2)] - \bar\theta_N(\vc \RR_{1,2})\big)  - \frac{1}{2}\E[E(\sigma^1)E(\sigma^2)]
}
The first expectation on the right-hand side is given by
\eeq{ \label{tilde_Y_cov}
\E[\wt Y(\sigma)\wt Y(\sigma')]
%&\stackref{VW_def}{=} \frac{1}{M}\sum_{p\geq1}\beta_p^2\Big(\frac{1}{N^{p-1}}-\frac{1}{(N+M)^{p-1}}\Big)\sum_{\vct i\in[N]^p}\Delta^2_{s(\vct i)}\sigma^1_{\vct i}\sigma^2_{\vct i} \\
&= \frac{N}{M}\sum_{p\geq1}\beta_p^2\Big(1 - \frac{N^{p-1}}{(N+M)^{p-1}}\Big)\sum_{\vct s\in\SSS^p}\Delta^2_{s}\lambda^{\vct s}(N)R^{\vct s}(\sigma,\sigma').
}
\begin{remark} \label{Y_existence}
If \eqref{tilde_Y_def} were replaced by
\eq{
 Y(\sigma) = \sum_{p\geq1}\beta_p\sqrt{\frac{p-1}{N}}\wt H_N^{(p)}(\sigma) \quad \text{with} \quad \wt H_N^{(p)}(\sigma)= \frac{1}{N^{(p-1)/2}}\sum_{\vct i\in[N]^p} \sqrt{\frac{\Delta^2_{s(\vct i)}\lambda^{s(\vct i)}}{\lambda^{s(\vct i)}(N)}}g_{\vct i}\sigma_{\vct i},
}
then \eqref{tilde_Y_cov} would be replaced by
\eq{
\E[Y(\sigma)Y(\sigma')]
= N\sum_{p\geq1}\beta_p^2\Big(\frac{p-1}{N}\Big)\sum_{\vct s\in\SSS^p}\Big(\frac{\Delta^2_s\lambda^{\vct s}}{\lambda^{\vct s}(N)}\Big)\lambda^{\vct s}(N) R^s(\sigma,\sigma')
&\stackref{theta_def}{=} \theta(\vc R(\sigma,\sigma')),
}
and so the process $Y$ from \eqref{A3_cov} does indeed exist.
\end{remark}
%Considering that for any $\eps>0$ we have
From Taylor approximation of the function $x\mapsto x^{p-1}$ about $x=1$, we find
\eeq{ \label{taylor_with_p}
\Big|\frac{N}{M}\Big(1 - \Big(\frac{N}{N+M}\Big)^{p-1}\Big)
- \frac{N}{N+M}(p-1)\Big|
\leq p^2\frac{MN}{(N+M)^2}.
}
From this inequality we deduce two facts. 
First, we immediately have that
\eq{
\lim_{N\to\infty}\frac{N}{M}\Big(1 - \Big(\frac{N}{N+M}\Big)^{p-1}\Big)=p-1.
}
Second, for any $\eps>0$, we can choose a constant $C_\eps>0$ large enough that $C_\eps(1+\eps)^p \geq 2p^2$ for all $p\geq1$, and so
\eq{
C_\eps(1+\eps)^p \geq (p-1)+p^2 \stackref{taylor_with_p}{\geq} \frac{N}{M}\Big(1 - \Big(\frac{N}{N+M}\Big)^{p-1}\Big)  \quad \text{for all $N,p\geq1$}.
}
Therefore, by \eqref{lambda_assumption} and \eqref{decay_condition}, we can apply dominated convergence and conclude from \eqref{tilde_Y_cov} that
\eq{
\lim_{N\to\infty}\E[\wt Y(\sigma)\wt Y(\sigma')] 
&\stackrefp{theta_def}{=} \sum_{p\geq1}\beta_p^2(p-1) \sum_{\vct s\in\SSS^p}\Delta^2_{s}\lambda^{\vct s} R^s(\sigma,\sigma')  \\
&\stackref{theta_def}{=} \theta(\vc R(\sigma,\sigma'))
= \E[Y(\sigma)Y(\sigma')].
}
Consequently, the right-hand side of \eqref{after_second_gibp} vanishes as $N\to\infty$, and this convergence is uniform in $\sigma,\sigma'$ because $\vc R(\sigma,\sigma')\in[-1,1]^\SSS$.
That is,
%Interpreting this calculation in terms of \eqref{after_second_gibp}, we have
\eq{ %\label{second_gibp_zero}
\sup_{\sigma,\sigma'\in\T_N}\E\Big[\frac{\dd\Hb_t(\sigma)}{\dd t}\Hb_t(\sigma')\Big] &= o_M(1).
}
In light of \eqref{second_gibp}, we have thus verified \eqref{control_Q3}.
\end{proof}

\settocdepth{subsection}
\section{Lower bound part III: synchronization and limiting overlap distributions} \label{final_lower_section}
%The lower bound derived in the previous section is still in terms of quantities that vary with $N$.
%Here we develop expressions for the limit of these quantities once $N$ is brought to infinity.
%What enables this step is the Ghirlanda--Guerra identities.
In this section we complete the proof of Theorem \ref{main_thm} by identifying a $\vc\lambda$-admissible pair $(\zeta,\Phi)$ such that
\eq{
\lim_{N\to\infty} F_N \geq \PPP(\zeta,\Phi) \quad \mathrm{a.s.}
}
Recall %from Sections \ref{perturb_section} and \ref{ass_section} 
the following definitions.
First, we have the Hamiltonian $\bar H_{M,N}$ from \eqref{barH_NM_def}, whose associated Gibbs measure on $\T_N$ is denoted by $\bar G_{M,N}$.
Note that the perturbative term $H_N^\pert$ from \eqref{Hpert_def} depends on the parameters $u = (u_{p,q})_{p,q\geq1}$. 
Next let $\sigma^1,\sigma^2,\dots$ denote independent samples from $\bar G_{M,N}$, and set $\vc \RR_{\ell,\ell'} = (\RR^s_{\ell,\ell'})_{s\in\SSS}$ to be the overlap vector $\vc R(\sigma^\ell,\sigma^{\ell'})$ defined in \eqref{overlap_def}.
Then $\vc\LL_{M,N} = \vc\LL_{M,N}(u) = \mathsf{Law}(\vc\RR;\bar G_{M,N})$ denotes the law of the array $\vc\RR = (\vc \RR_{\ell,\ell'})_{\ell,\ell'\geq1}$.
Finally, for $\vc w\in[0,1]^\SSS$, let $\RR^{\vc w}_{\ell,\ell'} \coloneqq R^{\vc w}(\sigma^\ell,\sigma^{\ell'})$ be the quantity defined in \eqref{warped_overlap_N}.
Recall that we chose $\WWW=\{\vc w_1,\vc w_2,\dots\}$ to be dense in $[0,1]^\SSS$.

\subsection{Multi-species Ghirlanda--Guerra identities}
Consider any measurable function $f = f(\sigma^1,\dots,\sigma^n)$ mapping $\T_N^n \to \R$.
%\coloneqq(\vc \RR_{\ell,\ell'})_{\ell,\ell'\in[n]}$, 
Denote by $\Delta_{M,N}(f,n,p,q,u)$ the quantity
\eeq{ \label{GG_quantity}
\Big|\E\langle f\cdot(\RR^{\vc w_q}_{1,n+1})^p\rangle_{M,N}
- \frac{1}{n}\E\langle f\rangle_{M,N}\E\langle (\RR^{\vc w_q}_{1,2})^p\rangle_{M,N}
- \frac{1}{n}\sum_{\ell=2}^n\E\langle f\cdot(\RR^{\vc w_q}_{1,\ell})^p\rangle_{M,N}\Big|.
}
%Assume that $(u_{p,q})_{p,q\geq1}$ are i.i.d.~uniform random variables in the interval $[1,2]$.
The Ghirlanda--Guerra identities are the assertion that quantities of the form \eqref{GG_quantity} are equal to 0.
Indeed, this statement is true in the large-$N$ limit, at least in the following averaged sense.

\begin{thm} \label{GG_quantitative_thm}
Assume $c_N = N^{-\varpi}$ for some $\varpi\in[0,1/4)$, and that $u_{p,q}\in[0,3]$ for all $p,q$.
Then for every pair $p,q$, there is a constant $C_{p,q}$ not depending on $M$ or $N$ such that for every bounded measurable function $f = f(\sigma^1,\dots,\sigma^n)$, we have
\eeq{ \label{GG_quantitative}
\int_1^2\Delta_{M,N}(f,n,p,q,u)\ \dd u_{p,q} \leq C_{p,q}\|f\|_\infty n^{-1} N^{-1/4+\varpi} \quad \text{for all $N$ sufficiently large}.
}
\end{thm}

The proof of Theorem \ref{GG_quantitative_thm} will be to simply invoke the more general Theorem \ref{general_GG_thm}.
To do so, we will need the following lemma, whose proof invokes the more general Lemma \ref{general_concentration_lemma}.

\begin{lemma} \label{concentration_lemma}
Assume $u_{p,q}\in[0,3]$ for all $p,q\geq1$.
Then for any $M\geq0$, we have
%\eq{
%\P\big(|\log \bar Z_{M,N} - \E\log\bar Z_{M,N}| \geq t\sqrt{N}) \leq 2\exp\Big(-\frac{t^2}{4\big(\xi_{N}(\vc1)+c_N^2\big)}\Big).
%}
%In particular,
\eeq{ \label{expectation_outcome}
\E|\log \bar Z_{M,N} - \E\log\bar Z_{M,N}| \leq 
2\sqrt{\pi N(\xi_{N}(\vc1)+c_N^2)}.
}
\end{lemma}

\begin{proof}
Apply Lemma \ref{general_concentration_lemma} with the following parameters: 
\begin{itemize}
    \item In \eqref{sigma_tau_H}, take $(\Sigma,\tau) = (\T_N,\tau_N)$ and $H=0$.
    \item In \eqref{Hu_def}, take $h_1 = \bar H_{M,N}$ (all other $h_i \equiv 0$), $c=1$, and $u_1 = 1$ so that $H_u = \bar H_{M,N}$.
\end{itemize}
In this case, the constant $\varsigma^2(u)$ from \eqref{varsigma_assumption} satisfies
\eq{
\varsigma^2(u) = \frac{1}{N}\E[\bar H_{M,N}(\sigma)^2]
&\stackrefpp{barH_NM_def}{part_1_covariance,pert_cov}{=} \frac{1}{N}\Big(\E[H_{M,N}(\sigma)^2] + c_N^2\E[H_N^\pert(\sigma)^2]\Big) \\
&\stackref{part_1_covariance,pert_cov}{=} \frac{N+M}{N}\xi_N\Big(\frac{N}{N+M}\vc 1\Big) + c_N^2\xi_N^\pert(\vc 1)
\stackref{xi_pert_bound}{\leq}\xi_N(\vc 1) + c_N^2,
}
where in the last inequality we used the fact that $\xi_N(\alpha\vc q) \leq \alpha\xi_N(\vc q)$ for any $\vc q\in[0,1]^\SSS$ and $\alpha\in[0,1]$.
Therefore, \eqref{expectation_outcome} is a special case of \eqref{general_expectation_outcome}.
\end{proof}

\begin{proof}[Proof of Theorem \ref{GG_quantitative_thm}]
%If $\RR^{\vc w_q}_{1,1}=0$, then $\RR^{\vc w_q}_{1,\ell}=0$ for all $\ell$, in which case $\Delta_{M,N}(f,n,p,q,u)=0$.
%So let us assume $\RR^{\vc w_q}_{1,1}>0$.
Apply Theorem \ref{general_GG_thm} with the following inputs:
\begin{itemize}
    \item In \eqref{sigma_tau_H}, take $(\Sigma,\tau) = (\T_N,\tau_N)$ and $H = H_{M,N}$.
    \item In \eqref{Hu_def}, take $(h_i)_{i\geq1} = (H_{N,p,q}^\pert)_{p,q\geq1}$, $c = c_N$ so that $h_u = H_N^\pert$, $H_u = \bar H_{M,N}$.
\end{itemize}
Indeed, recall from \eqref{pert_pq_cov} that
\eq{
\frac{1}{N}\E[H_{N,p,q}^\pert(\sigma) H_{N,p,q}^\pert(\sigma')] = 4^{-(p+q)}(R^{\vc w_q}(\sigma,\sigma'))^p.
}
By Lemma \ref{concentration_lemma}, the quantity defined in \eqref{vartheta_def} satisfies $\vartheta \leq 2\sqrt{\pi N(\xi_{N}(\vc1)+c_N^2)}$.
Since $\xi_N(\vc 1)\to\xi(\vc 1)$ as $N\to\infty$  and $c_N^2\leq1$, we have $\vartheta = O(N^{1/2})$.
Therefore, the condition $N^\varpi\sqrt{\frac{4^{p+q}\vartheta }{R^{\vc w_q}(\sigma,\sigma)N}} < 1$ is satisfied by all large $N$, since $\varpi<1/4$.
Now \eqref{general_GG_claim} yields
\eq{
\int_1^2 \Delta_{M,N}(f,n,p,q,u)\ \dd u_{p,q}
\leq 24\|f\|_\infty {2^{p+q}}n^{-1}N^{\varpi-1/2}\big(1+O(N^{1/4})\big).
}
By inspection and the fact that $\varpi<1/4$, we conclude \eqref{GG_quantitative}.
\end{proof}

%Our proof of Theorem \ref{GG_quantitative_thm} is a direct adaptation of that of \cite[Thm.~3.2]{panchenko13a}, although for %the reader's convenience, the argument is written to be essentially self-contained.
%Since no serious complications arise in the present setting, we defer the proof to REF.

In order to apply Theorem \ref{GG_quantitative_thm} simultaneously for all test functions $f$, let us enumerate for each $n$ all monic monomials in the entries of $\vc\RR^n = (\RR^s_{\ell,\ell'})_{\ell,\ell'\in[n],s\in\SSS}$.
Combining all these enumerations, we 
obtain a sequence $(f_r)_{r\geq1}$, where $f_r$ is a monomial in the entries of $\vc\RR^{n_r}$.
We then define
\eq{
\Delta_{M,N}(u) \coloneqq \sum_{p,q,r\geq1}\frac{\Delta_{M,N}(f_r,n_r,p,q,u)}{2^{p+q+r}}.
}

\begin{remark}
To clarify possible confusion, we note that every monic monomial will actually appear in the list $(f_r)_{r\geq1}$ infinitely many times, but just once for each appropriate $n$.
For example, for each $n\geq3$, there is exactly one value of $r$ such that $f_r = \RR_{1,2}^s\RR_{2,3}^s$ and $n_r = n$.
These repetitions are necessary because \eqref{GG_quantity} depends not just on $f$ but also on $n$.
\end{remark}

Recall that $\P_u$ is the product measure under measure each $u_{p,q}$ is an independent uniform random variable in $[1,2]$, and $\E_u$ denotes expectation with respect to $\P_u$.
Since $\Delta_{M,N}(f_r,n_r,p,q,u) \leq 2$, it follows from Tonelli's theorem, dominated convergence, and Theorem \ref{GG_quantitative_thm} that
\eeq{ \label{all_pqr}
\lim_{N\to\infty}\E_u\Delta_{M,N}(u)
= \sum_{p,q,r\geq1}\E_u\bigg[\lim_{N\to\infty}\int_1^2 \frac{\Delta_{M,N}(f_r,n_r,p,q,u)}{2^{p+q+r}}\ \dd u_{p,q}\bigg] = 0.
}
This allows us to choose a deterministic sequence of perturbation parameters $(u_N)_{N\geq1}$, where $u_N=(u_{p,q}(M,N))_{p,q\geq1}$, such that
\begin{subequations} \label{coordination}
\eeq{ \label{GG_together}
\lim_{N\to\infty}\Delta_{M,N}(u_N) = 0,
}
but we need to coordinate this choice with Theorem \ref{ass_thm}.
%To be more precise, let us write
%\eq{ %\label{Theta_def}
%\CC_{M}(N,u)\coloneqq
%\E\log\int_{\mathbf{T}_M}\Big\langle\exp\Big(\sum_{j=1}^M\kappa_jX_j(\sigma)\Big)\Big\rangle_{M,N}\vc\tau_{M}(\dd\kappa) - \E\log\big\langle \exp\big(\sqrt{M}Y(\sigma)\big)\big\rangle_{M,N},
%}
%where $(X_j)_{j\in[M]}$ and $Y$ are the centered Gaussian processes defined by \eqref{S_cov} and \eqref{Y_def}.
%In order to recover the result of Theorem \ref{ass_thm} for a single sequence of choices $(u_N)_{N\geq1}$ of the perturbation parameters, we want
That is, we also want
\eeq{ \label{coordinated_want}
\limsup_{N\to\infty}[\E_u\Pi_{M}(\vc\LL_{M,N}(u)) - \Pi_{M}(\vc\LL_{M,N}(u_N))] \geq 0.
}
\end{subequations}

\begin{lemma} \label{coordinate_choice_lemma}
Assume $c_N = N^{-\varpi}$ for some $\varpi\in[0,1/4)$.
Then there is a sequence $(u_N)_{N\geq1}$ (which depends on $M$) such that \eqref{coordination} holds.
\end{lemma}

\begin{proof}
%The purpose of demanding $\varpi>0$ is, of course, to ensure that the Hamiltonian perturbation in \eqref{H_pert_def} has no effect on the limiting free energy; see Lemma \ref{no_change_lemma}.
%The second constraint $\varpi<1/4$ forces the perturbation to be large enough that \eqref{GG_quantitative} admits a quantity which tends to $0$ as $N\to\infty$.
Here we follow the standard example of \cite[Lem.~3.3]{panchenko13a}.
Consider the events
\eq{
\Asf_{N,\eps} \coloneqq \{u:\, \Pi_{M}(\vc\LL_{M,N}(u))\leq\E_{u'}\Pi_{M}(\vc\LL_{M,N}(u'))+\eps\}, \qquad
\Bsf_{N,\eps} \coloneqq \{u:\, \Delta_{M,N}(u)\leq\eps\}.
}
The goal is to identify $\eps_N\to0$ such that
$\P_u(\Asf_{N,\eps_N}\cap \Bsf_{N,\eps_N})>0$ for all large $N$.
Recall the centered Gaussian processes $X_j$ and $Y$ appearing in the expression \eqref{Psi_M_applied} for $\Pi_M(\vc\LL_{M,N})$, which are independent of the random disorder defining $\bar G_{M,N}$.
By applying Jensen's inequality twice, we see that
\eq{
0&= \E\int_{\mathbf{T}_M}\Big\langle\sum_{j=1}^M\kappa_jX_j(\sigma)\Big\rangle_{M,N}\vc\tau_{M}(\dd\kappa)  \\
&\leq \E\log\int_{\mathbf{T}_M}\Big\langle\exp\Big(\sum_{j=1}^M\kappa_jX_j(\sigma)\Big)\Big\rangle_{M,N} \vc\tau_{M}(\dd\kappa) \\
&\leq \log\int_{\mathbf{T}_M}\E\Big\langle\exp\Big(\sum_{j=1}^M\kappa_jX_j(\sigma)\Big)\Big\rangle_{M,N} \vc\tau_{M}(\dd\kappa)
\stackref{EX_computation}{=}
\sum_{s\in\SSS}\frac{M^s\xi^s(\vc 1)}{2}.
}
By similar reasoning (using \eqref{EY_computation} instead of \eqref{EX_computation}), we also have
\eq{
0 \leq \E\log\big\langle \exp\big(\sqrt{M}Y(\sigma)\big)\big\rangle_{M,N}
\leq \frac{M\theta(\vc1)}{2}.
}
%
%Let $\E_X$ and $\E_Y$ denote expectation over only $(X_j)_{j\in[M]}$ and $Y$, respectively.
%Since these processes are independent of the randomness defining the Gibbs measure $\bar G_{M,N}$, we can freely reverse the order of $\E_X$ or $\E_Y$ with $\langle\cdot\rangle_{M,N}$.
%Using this fact as needed, we have the following by %applying Jensen's inequality twice (first with %$\int_{\mathbf{T}_M}\langle\cdot\rangle_{M,N}\,\tau_%M(\dd\kappa)$, and then with $\E_X$):
%
It follows from the two previous displays that
\eq{
-\frac{M\theta(\vc 1)}{2} \leq \Pi_{M}(\vc\LL_{M,N}(u)) \leq \sum_{s\in\SSS}\frac{M\xi^s(\vc1)}{2} \quad \text{for all $N$ and $u$}.
}
For simplicity, we will write $C_M \coloneqq \max\{1,M\theta(\vc1)/2,\sum_{s\in\SSS}M^s\xi^s(\vc1)/2\}$.
For any $\eps>0$, we trivially have
\eq{
\E_u\Pi_{M}(\vc\LL_{M,N}(u)) &\geq (\E_u\Pi_{M}(\vc\LL_{M,N}(u))+\eps)\cdot\P_u(\Asf_{N,\eps}^\cc)-C_M\cdot\P_u(\Asf_{N,\eps}) \\
\implies \quad \P_u(\Asf_{N,\eps}) 
&\geq \frac{\eps}{\E_u\Pi_{M}(\vc\LL_{M,N}(u))+\eps+C_M} \geq \frac{\eps}{2C_M+\eps}.
}
On other hand, Markov's inequality gives
\eq{
\P_u(\Bsf_{N,\eps}) \geq 1 - \frac{\E_u\Delta_{M,N}(u)}{\eps}.
%\geq 1 - C_M\frac{\E_u\Delta_{M,N}(u)}{\eps}.
}
Now set $\eps_N = 2\sqrt{C_M\E_u\Delta_{M,N}(u)}$, which tends to $0$ as $N\to\infty$ by \eqref{all_pqr}.
Assuming $N$ is large enough that $\eps_N<C_M$, we have
\eq{
\P_u(\Asf_{N,\eps_N})+\P_u(\Bsf_{N,\eps_N})
\geq \frac{\eps_N}{3C_M}+1-\frac{\E_u\Delta_{M,N}(u)}{\eps_N} = 1 + \frac{1}{6}\sqrt{\frac{\E_u\Delta_{M,N}(u)}{C_M}} > 1.
}
This final display assumes that $\E_u\Delta_{M,N}(u) > 0$, but even if $\E_u\Delta_{M,N}(u)$ were $0$, we would trivially have $\P_u(\Asf_{N,0}) > 0$ and $\P_u(\Bsf_{N,0}) = 1$.
\end{proof}

\subsection{Synchronization and asymptotic Gibbs measures}
In accordance with Lemma \ref{coordinate_choice_lemma}, assume henceforth that $c_N = N^{-\varpi}$ for some $\varpi\in(0,1/4)$.
Once the parameters $(u_N)_{N\geq1}$ are chosen such that \eqref{coordination} holds, let us restrict our attention to the sequence $(N_k)_{k\geq1}$ from Theorem \ref{ass_thm}, so that
\eeq{ \label{restrict_attention}
\liminf_{N\to\infty} \E F_N \geq
\frac{1}{M}\limsup_{k\to\infty}
\Pi_{M}(\vc\LL_{M,N_k}(u_{N_k}))
%\bigg[&\E\log\Big\langle\int_{\mathbf{T}_{M}}\exp\Big(\sum_{j=1}^M\kappa_jX_j(\sigma)\Big)\ \vc\tau_{M}(\dd\kappa)\Big\rangle_{M,N_k} \\
%&- \E\log\big\langle \exp\big(\sqrt{M}Y(\sigma)\big)\big\rangle_{M,N_k}\bigg] 
- C\delta + \frac{1}{M}\log\gamma_M(\Ab_{M,\delta}).
}
Since the overlaps are bounded, by passing to a suitable subsequence of $(N_k)_{k\geq1}$, we may assume that as $k\to\infty$, $\vc\LL_{M,N_k}(u_{N_k})$ converges weakly to some law $\vc\LL_M$.
By Corollary \ref{extension_cor}, the quantity $\Pi_M(\vc\LL_{M,N_k}(u_{N_k}))$ converges to some limit we can call $\Pi_M(\vc\LL_M)$, and then \eqref{restrict_attention} becomes
\eeq{ \label{restrict_attention_2}
\liminf_{N\to\infty} \E F_N \geq
\frac{1}{M}
\Pi_{M}(\vc\LL_M)
- C\delta + \frac{1}{M}\log\gamma_M(\Ab_{M,\delta}).
}
%Corollary \ref{extension_cor} also implies that the same limit would be achieved if we took a different sequence converging to $\vc \LL_M$.
Now recall the function $\PPP_M$ that was defined in \eqref{restriction_def}.
Namely, $\PPP_M$ is the restriction of $\Pi_M$ to overlap distributions of the form $\vc\LL(\zeta,\Phi)$ for some $\vc\lambda$-admissible pair $(\zeta,\Phi)$ in which $\zeta$ has finite support; see \eqref{rpc_relevant_law}.
What we do next is to identify---by way of synchronization---a sequence of such pairs $(\zeta^k,\Phi)$ such that $\vc\LL(\zeta^k,\Phi)\to\vc\LL_M$ as $k\to\infty$.
In this way we will be able to rewrite \eqref{restrict_attention_2} as follows.

\begin{prop} \label{find_pair}
There is a $\vc\lambda$-admissible pair $(\zeta_M,\Phi_M)$ such that for any $\delta\in(0,1]$,
\eeq{ \label{find_pair_eq}
\liminf_{N\to\infty} \E F_N \geq \frac{1}{M}\PPP_M(\zeta_M,\Phi_M) - C\delta + \frac{1}{M}\log\gamma_M(\Ab_{M,\delta}).
}
\end{prop}

The key step toward proving Proposition \ref{find_pair} is the following consequence of Theorem \ref{GG_quantitative_thm}: the so-called multi-species Ghirlanda--Guerra identities as put forth in \cite{panchenko15}.
Since we have \eqref{GG_together}, the proof of Lemma \ref{GG_projection_prop} is identical to that of \cite[Thm.~3]{panchenko15}.

%For simplicity, expectation with respect to the law of $\vc\RR_M$ will be denoted by $\E(\cdot)$.
%And mimicking the notation from before, we denote a finite sub-array of $\vc\RR_M$ as $\vc\RR^n_M = (\vc \RR_{\ell,\ell'})_{\ell,\ell'\in[n]}$.
%By the discussion from above, the right-hand side of \eqref{restrict_attention} also converges, but we wish to have a description of the limit still in terms of Gibbs measures.
%Let $\vc\RR_M$ be an infinite array with the law of $\vc\LL_M$.

\begin{lemma} \label{GG_projection_prop}
Let $\vc\RR$ be a random vector array with law $\vc\LL_M$.
Given any bounded measurable function $\vphi\colon[-1,1]^\SSS\to\R$, define $Q_{\ell,\ell'} = \vphi(\vc \RR_{\ell,\ell'})$.
For any bounded measurable function $f$ of the finite sub-array $\vc\RR^n = (\vc\RR_{\ell,\ell'})_{\ell,\ell'\in[n]}$, we have
\eeq{ \label{GG_projection}
\E[f(\vc\RR^n)Q_{1,n+1}]
= \frac{1}{n}\E[f(\vc\RR^n)]\cdot\E[Q_{1,2}]
+ \frac{1}{n}\sum_{\ell=2}^n \E[f(\vc\RR^n)Q_{1,\ell}].
}
\end{lemma}

Given any realization of the vector array $\vc\RR = (\RR_{\ell,\ell'}^s)_{\ell,\ell'\geq1,s\in\SSS}$, define a scalar array $\RR = (\RR_{\ell,\ell'})_{\ell,\ell'\geq1}$ by averaging the across all species:
\eeq{ \label{average_overlap_def}
\RR_{\ell,\ell'} \coloneqq %\RR_{\ell,\ell'}^{\vc 1} = 
\sum_{s\in\SSS} \lambda^s\RR_{\ell,\ell'}^s.
}
Let us first check the basic fact that all relevant scalar arrays are Gram de-Finetti arrays (i.e.~symmetric, nonnegative definite, and having entries that are exchangeable under finite permutations).

\begin{lemma} \label{check_gram}
If $\vc\RR$ has the law $\vc\LL_M$, then $\RR^s = (\RR_{\ell,\ell'}^s)_{\ell,\ell'\geq1}$ and $\RR$ are Gram--de Finetti arrays such that $\RR_{\ell,\ell}^s=\RR_{\ell,\ell}=1$ for every $\ell\geq1$.
\end{lemma}

\begin{proof}
Recall that $\vc\LL_M$ is the large-$k$ weak limit of $\vc\LL_{M,N_k}$, where $\vc\LL_{M,N}$ is the law of the overlap array generated by i.i.d.~samples from the Gibbs measure $\bar G_{M,N}$.
By Skorokhod's representation theorem, there is coupling of $\vc\RR_k\sim\vc\LL_{M,N_k}$ and $\vc\RR\sim\vc\LL_M$ such that $\vc\RR_k\to\vc\RR$ almost surely as $k\to\infty$.
That is, almost surely every entry of $\vc \RR_k$ converges to the corresponding entry of $\vc\RR$.
Therefore, if we write $\vc\RR_k = (\RR_k^s)_{s\in\SSS}$, where $\RR_k^s = (\RR_{\ell,\ell'}^s)_{\ell,\ell'\geq1}$, then it suffices to show that the desired statements hold for $\RR_k^s$, as well as $\RR_k$ defined as in \eqref{average_overlap_def}.

So let us fix $k$ and recall that the entries of $\RR_k^s$ are given by
\eq{
\RR_{\ell,\ell'}^s = \frac{1}{N^s}\sum_{i\in\II^s}\sigma^\ell_i\sigma^{\ell'}_i,
}
where $N = N_k$ and $(\sigma^\ell)_{\ell\geq1}$ are i.i.d.~samples from $\bar G_{M,N}$ that are the same across all $s\in\SSS$.
It is immediately clear that $\RR_{\ell,\ell}^s=1$ (since $\sigma\in\T_N=\Motimes_{s\in\SSS} S_{N^s}$) and that symmetry holds: $\RR_{\ell,\ell'}^s = \RR_{\ell',\ell}^s$.
These two facts extend of course to the array $\RR_k$, which is just a convex combination of the $\RR_k^s$.
Furthermore, the fact that $(\sigma^\ell)_{\ell\geq1}$ are i.i.d.~(conditional on $\bar G_{M,N}$) implies that the entries of $\RR^s_k$ are exchangeable.
Again, this fact trivially extends to $\RR_k$. %which is the convex combination \eqref{average_overlap_def} of the $\RR_k^s$.
Finally, we check nonnegative definiteness directly:
For any $n\geq1$ and any vector $(x_\ell)_{\ell\in[n]}$, we have
\eq{
\sum_{\ell,\ell'\in[n]} x_\ell x_{\ell'}\RR^s_{\ell,\ell'}
=\frac{1}{N^s}\sum_{i\in\II^s}\sum_{\ell,\ell'\in[n]} x_\ell\sigma^{\ell}_ix_{\ell'}\sigma^{\ell'}_i
    = \frac{1}{N^s}\sum_{i\in\II^s}\bigg(\sum_{\ell\in[n]} x_\ell\sigma_i^\ell\bigg)^2
    \geq0.
}
Indeed, $\RR^s_k$ is nonnegative definite.
Since this property is closed under linear combination with nonnegative coefficients, the array $\RR_k$ is also nonnegative definite. 
%We have thus shown that $\RR_k^s$ and $\RR_k$ are a Gram de-Finetti arrays.
\end{proof}

The purpose of Lemmas \ref{GG_projection_prop} and \ref{check_gram} is to relate $\vc\RR$ and $\RR$ via synchronization.  That is, we invoke Theorem \ref{sync_thm}, which is recalled here for convenience.

\setcounter{theirthm}{2}
\begin{theirthm}\textup{\cite[Thm.~4]{panchenko15}}
\label{sync_thm_2}
%Let $\vc\RR = (\vc \RR_{\ell,\ell'})_{\ell,\ell'\geq1}$ be any array of $\SSS$-tuples such that for each $s\in\SSS$, the array of scalars $(\RR_{\ell,\ell'}^s)_{\ell,\ell'\geq1}$ is a Gram--de Finetti array.
%Let $\RR = (\RR_{\ell,\ell'})_{\ell,\ell'\geq1}$ be the convex combination of these scalar arrays as in \eqref{average_overlap_def}.
If $\vc\RR$ satisfies the multi-species G.G. identities \eqref{GG_projection}, then there exist non-decreasing $(1/\lambda^s)$-Lipschitz functions $\Phi^s\colon[0,1]\to[0,1]$ such that almost surely,
\eeq{ \label{sync_eq_2}
\RR_{\ell,\ell'}^s = \Phi^s(\RR_{\ell,\ell'}) \quad \text{for all $\ell,\ell'\geq1$, $s\in\SSS,$}
}
where $\RR_{\ell,\ell'}$ is defined in \eqref{average_overlap_def}.
\end{theirthm}

Recall from Remark \ref{gg_remark} that if
 $\vc\RR$ satisfies the multi-species Ghirlanda--Guerra identities \eqref{GG_projection}, then the scalar array $\RR = (\RR_{\ell,\ell'})_{\ell,\ell'\geq1}$ automatically satisfies the ordinary G.G. identities and thus has nonnegative entries almost surely.
Let us make another important remark about Theorem \ref{sync_thm_2}.

%\begin{remark} \label{gg_remark}
%As usual, for any positive integer $n$, we will write $\RR_M^n = (\R\RR_{\ell,\ell'})_{\ell,\ell'\in[n]}$.
%If $\vc\RR$ satisfies the conclusion of Lemma \ref{GG_projection_prop}, then $\RR$ satisfies the classical G.G. identities.
%That is, for any bounded measurable function $f$ of the finite sub-array $\RR^n = (\RR_{\ell,\ell'})_{\ell,\ell'\in[n]}$ and any bounded measurable $\psi\colon[-1,1]\to\R$, we have
%\eeq{ \label{GG_for_RR}
%\E[ f \psi(\RR_{1,n+1})]
%= \frac{1}{n}\E[f]\cdot\E[ \psi(\RR_{1,2})] + \frac{1}{n}\sum_{\ell=2}^n \E[ f\cdot\psi(\RR_{1,\ell})],
%}
%where $\langle\cdot\rangle$ denotes expectation with respect to the measure $\GG$ from Theorem \ref{ds_rep}.
%To verify \eqref{GG_for_RR}, simply set $\phi(\vc x) = \sum_{s\in\SSS}\lambda^s x^s$, and take $\vphi = \psi\circ\phi$ in Lemma \ref{GG_projection_prop}.
%Once the G.G. identities are known to hold for $\RR$, Talagrand's positivity principle \cite[Thm.~2.16]{panchenko13a} guarantees that $\RR_{\ell,\ell'}\geq0$ with probability one.
%Therefore, the domain of $\Phi^s$ makes sense.
%\end{remark}

\begin{remark} \label{assume_admissible_remark}
Given the array $\RR$ from \eqref{average_overlap_def},
consider the probability measure $\zeta$ on $[0,1]$ defined by
\eeq{ \label{zeta_from_R12}
\zeta(\cdot) = \E\langle \one_{\{\RR_{1,2}\in \cdot\}}\rangle.
}
It follows from \eqref{sync_eq_2} that
\eq{
\RR_{1,2} = \sum_{s\in\SSS}\lambda^s \RR_{1,2}^s = \sum_{s\in\SSS}\lambda^s \Phi^s(\RR_{1,2}) \quad \mathrm{a.s.}
}
This equality implies that for every $q$ belonging to the support of $\zeta$, we have
\eeq{ \label{for_admissibility}
q = \sum_{s\in\SSS}\lambda^s \Phi^s(q).
}
If necessary, we can use linear interpolation to redefine each $\Phi^s$ outside the support of $\zeta$ (with $\Phi^s(0)=0$ and $\Phi^s(1)=1$) so that \eqref{for_admissibility} holds for all $q\in[0,1]$.
In this way, we may assume that the map $\Phi = (\Phi^s)_{s\in\SSS}$ in Theorem \ref{sync_thm_2} is $\vc\lambda$-admissible.
\end{remark}

%We will denote by $\Phi$ the map $[0,1]\to[0,1]^\SSS$ obtained by concatenating the results of applying each $\Phi^s$:
%\eeq{ \label{L_def}
%\RR_{\ell,\ell'}\mapsto \Phi(\RR_{\ell,\ell'})
%= \big(\Phi^s(\RR_{\ell,\ell'})\big)_{s\in\SSS}.
%}

%In order to prove Proposition \ref{find_pair}, we will need one more fundamental result.

%\begin{theirthm} \label{representation_thm}
%\textup{\cite[Thm.~2.13 and 2.17]{panchenko13a}}
%Let $\RR$ and $\GG$ be as in Theorem \ref{ds_rep}.
%If $\RR$ satisfies the G.G. identities \eqref{GG_for_RR}, then $\mathsf{Law}(\RR;\GG)$ depends only on
%the probability measure $\zeta$ on $[0,1]$ defined by
%\eeq{ \label{zeta_from_R12}
%\zeta(\cdot) = \E\langle \one_{\{\RR_{1,2}\in \cdot\}}\rangle.
%}
%Moreover, under these circumstances, the map $\zeta\mapsto\mathsf{Law}(\RR;\GG)$ is continuous with respect to weak convergence.
%\end{theirthm}

We are now ready to prove Proposition \ref{find_pair}.

\begin{proof}[Proof of Proposition \ref{find_pair}]
Let $\LL_M$ be the pushforward of $\vc\LL_M$ under the map $\vc\RR\mapsto\RR$ defined in \eqref{average_overlap_def}.
By Lemmas \ref{GG_projection_prop} and \ref{check_gram}, we can apply Theorem \ref{sync_thm_2}, which says there is a map $\Phi = ( \Phi^s)_{s\in\SSS}\colon[0,1]\to[0,1]^\SSS$ such that 
$\vc\LL_M = \LL_M \circ \Phi^{-1}$.
By Remark \ref{assume_admissible_remark}, we may assume $\Phi$ is $\vc\lambda$-admissible.
By Lemma \ref{check_gram}, we can also apply Theorem \ref{ds_rep} to identify a random measure $\GG$ on the unit ball of some separable Hilbert space, such that $\LL_M = \mathsf{Law}(\RR;\GG)$.
Let $\zeta$ be defined by \eqref{zeta_from_R12}.

By Remark \ref{gg_remark}, the law $\LL_M$ satisfies the G.G. identities \eqref{GG_for_RR_with_Gibbs}.
Now take any sequence of finitely supported measures $(\zeta_k)_{k\geq1}$ converging weakly to $\zeta$.
Let $\GG_k=\GG_{\zeta_k}$ be the Ruelle probability cascade \eqref{rpc_def} associated to $\zeta_k$.
By \cite[Thm.~15.2.1]{talagrand11b}, $\mathsf{Law}(\RR;\GG_{k})$ also satisfies the G.G. identities.
It thus follows from \eqref{overlap_for_rpc} and Theorem \hyperref[representation_thm_c]{\ref*{representation_thm}\ref*{representation_thm_c}} that 
$\mathsf{Law}(\RR;\GG_{k})$ converges to $\mathsf{Law}(\RR;\GG)=\LL_M$ as $k\to\infty$.
Since $\Phi$ is continuous (for instance, see \eqref{lambda_av_consequence}), it must then be the case that the law $\vc\LL(\zeta_k,\Phi)=\mathsf{Law}(\RR;\GG_k)\circ\Phi^{-1}$ from \eqref{equivalent_gibbs} converges to $\LL_M\circ\Phi^{-1}=\vc\LL_M$.
Hence
\eq{
\PPP_M(\zeta,\Phi)
&\stackref{extension_def}{=}\lim_{k\to\infty}\PPP_M(\zeta_k,\Phi) \stackref{restriction_def}{=}\lim_{k\to\infty}\Pi_M(\vc\LL(\zeta_k,\Phi))
\stackrel{\mbox{\footnotesize\text{(Cor.~\ref{extension_cor})}}}{=} \Pi_M(\vc\LL_M).
}
In light of \eqref{restrict_attention_2}, the proof is complete with $(\zeta_M,\Phi_M) = (\zeta,\Phi)$.
\end{proof}

\subsection{Conclusion of proofs for main results} \label{main_proofs}

We can now complete the proof of Theorem \ref{main_thm} by establishing the lower bound \eqref{parisi_lower}.

\begin{proof}[Proof of Theorem \ref{main_thm}] %and Corollary \ref{main_cor}]
By using the concentration inequality from Lemma \ref{concentration_lemma} (with $M=0$ and every $u_{p,q}=0$) together with Borel--Cantelli, we see that
\eq{
\lim_{N\to\infty} |F_N - \E F_N| = 0 \quad \mathrm{a.s.}
}
Therefore, to show \eqref{parisi_formula} it suffices to prove
\eq{
\limsup_{N\to\infty} \E F_N \leq \inf_{\zeta,\Phi}\PPP(\zeta,\Phi) \leq
\liminf_{N\to\infty} \E F_N. 
}
By Proposition \ref{prop:upper:bound} we already have the first inequality, and so it suffices to exhibit a $\vc\lambda$-admissible pair $(\zeta,\Phi)$ such that
\eeq{ \label{final_need_lower}
\liminf_{N\to\infty} \E F_N \geq \PPP(\zeta,\Phi).
}
To this end, let $(\zeta_M,\Phi_M)$ be the $\vc\lambda$-admissible pair from Proposition \ref{find_pair}.
By the Central Limit Theorem, for any fixed $\delta>0$, the quantity $\gamma_M(\Ab_{M,\delta})$ tends to $1/2$ as $M\to\infty$.
Therefore, the inequality \eqref{find_pair_eq} leads to
\eeq{ \label{find_pair_conseqence}
\liminf_{N\to\infty}\E F_N \geq \limsup_{M\to\infty}\frac{1}{M}\PPP_M(\zeta_M,\Phi_M).
}

Recall from \eqref{lambda_av_consequence} that any $\vc\lambda$-admissible map $\Phi:[0,1]\to[0,1]^\SSS$ is Lipschitz continuous with a Lipschitz constant not depending on $\Phi$.
Therefore, by the Arzel\`a–Ascoli theorem  \cite[Thm.~47.1]{munkres00}, there exists a sequence $(M_k)_{k\geq1}$ tending to infinity such that $\Phi_{M_k}$ converges uniformly to some function $\Phi$, which is necessarily $\vc\lambda$-admissible.
Since the space of probability measures on $[0,1]$ is compact, we may assume that $\zeta_{M_k}$ also converges weakly to some $\zeta$.
It is then clear that $\zeta_{M_k}\circ \Phi_{M_k}^{-1}$ converges weakly to $\zeta\circ \Phi^{-1}$, hence $\DD\big((\zeta_{M_k},\Phi_{M_k}),(\zeta,\Phi)\big)\to0$ as $k\to\infty$.

We now complete the proof by appealing to the results from Section \ref{extension_section}.
By the triangle inequality and Corollary \ref{lipschitz_continuity_prelimit}, we have
\eq{
&\Big|\frac{1}{M}\PPP_{M}(\zeta_{M},\Phi_{M})-\PPP(\zeta,\Phi)\Big| \\
&\leq \Big|\frac{1}{M}\PPP_{M}(\zeta_{M},\Phi_{M})-\frac{1}{M}\PPP_{M}(\zeta,\Phi)\Big|
+ \Big|\frac{1}{M}\PPP_{M}(\zeta,\Phi)-\PPP(\zeta,\Phi)\Big| \\
&\leq \frac{C_*}{2}\bigg(1+\sum_{s\in\SSS}\Big|\frac{M^s}{M}-\lambda^s\Big|\bigg)\DD\big((\zeta_M,\Phi_M),(\zeta,{\Phi})\big)
+ \Big|\frac{1}{M}\PPP_{M}(\zeta,\Phi)-\PPP(\zeta,\Phi)\Big|.
}
The first term in the last line tends to $0$ when 
$M$ is brought to infinity along the sequence $(M_k)_{k\geq1}$.
By Proposition \ref{appearance_explained} (which is enabled by \eqref{cavity_fraction_converges}), the second term also tends to $0$.
In combination with \eqref{find_pair_conseqence}, these observations yield \eqref{final_need_lower}.
%which must then be an equality by Proposition \ref{prop:upper:bound}.
\end{proof}

\appendix

\section{General facts about perturbed Gibbs measures} \label{appendix}
In order for the results of this appendix to be widely applicable, we consider a general setting.
Let $(\Sigma,\FF,\tau)$ be a finite measure space.
Take $H\colon\Sigma\to\R$ to be any $\FF$-measurable function (possibly random) satisfying
\eeq{ \label{sigma_tau_H}
\int_{\Sigma}\E\exp|H(\sigma)|\ \tau(\dd\sigma) < \infty.
}
Let $(h_i)_{i\geq1}$ be independent Gaussian processes on $\Sigma$, which are also independent of $H$.
We assume that $h_i(\cdot)$ is almost surely $\FF$-measurable.
We also assume that for each $i$, there is a constant $r_i$ such that 
\eeq{ \label{rpN_assumption}
\E[h_i(\sigma)^2] =r_iN \quad \text{for all $\sigma\in\Sigma$}.
}
(Here $N<\infty$ is merely a parameter and need not be an integer.)
More generally, we define
\eeq{ \label{frak_R_def}
\mathfrak{R}^i(\sigma,\sigma') \coloneqq \frac{1}{N}\E[h_i(\sigma)h_i(\sigma')], \quad \sigma,\sigma'\in\Sigma.
}
In particular, we have $\mathfrak{R}^i(\sigma,\sigma) = r_i$.
%We also assume that these constants are summable:
%\eeq{ \label{rN_assumption}
%\sum_{i=1}^\infty r_i < \infty.
%}
%Finally, let $c>0$ be a parameter.

Given a parameter $c\geq0$ and any sequence $u=(u_i)_{i\geq1}$ of real numbers, define the Hamiltonian
\eeq{ \label{Hu_def}
H_u(\sigma) \coloneqq H(\sigma)+ch_u(\sigma), \quad \text{where} \quad h_u(\sigma)\coloneqq\sum_{i=1}^\infty u_ih_i(\sigma).
}
Whenever the following quantity is finite,
\eeq{ \label{varsigma_assumption}
\varsigma^2(u)\coloneqq\E[h_u(\sigma)^2]=\sum_{i=1}^\infty u_i^2r_i,
}
we can consider the associated Gibbs measure:
\eq{
G_u(\dd\sigma)&\coloneqq \frac{1}{\exp\vphi(u)}\exp H_u(\sigma)\ \tau(\dd\sigma), \quad \text{where} \\
\vphi(u) &\coloneqq \log \int_{\Sigma}\exp H_u(\sigma)\ \tau(\dd\sigma). %\label{free_energy_general_def}
}
%Denote expectation according to $G$ by $\langle\cdot\rangle$.
Let us write $\vphi(0)$ when we wish to set all $u_i$ equal to $0$. 
%and analogously $\vphi(1)$ for all $u_i$ equal to $1$.
%In the latter scenario, we will simply write
%\eq{
%h(\sigma) \coloneqq \sum_{i=1}^\infty h_i(\sigma).
%}
This number can be compared to $\vphi(u)$ as follows.

\begin{lemma} \label{general_energy_lemma}
If $\varsigma^2(u)<\infty$, then
\eeq{ \label{general_perturbation_control}
\E\vphi(0)\leq\E\vphi(u)\leq\E\vphi(0)+\frac{c^2\varsigma^2(u) N}{2}.
}
\end{lemma}

\begin{proof}
We begin by writing
\eeq{ \label{starting_difference}
\vphi(u) - \vphi(0) = \log\frac{\int_{\Sigma}\exp(H(\sigma))\exp(ch_u(\sigma))\ \tau(\dd\sigma)}{\int_{\Sigma}\exp(H(\sigma))\ \tau(\dd\sigma)}.
}
Notice that the right-hand side is simply the average of $\exp(ch_u(\sigma))$ with respect to the Gibbs measure associated to $H$.
Therefore, by applying Jensen's inequality to the map $x\mapsto \exp(x)$, we obtain
\eq{
\vphi(u) - \vphi(0) \geq \frac{\int_{\Sigma}c h_u(\sigma)\exp(H(\sigma))\ \tau(\dd\sigma)}{\int_{\Sigma}\exp(H(\sigma))\ \tau(\dd\sigma)}.
}
As $h_u$ is independent of $H$, the expectation of the right-hand side can be obtained by first taking expectation of just $h_u(\sigma)$ in the numerator.
Since $\E(h_u(\sigma)) = 0$, we conclude that $\E\vphi(u) \geq \E\vphi(0)$.

For the second inequality, we again start with \eqref{starting_difference} and apply Jensen's inequality, but in this case to the function $x \mapsto \log x$:
\eq{ %\label{leads_to_second}
\E\vphi(u) - \E\vphi(0) \leq \log \E\Bigg[\frac{\int_{\Sigma}\exp(H(\sigma))\exp(ch_u(\sigma))\ \tau(\dd\sigma)}{\int_{\Sigma}\exp(H(\sigma))\ \tau(\dd\sigma)}\Bigg].
}
As before, the expectation on the right-hand side can first be taken just over $h_u$.
From \eqref{varsigma_assumption} we have $\E \exp(ch_u(\sigma)) = \exp(c^2\varsigma^2(u) N/2)$, and thus
we obtain the second inequality in \eqref{general_perturbation_control}.
\end{proof}

Next we state a concentration inequality together with the resulting moment bound.

\begin{lemma} \label{general_concentration_lemma}
If $H$ is non-random and $\varsigma^2(u)<\infty$, then
\eeq{ \label{general_concentration_ineq}
\P(|\vphi(u) - \E\vphi(u)|\geq t\sqrt{N})
\leq 2\exp\Big(-\frac{t^2}{4c^2\varsigma^2(u)}\Big).
}
In particular,
\eeq{ \label{general_expectation_outcome}
\E|\vphi(u)-\E\vphi(u)| \leq 2\sqrt{\pi c^2\varsigma^2(u)N}.
}
\end{lemma}

\begin{proof}
The inequality \eqref{general_concentration_ineq} is a consequence of concentration for Lipschitz functions of Gaussian random variables.
For instance, see the proof of \cite[Lem.~3]{panchenko07}.
The moment estimate \eqref{general_expectation_outcome} is realized by integrating the tail in \eqref{general_concentration_ineq}.
\end{proof}

Finally we discuss the Ghirlanda--Guerra identities.
%Now consider any $i$ for which $r_i>0$.
Let $\sigma^1,\sigma^2,\dots$ be independent samples from $G_u$, and define an array $(\mathfrak{R}_{\ell,\ell'}^i)_{\ell,\ell'\geq1}$ using the function $\mathfrak{R}^i$ from \eqref{frak_R_def}:
\eq{
\mathfrak{R}_{\ell,\ell'}^i \coloneqq \mathfrak{R}^i(\sigma^\ell,\sigma^{\ell'}).
}
With $\langle\cdot\rangle_u$ denoting expectation according to $G_u$, and $f = f(\sigma^1,\dots,\sigma^n)$ some non-random measurable function $\Sigma^n\to\R$, we define
\eq{
\Delta(f,n,i,u) \coloneqq \Big|\E\langle f\mathfrak{R}_{1,n+1}^i\rangle_u - \frac{1}{n}\E\langle f\rangle_u\E\langle\mathfrak{R}_{1,2}^i\rangle_u-\frac{1}{n}\sum_{\ell=2}^n\E\langle f\mathfrak{R}_{1,\ell}^i\rangle_u\Big|.
}
By averaging over just $u_i$, we can obtain a useful upper bound on $\Delta(f,n,i,u)$.

\begin{thm} \label{general_GG_thm}
Assume $\varsigma^2(u)<\infty$ whenever $u_{i}\in[0,3]$ for every $i$, and define
\eeq{ \label{vartheta_def}
\vartheta \coloneqq \sup\{\E|\vphi(u)-\E\vphi(u)|:\, u_{i}\in[0,3]\text{ for all $i$}\}.
}
For any $i$, any $u_{i'}\in[0,3]$ for $i'\neq i$, any $f = f(\sigma^1,\dots,\sigma^n)$ as above, and any $N$ such that $(2c)^{-1}\sqrt{\frac{\vartheta }{r_i N}}<1$, we have
\eeq{ \label{general_GG_claim}
\int_1^2\Delta(f,n,i,u)\ \dd u_i &\leq 2\|f\|_\infty(cn)^{-1}\Big(2\sqrt{\frac{r_i}{N}} + 12\sqrt{\frac{r_i\vartheta }{N}}\Big) \\
&\leq 24\|f\|_\infty\sqrt{r_i}(cn\sqrt{N})^{-1}(1+\sqrt{\vartheta}).
}
\end{thm}

\begin{proof}
Our proof is a direct adaptation of \cite[Thm.~3.2]{panchenko13a}.
We will use the notation $\nu(\cdot) = \E\langle\cdot\rangle_u$ and simply write $\langle\cdot\rangle$ for $\langle\cdot\rangle_u$.
Fix the value of $i$.
Our access point to the quantity $\Delta(f,n,i,u)$ is through the difference
\eeq{ \label{access_point}
|\nu(f h_i(\sigma^1)) - \nu(f)\nu(h_i)|
&= \Big|\E\Big\langle (f - \nu(f))(h_i(\sigma^1)-\nu(h_i))\Big\rangle\Big| \\
&\leq 2\|f\|_\infty\E\big\langle|h_i-\nu(h_i)|\big\rangle.
}
Recalling \eqref{frak_R_def} and applying Gaussian integration by parts (see \cite[Exercise 1.1]{panchenko13a}), we have
\eeq{ \label{upon_recalling_hi}
\nu(fh_i(\sigma^1)) 
%&\stackrefp{frak_R_def}{=} \sum_{j=1}^\infty \E\big[g_{i,j}\langle f h_{i,j}(\sigma^1)\rangle\big] \\
%&\stackrefp{frak_R_def}{=} \sum_{i=1}^\infty \E\Big[\frac{\partial}{\partial g_{i,j}}\langle fh_{i,j}(\sigma^1)\rangle\Big] \\
%&\stackrefp{frak_R_def}{=} cu_i\sum_{j=1}^\infty \E\Big[\Big\langle fh_{i,j}(\sigma^1)\Big(\sum_{\ell=1}^nh_{i,j}(\sigma^\ell)-nh_{i,j}(\sigma^{n+1})\Big)\Big\rangle\Big] \\
= 
cu_iN\Big[\E\langle f\mathfrak{R}_{1,1}^i\rangle+\sum_{\ell=2}^n\E\langle f\mathfrak{R}_{1,\ell}^i\rangle - n\E\langle f\mathfrak{R}_{1,n+1}^i\rangle\Big].
}
The special case of ($n=1$, $f\equiv 1$) yields
\eeq{ \label{nu_h}
\nu(h_i) = cu_iN\big[\nu(\mathfrak{R}_{1,1}^i) - \nu(\mathfrak{R}_{1,2}^i)\big].
}
From the definition \eqref{frak_R_def} of $\mathfrak{R}^i(\cdot,\cdot)$, it is clear that
\eq{
%\langle \mathfrak{R}_{1,2}^i\rangle = N^{-1}\sum_{j=1}^\infty\langle h_{i,j}\rangle^2\geq0 \quad \text{and} \quad
\mathfrak{R}_{1,2}^i \leq \sqrt{\mathfrak{R}_{1,1}^i\mathfrak{R}_{2,2}^i} = r_i,
}
and so it follows from \eqref{nu_h} that
\eeq{ \label{nu_h_inequality}
0 \leq \nu(h_i) \leq 2cu_ir_iN.
}
Now we combine \eqref{upon_recalling_hi} and \eqref{nu_h} to obtain an expression for the difference $\nu(fh_i(\sigma^1))-\nu(f)\nu(h_i)$.
Since $\mathfrak{R}_{1,1}^i=r_i$ for any realization of $\sigma^1$,
the terms involving $\mathfrak{R}_{1,1}^i$ cancel each other, leaving us with
\eeq{ \label{h_motivation_1}
|\nu(fh_i(\sigma^1))-\nu(f)\nu(h_i)| = cu_iNn\cdot \Delta(f,n,i,u).
}
%Given that $f$ is bounded, we may assume without loss of generality that $|f|\leq1$.
%Under this assumption, it is clear from \eqref{access_point} that
%\eq{ %\label{h_motivation_2}
%|\nu(f h(\sigma^1)) - \nu(f)\nu(h)|
%&\leq \E\big\langle| h - \nu(h)|\big\rangle.
%}
The right-hand side of \eqref{h_motivation_1} can now replace the leftmost expression in \eqref{access_point}.
To then conclude \eqref{general_GG_claim}, it suffices to control the expectation in the final expression of \eqref{access_point}.
Indeed, we claim that
\eeq{ \label{claim_integral}
\int_1^2 \E\big\langle| h_i - \nu(h_i)|\big\rangle\ \dd u_i \leq
2\sqrt{r_iN} + 12\sqrt{\vartheta r_i N}.
}
Once this is proved, we will have established the desired statement \eqref{general_GG_claim}.
%\eq{
%\int_1^2 \Delta_{M,N}(f,p,q)\ \dd u_{p,q}
%= O_{M,p,q}(N^{-1/4})c_{N+M}^{-1}.
%}

The rest of the proof is to establish \eqref{claim_integral}.
Define
\eq{
\phi(u) \coloneqq \E\vphi(u) = \E\log \int_{\Sigma}\exp H_u(\sigma)\ \tau(\dd\sigma).
}
Fixing the value of $u_{i'}\in[0,3]$ for every $i'\neq i$,
let us regard $\vphi$ and $\phi$ as functions of only $u_i\in[1,2]$.
Direct calculation yields the standard identities
\begin{subequations}
\label{derivative_identities}
\eeq{ 
\vphi'(u_i)= c\langle h_i\rangle, \qquad
\vphi''(u_i) = c^2\big\langle (h_i-\langle h_i\rangle)^2\big\rangle.
}
Moreover, Gaussian tails provide sufficient regularity to exchange differentiation and expectation in order to write
\eeq{ \label{derivative_identities_2}
\phi'(u_i)= c\nu(h_i), \qquad
\phi''(u_i) = c^2\E\big\langle (h_i-\langle h_i\rangle)^2\big\rangle.
}
\end{subequations}
In particular, both $\vphi$ and $\phi$ are convex in $u_i$, and integrating $\phi''$ gives
\eq{ %\label{ftoc}
 c^2\int_1^2 \E\big\langle (h_i-\langle h_i\rangle)^2\big\rangle\ \dd u_i
 = \phi'(2)-\phi'(1) \stackref{derivative_identities_2,nu_h_inequality}{\leq} 4c^2r_iN.
}
%It is then clear that $\phi'(2)-\phi'(1)$ must be nonnegative.
Canceling factors of $c^2$ and applying Jensen's inequality, we arrive at
\eeq{ \label{first_averaging}
\int_1^2 \E\big\langle|h_i-\langle h_i\rangle|\big\rangle\ \dd u_i \leq 2\sqrt{r_iN}.
}
To bootstrap this inequality to \eqref{claim_integral}, we next need to compare $\langle h_i\rangle$ and $\nu(h_i)$.
%In fact, it is the difference between these two terms that contributes the leading order term, not \eqref{first_averaging}.

By appealing to \cite[Lem.~3.2]{panchenko13a} and then taking expectation, we obtain the following for any $y\in(0,1)$:
\eeq{ \label{from_convexity_lemma}
\E|\vphi'(u_i)-\phi'(u_i)| 
&\leq \phi'(u_i+y) - \phi'(u_i-y)
+ \frac{\E|\vphi(u_i+y)-\phi(u_i+y)|}{y} \\ &\phantom{\leq}+\frac{\E|\vphi(u_i-y)-\phi(u_i-y)|}{y} + \frac{\E|\vphi(u_i)-\phi(u_i)|}{y}.
}
Upon integration, the first two terms on the right-hand side become 
\eq{
\int_1^2 [\phi'(u_i+y) - \phi'(u_i-y)]\ \dd u_i
&= [\phi(2+y) - \phi(2-y)] - [\phi(1+y)-\phi(1-y)] \\
&\leq 2y\sup_{x\in(0,3)}\phi'(x) \stackref{derivative_identities_2,nu_h_inequality}{\leq} 12yc^2r_iN.
}
By definition \eqref{vartheta_def}, the remaining three terms on the right-hand side of \eqref{from_convexity_lemma} are all bounded by $\vartheta /y$, which leads to
\eq{
\int_1^2 \E|\vphi'(u_i)-\phi'(u_i)|\ \dd u_i
&\leq 12yc^2r_iN+3\vartheta /y.
}
Recalling \eqref{derivative_identities}, we can rewrite this inequality as
\eq{
\int_1^2 \E|\langle h_i\rangle - \nu(h_i)|\ \dd u_i
\leq 12ycr_iN + \frac{3\vartheta }{yc}.
%&\leq 6N^{3/4}\Big[\Big(y\frac{c_{N}N^{1/4}}{2^{p+q}}\Big)+\Big(y\frac{c_{N}N^{1/4}}{2^{p+q}}\Big)^{-1}\sqrt{\pi(\xi_{N}(\vc 1)+\xi_N^\pert(\vc 1))}\, \Big].
}
Finally, we choose $y=(2c)^{-1}\sqrt{\frac{\vartheta }{r_iN}}$, where $N$ is assumed to be sufficiently large that $y<1$.
This choice results in
\eq{
\int_1^2 \E|\langle h_i\rangle - \nu(h_i)|\ \dd u_i
&\leq 12\sqrt{\vartheta r_i N}.
}
Combining this inequality with \eqref{first_averaging} yields \eqref{claim_integral}, as claimed.
\end{proof}

%\section{A general continuity result}
%Let $\Sigma$ and $\RRR$ be metric spaces, and suppose that $\vc R:\Sigma\times\Sigma\to\RRR$ is a symmetric function. Assume there are independent, centered Gaussian processes $X_1,\dots,X_M$ on $\Sigma$ whose covariances are of the form
%\eq{
%\E[X_j(\sigma)X_j(\sigma')] = \CC_j(\vc R(\sigma,\sigma')),
%}
%for some measurable functions $\CC_j:\RRR\to\R$.
%Also assume there is a constant $C$ such that
%\eq{
%\CC_j(\vc R(\sigma,\sigma)) \leq C \quad \text{for all $j\in[M]$, $\sigma\in\Sigma$}.
%}
%
%\begin{prop}
%Let $f:\R^M\to\R$ be continuous.
%Let $\langle\cdot\rangle$ denote expectation with respect to a random probability measure on $\Sigma$ which is independent of $X_1,\dots,X_M$.
%Let $\eta_1,\dots,\eta_M$ be standard normal random variables independent of everything else, and set
%\eq{
%\psi = \E_\eta\langle f(\wt X_1(\sigma),\dots,\wt X_M(\sigma))\rangle, \quad \text{where} \quad \wt X_j(\sigma) \coloneqq X_j(\sigma) + \eta_j\sqrt{C-\CC_j(\vc R(\sigma,\sigma))}.
%}
%
%\end{prop}

\section*{Acknowledgments}
We are grateful to Amir Dembo for valuable feedback and suggestions, and to Pax Kivimae for the detection of a computational error in a previous draft.
We thank the referee for several corrections resulting from their careful reading.

%: BIBLIOGRAPHY
%\nocite{*} Uncomment to include all entries in the bibliography.
\bibliography{erikbib}

\end{document}